\documentclass[11pt, twoside, leqno]{article}

\usepackage{amssymb}
\usepackage{amsmath}
\usepackage{amsthm}
\usepackage{mathrsfs}

\allowdisplaybreaks

\def\ls{\lesssim}

\def\fz{\infty}

\def\r{\right}
\def\lf{\left}

\def\aa{{\mathbb A}}
\def\rr{{\mathbb R}}
\def\rh{{\mathbb R}{\mathbb H}}
\def\rn{{{\rr}^n}}
\def\zz{{\mathbb Z}}
\def\nn{{\mathbb N}}

\newcommand{\wz}{\widetilde}
\newcommand{\oz}{\overline}

\newcommand{\cs}{{\mathcal S}}
\newcommand{\cx}{{\mathcal X}}

\def\az{\alpha}
\def\lz{\lambda}

\def\bz{\beta}
\def\rz{\rho}

\def\fai{\varphi}
\def\gz{{\gamma}}

\def\vz{\varphi}

\def\wz{\widetilde}

\def\ls{\lesssim}

\def\oz{\omega}

\def\esup{\mathop\mathrm{\,ess\,sup\,}}
\def\einf{\mathop\mathrm{\,ess\,inf\,}}

\def\bbmo{{{\mathop\mathrm{BMO}}}}

\def\loc{{\mathop\mathrm{loc\,}}}
\def\lfz{\lfloor}
\def\rfz{\rfloor}

\def\hs{\hspace{0.3cm}}

\newtheorem{theorem}{Theorem}[section]
\newtheorem{lemma}[theorem]{Lemma}
\newtheorem{corollary}[theorem]{Corollary}
\newtheorem{proposition}[theorem]{Proposition}

\theoremstyle{definition}

\newtheorem{remark}[theorem]{Remark}
\newtheorem{definition}[theorem]{Definition}

\theoremstyle{remark}

\numberwithin{equation}{section}

\begin{document}

\arraycolsep=1pt

\title{\bf\Large Musielak-Orlicz BMO-Type Spaces Associated with Generalized
Approximations to the Identity
\footnotetext{\hspace{-0.35cm} 2010 {\it
Mathematics Subject Classification}. Primary 42B35; Secondary 42B30, 46E30, 30L99.
\endgraf {\it Key words and phrases}. growth function,
BMO-type space, approximation to the identity, John-Nirenberg inequality,
Carleson measure.
\endgraf Dachun Yang is supported by the National
Natural Science Foundation  of China (Grant No. 11171027) and
the Specialized Research Fund for the Doctoral Program of Higher Education
of China (Grant No. 20120003110003).}}
\author{Shaoxiong Hou, Dachun Yang
\footnote{Corresponding author} \ and Sibei Yang}
\date{ }
\maketitle

\vspace{-0.5cm}

\begin{center}
\begin{minipage}{13cm}
{\small {\bf Abstract}\quad Let $\mathcal{X}$ be
a space of homogenous type and
$\varphi:\ \mathcal{X}\times[0,\infty)
\to[0,\infty)$ a growth function such that $\varphi(\cdot,t)$
is a Muckenhoupt weight uniformly in $t$
and $\varphi(x,\cdot)$ an Orlicz function
of uniformly upper type 1
and lower type $p\in(0,1]$. In this article, the authors
introduce a new Musielak-Orlicz BMO-type space
$\mathrm{BMO}^{\varphi}_A(\mathcal{X})$ associated with the generalized
approximation to the identity, give out its basic
properties and establish its
two equivalent characterizations, respectively, in terms of the spaces
$\mathrm{BMO}^{\varphi}_{A,\,\mathrm{max}}(\mathcal{X})$ and
$\widetilde{\mathrm{BMO}}^{\varphi}_A(\mathcal{X})$.
Moreover, two variants of the John-Nirenberg inequality on
$\mathrm{BMO}^{\varphi}_A(\mathcal{X})$ are
obtained. As an application, the authors
further prove that the space
$\mathrm{BMO}^{\varphi}_{\sqrt{\Delta}}(\mathbb{R}^n)$,
associated
with the Poisson semigroup of the Laplace operator $\Delta$ on $\mathbb{R}^n$,
coincides with the space
$\mathrm{BMO}^{\varphi}(\mathbb{R}^n)$ introduced by L. D. Ky.}
\end{minipage}
\end{center}

\section{Introduction}
\hskip\parindent
The classical $\mathrm{BMO}(\mathbb R^n)$ space
(the \emph{space of functions with
bounded mean oscillation}), originally introduced
by John and Nirenberg \cite{jn},
plays an important role in partial differential equations and
modern harmonic analysis
(see, for example, \cite{jn,fs}). Recall that a locally integrable
function $f$ on the
$n$-dimensional Euclidean space $\rn$ is said to be in the \emph{space}
$\mathrm{BMO}(\mathbb R^n)$, if
$$\|f\|_{\mathrm{BMO}(\mathbb R^n)}
:= \sup_{B\subset\rn}\frac{1}{|B|}\int_{B}|f(x)-f_B|\,dx<\infty,$$
where the supremum is taken over all balls $B\subset \mathbb R^n$ and
$f_B:=\frac{1}{|B|}\int_Bf(x)\,dx$.
It is well known that many operators such as Carder\'{o}n-Zygmund
singular integral operators are not bounded
on the Lebesgue space $L^\infty(\mathbb{R}^n)$, but they are bounded from
$L^\infty(\mathbb{R}^n)$ to $\mathrm{BMO}(\mathbb R^n)$
(see, for example, \cite{gra}).
Therefore, the space $\mathrm{BMO}(\mathbb R^n)$ is considered
as a natural substitute for
$L^\infty(\mathbb{R}^n)$ when studying the boundedness of
operators on function spaces.
Moreover, $\mathrm{BMO}(\mathbb R^n)$ plays a significant role in the
interpolation theory of linear operators. Precisely,
if a linear operator $T$ is bounded
from $L^q(\mathbb R^n)$ to $L^q(\mathbb R^n)$ for some $q\in [1,\fz)$ and
bounded from $L^\infty(\mathbb{R}^n)$ to $\mathrm{BMO}(\mathbb R^n)$,
then $T$ is also bounded from $L^p(\mathbb R^n)$ to $L^p(\mathbb R^n)$
for all $p\in[q,\infty)$ (see, for example, \cite{gra}). Furthermore,
Fefferman and Stein \cite{fs} proved that $\bbmo(\rn)$ is the dual space
of the Hardy space $H^1(\rn)$.

Recently, Ky \cite{ky} introduced Musielak-Orlicz BMO-type spaces
$\bbmo^{\fai}(\rn)$,
which generalize the classical space $\mathrm{BMO}(\mathbb{R}^n)$,
the weighted BMO
space $\bbmo_\oz(\rn)$ (see, for example, \cite{mw2, mw1, bu})
and the Orlicz BMO-type spaces
$\bbmo_\rz(\rn)$ (see, for example, \cite{s79, ja80, vi87}). Here,
$\varphi:\ \mathbb R^n\times[0,\infty)\to[0,\fz)$
is a growth function such
that $\varphi(\cdot,t)$ is a Muckenhoupt weight uniformly in $t$,
and $\varphi(x,\cdot)$ is an
Orlicz function of uniformly upper type 1 and lower type $p\in(0,1]$
 (see Subsection \ref{s-hrg} below for the definitions of
 uniformly upper and lower types).

Recall that the \emph{Musielak-Orlicz BMO-type space}
$\mathrm{BMO}^{\varphi}(\mathbb{R}^n)$
is defined as the set of all locally integrable functions
$f$ on $\rn$ such that
$$\|f\|_{\mathrm{BMO}^{\varphi}(\mathbb{R}^n)}:=
\sup_{B\subset\rn}\frac{1}{\|\chi_B\|_{L^{\varphi}
(\mathbb{R}^n)}}\int_{B}|f(x)-f_{B}|\,dx<\infty,$$
where the supremum is taken over all balls $B\subset\mathbb{R}^n$,
$\chi_B$ is the characteristic
function of $B$ and
$$\|\chi_B\|_{L^{\varphi}(\mathbb{R}^n)}:=
\inf\left\{\lambda\in(0,\infty):\ \int_{\mathbb{R}^n}\varphi
\left(x,\frac{\chi_B(x)}{\lambda}\right)
\,dx\leq1\right\}.$$
Notice that Nakai and Yabuta
\cite{ny} proved that the class of pointwise multipliers
for $\mathrm{BMO}(\mathbb{R}^n)$ is just
the space of $L^{\infty}(\mathbb{R}^n)\cap
\mathrm{BMO}^{\mathrm{log}}(\mathbb{R}^n)$,
where $\mathrm{BMO}^{\mathrm{log}}(\mathbb{R}^n)$ denotes
the Musielak-Orlicz BMO-type space $\mathrm{BMO}^{\varphi}(\mathbb{R}^n)$
related to the growth function
$$\varphi(x,t):=\frac{t}{\ln(e+|x|)+\ln(e+t)}$$
for all $x\in\mathbb{R}^n$ and $t\in[0,\infty)$. Furthermore,
Ky \cite{ky} proved that
$\mathrm{BMO}^{\varphi}(\mathbb{R}^n)$ is the dual space of
the Musielak-Orlicz Hardy
space $H^{\fai}(\rn)$, which was also introduced in \cite{ky} and
includes both the Orlicz-Hardy space
$H_\Phi(\rn)$ in \cite{s79,ja80} and the weighted Hardy space
$H^p_\oz(\rn)$ with
$p\in(0,1]$ and $\oz\in A_{\fz}(\rn)$ in \cite{gr,st}. Here,
$A_{q}(\rn)$, $q\in[1,\infty]$, denotes the \emph{class of Muckenhoupt weights}.
Moreover, more interesting applications
of these spaces were also presented in
\cite{bfg10,bgk12,lhy12,bijz07,ky,k2,k3,k4,k5}.
Notice that Musielak-Orlicz functions are the
natural generalization of Orlicz functions which may vary in the
spatial variable (see, for example, \cite{d05,dhr09,ky,m83}).
The motivation to study function spaces of Musielak-Orlicz type
is due to that they have wide applications to
many branches of physics and mathematics
(see, for example,
\cite{bg10,bgk12,bijz07,d05,dhr09,ky,l05,yys}).

Moreover, Duong and Yan \cite{dy} introduced a new BMO-type
function space on a space $\cx$ of homogeneous type
in the sense of Coifman and Weiss \cite{cw2,cw1},
which is associated with a generalized approximation to
the identity and generalizes the classical BMO spaces in another way.
Precisely,
let $\{A_t\}_{t>0}$ be a class of integral operators, defined by
kernels $\{a_t\}_{t>0}$ (which decay fast enough) in the sense that,
for all $x\in\cx$ and
functions $f$ satisfying some growth condition on $\cx$,
$$A_tf(x):=\int_{\mathcal{X}}a_t(x,y)f(y)\,d\mu(y).$$
Duong and Yan \cite{dy} first introduced the suitable function set
$\mathcal{M}(\mathcal{X})$ such that,
for all $f\in \mathcal{M}(\mathcal{X})$
and all $t,\,s\in(0,\fz)$, $A_tf<\infty$
and $A_s(A_tf)<\infty$ almost everywhere.
Then the BMO\emph{-type space}
$\mathrm{BMO}_A(\mathcal{X})$
is defined as the set of all
$f\in \mathcal{M}(\mathcal{X})$ such that
\begin{equation*}
\|f\|_{\mathrm{BMO}_A(\mathcal{X})}:=\sup_{B\subset\cx}
\frac{1}{\mu(B)}\int_{B}
|f(x)-A_{t_B}f(x)|\,d\mu(x)<\infty,
\end{equation*}
where the supremum is taken over all balls $B$ in $\cx$ and
$t_{B}:=r_{B}^m$ with $r_{B}$ being
the radius of the ball $B$ and $m$ a positive constant.
Duong and Yan \cite{dy} gave out
some basic properties of $\bbmo_A(\cx)$ including a variant of
the John-Nirenberg inequality and
further proved that the space $\mathrm{BMO}_{\sqrt{\Delta}}(\mathbb{R}^n)$,
associated with the Poisson semigroup of the Laplace operator $\Delta$ on
$\mathbb{R}^n$, and
$\mathrm{BMO}(\mathbb{R}^n)$ coincide with equivalent norms.
Tang \cite{tang} introduced the Morrey-Campanato
type spaces $\mathrm{Lip}_A(\alpha,\mathcal{X})$
associated with the generalized approximation to the identity
$\{A_t\}_{t> 0}$ and established the John-Nirenberg inequality on
these spaces. Furthermore, Deng, Duong and Yan \cite{ddy05} established
a new characterization of the classical Morrey-Campanato space on $\rn$
by using an appropriate convolution to replace the
minimizing polynomial of a function $f$ in the Morrey-Campanato norm.
Moreover, a similar characterization for the Morrey space on $\rn$ was also
obtained by Duong, Xiao and Yan in \cite{dxy07}.
Yang and Zhou \cite{yz10} introduced some
generalized approximations to the identity
with optimal decay conditions in the sense that these
conditions are sufficient and necessary
for these generalized approximations to the identity
to characterize $\mathrm{BMO}(\mathcal{X})$, which was
introduced by Long and Yang \cite{ly84}.
Furthermore, a new John-Nirenberg-type inequality associated
with the generalized approximations to
the identity on $\mathrm{BMO}(\cx)$
was also established in \cite{yz10}. Recently,
Bui and Duong \cite{bd} introduced the weighted BMO space
 $\mathrm{BMO}_{A}(\mathcal{X},\omega)$
associated to the generalized approximations
to the identity, $\{A_t\}_{t> 0}$,
and also obtained the  John-Nirenberg inequality on
these spaces.

Let $\mathcal{X}$ be a space of homogeneous type with degree
$(\alpha_0,n_0,N_0)$ (see Remark \ref{r-deg} below for its definition),
where $\alpha_0$, $n_0$ and $N_0$ are as in
\eqref{d-ar0}, \eqref{d-n0} and \eqref{d-N0} below,
respectively.
Let $\varphi:\ \cx\times[0,\infty)\to[0,\fz)$ be a growth function such
that $\varphi(\cdot,t)$ is a Muckenhoupt weight uniformly in $t$, and
$\varphi(x,\cdot)$ is an
Orlicz function of uniformly upper type 1
and lower type $p\in(0,1]$.
Motivated by \cite{ky,dy,tang}, in this article,
we first introduce the Musielak-Orlicz type space
$\mathrm{BMO}^{\varphi}(\mathcal{X})$ by a
way similar to that used in \cite{ky}, and then
introduce the new Musielak-Orlicz BMO-type space
$\mathrm{BMO}^{\varphi}_A(\mathcal{X})$, via replacing
the mean value $f_B$ (see \eqref{d-fb}) in the definition of
$\mathrm{BMO}^{\varphi}(\mathcal{X})$ by
$A_{t_B}f$, motivated by Duong and Yan \cite{dy}, which generalizes
the space $\mathrm{BMO}_A(\mathcal{X})$
associated with the generalized approximation to the identity
$\{A_t\}_{t> 0}$ in \cite{dy},
the Morrey-Campanato type space $\mathrm{Lip}_A(\alpha,\,\mathcal{X})$
associated with $\{A_t\}_{t> 0}$ in \cite{tang},
and the weighted BMO space $\mathrm{BMO}_{A}(\mathcal{X},\,\omega)$
associated with $\{A_t\}_{t> 0}$ in \cite{bd}.
Then we give out some basic properties of
$\mathrm{BMO}^{\varphi}_A(\mathcal{X})$ and establish its
two equivalent characterizations, respectively,
in terms of the spaces
$\mathrm{BMO}^{\varphi}_{A,\,\mathrm{max}}(\mathcal{X})$ and
$\widetilde{\mathrm{BMO}}^{\varphi}_A(\mathcal{X})$.
Moreover, two variants of the John-Nirenberg inequality are
obtained on $\mathrm{BMO}^{\varphi}_A(\mathcal{X})$, which
generalize the John-Nirenberg inequalities
established in \cite{dy,tang,bd}.
As an application, we further prove that the space
$\mathrm{BMO}^{\varphi}_{\sqrt{\Delta}}(\mathbb{R}^n)$, associated
with the Poisson semigroup of the Laplace operator on
$\mathbb{R}^n$, and
$\mathrm{BMO}^{\varphi}(\mathbb{R}^n)$ coincide with equivalent norms,
which means that the new Musielak-Orlicz BMO-type spaces
$\mathrm{BMO}^{\varphi}_A(\rn)$ also generalize
$\mathrm{BMO}^{\varphi}(\mathbb{R}^n)$ introduced by Ky \cite{ky}.

Precisely, this article is organized as follows.
In Section \ref{s-BMf}, we first recall notions
of spaces of homogenous type and growth functions $\varphi$
considered in this article. We then give out several examples
of growth functions as well as their basic properties. After recalling the Musielak-Orlicz space
$L^{\varphi}(\mathcal{X})$, we then introduce the Musielak-Orlicz
BMO-type space $\mathrm{BMO}^{\varphi}(\mathcal{X})$
on the space $\mathcal{X}$ of homogeneous type and further
give out some useful properties for $L^{\varphi}(\mathcal{X})$
(see Lemma \ref{l-kai} below) and $\mathrm{BMO}^{\varphi}(\mathcal{X})$
(see Proposition \ref{p-fKB} below), which are needed in establishing
the equivalence between $\mathrm{BMO}^{\varphi}(\mathcal{X})$
and the new Musielak-Orlicz BMO-type space
$\mathrm{BMO}^{\varphi}_A(\mathcal{X})$ introduced in the next section.

In Section \ref{s-BfA}, we introduce the generalized approximation to
the identity $\{A_t\}_{t>0}$ with kernels $\{a_t\}_{t>0}$,
which satisfy appropriate decay conditions related to
the growth function $\varphi$ (see \eqref{pdcu} and
\eqref{dcubf} below), and the class
$\mathcal{M}(\mathcal{X})$ of functions in which functions
have proper growth condition
(see \eqref{Mx0b} below) and are suitable to
$\{A_t\}_{t>0}$ (see Lemma \ref{ubfA} below).
Based on this, we introduce
the new Musielak-Orlicz BMO-type space
$\mathrm{BMO}^{\varphi}_A(\mathcal{X})$
associated with $\{A_t\}_{t>0}$ (see Definition \ref{d-BfA} below).
We prove that, if $\{A_t\}_{t>0}$ satisfies that,
for all $t\in(0,\fz)$,
$A_t(1)=1$ almost everywhere, then
$\mathrm{BMO}^{\varphi}(\mathcal{X})\subset
\mathrm{BMO}^{\varphi}_A(\mathcal{X})$ (see Proposition \ref{p-BAB} below).
We also give out some useful properties
for $\mathrm{BMO}^{\varphi}_A(\mathcal{X})$
(see Propositions \ref{p-AtB} and \ref{p-isB} below),
including some size estimates for functions
in $\mathrm{BMO}^{\varphi}_A(\mathcal{X})$, which play an
important role in the study for
$\mathrm{BMO}^{\varphi}_A(\mathcal{X})$.
Moreover, we also introduce the Musielak-Orlicz BMO-type spaces,
$\mathrm{BMO}^{\varphi}_{A,\,\mathrm{max}}(\mathcal{X})$
(see Definition \ref{d-Bma}) and $\widetilde{\mathrm{BMO}}^{\varphi}_A(\mathcal{X})$
(see Definition \ref{d-wB}), associated with $\{A_t\}_{t>0}$,
and further prove that, when $\{A_t\}_{t>0}$ satisfies an additional size
condition (see \eqref{atbB} below), these spaces are equivalent with
$\mathrm{BMO}^{\varphi}_A(\mathcal{X})$
(see Theorems \ref{t-eBm} and \ref{t-ewB} below). We point out that
Theorems \ref{t-eBm} and \ref{t-ewB} completely cover, respectively,
\cite[Proposition 2.10 and 2.12]{dy} by taking
\begin{eqnarray}\label{fat}
\fai(x,t):=t \ \ \mathrm{for\ all}\
x\in\cx\  \mathrm{and}\  t\in[0,\fz),
\end{eqnarray}
and Theorem \ref{t-ewB} completely covers
\cite[Proposition 2.4]{tang} by taking
\begin{eqnarray}\label{fatb}
\fai(x,t):=t^{1/(1+\bz)},\ \
\mathrm{with}\ \bz\in(0,\fz),\
\mathrm{for\ all}\ x\in\cx\ \mathrm{and}\ t\in[0,\fz)
\end{eqnarray}
(see Remark \ref{r4.1} below).

In Section \ref{s-tuj}, we establish two variants of
the John-Nirenberg inequality on
$\mathrm{BMO}^{\varphi}_A(\mathcal{X})$. The first one
(see Theorem \ref{t-uj} below) is closer to the John-Nirenberg inequalities
established in \cite{dy,tang}. We remark that Theorem \ref{t-uj} completely
covers \cite[Theorem 3.1]{dy} and \cite[Theorem 3.1]{tang}
by taking $\varphi$, respectively, as in \eqref{fat} and \eqref{fatb}
(see Remark \ref{r-uj} below).
While the second one (see Theorem \ref{t-j} below)
 is closer to
the John-Nirenberg inequalities on the weighted BMO-type spaces obtained
in \cite{mw1,bd,lai}.
It is worth pointing out that
Theorem \ref{t-j}(i) completely covers \cite[Theorem 3.1]{dy}
and \cite[Theorem 3.1]{tang} by taking $\fai$,
respectively, as in \eqref{fat} and \eqref{fatb}, and
\cite[Theorem 3.6]{bd} by taking
\begin{eqnarray}\label{faw1t}
\varphi(x,t):=\omega(x)t,\
\mathrm{with}\ \omega\in A_1(\mathcal{X}),\ \mathrm{for\ all}\
x\in\cx\ \mathrm{and}\
t\in[0,\fz).
\end{eqnarray}
Moreover, Theorem \ref{t-j}(ii) is new even when
\begin{eqnarray}\label{fawt}
\fai(x,t):=&&\oz(x)t\ \mathrm{for\ all}\ x\in\cx\ \mathrm{and}\
t\in[0,\fz)\ \mathrm{with}\ \oz\in A_\fz(\cx)\\
&&\hs\mathrm{satisfying}\  p_\oz\le 1+\frac{1}{r_\oz'},\nonumber
\end{eqnarray}
where $p_\oz$ and $r_\oz$ denote the \emph{critical indices}
of the weight $\oz$,
which are defined by a way
similar to that used in \eqref{crAp} and \eqref{crRD} below
and $r_w'$ denotes the conjugate index of $r_w$
(see Remark \ref{r-j} below).
Furthermore, we study the relationship between these
two John-Nirenberg inequalities
in Remark \ref{r-j}(iii) below when
$\varphi\in\mathbb{A}_1(\mathcal{X})$.

For $\widetilde{p}\in[1,\infty)$, we also introduce
the Musielak-Orlicz BMO-type spaces
$\mathrm{BMO}^{\varphi,\,\widetilde{p}}_A(\mathcal{X})$ and
$\widetilde{\mathrm{BMO}}^{\varphi,\,\widetilde{p}}_A(\mathcal{X})$
(see Definitions \ref{d-puj} and \ref{d-pj} below).
As applications of these John-Nirenberg inequalities on $
\mathrm{BMO}^{\varphi}_A(\mathcal{X})$ obtained in Theorems \ref{t-uj} and
\ref{t-j}, we further prove that, for any $\widetilde{p}\in[1,\fz)$,
the spaces $\mathrm{BMO}^{\varphi,\,\widetilde{p}}_A(\mathcal{X})$,
$\widetilde{\mathrm{BMO}}^{\varphi,\,\widetilde{p}}_A(\mathcal{X})$
and $\mathrm{BMO}^{\varphi}_A(\mathcal{X})$
coincide with  equivalent norms
(see Theorems \ref{t-puj} and \ref{t-pj} below).
We remark that Theorem \ref{t-puj}
completely covers \cite[Theorem 3.4]{dy} and \cite[Theorem 3.4]{tang}
by taking $\varphi$, respectively,
as in \eqref{fat} and \eqref{fatb}. Moreover,
Theorem \ref{t-pj} is also new even when
$\fai$ is as in \eqref{fawt}.

In Section \ref{s-eBd}, as applications of Theorems \ref{t-puj} and
\ref{t-pj}, the boundedness of
the classical Littlewood-Paley $g$-function on $L^2(\rn)$ and the
$\varphi$-Carleson measure characterization of
$\mathrm{BMO}^{\varphi}(\mathbb{R}^n)$ obtained
in \cite{hyy} (see also Lemma \ref{l-fBr} below),
we prove that the space $\mathrm{BMO}^{\varphi}_{\sqrt{\Delta}}
(\mathbb{R}^n)$, associated with
the Poisson semigroup of the Laplace operator on $\rn$, and
$\mathrm{BMO}^{\varphi}(\mathbb{R}^n)$ coincide with equivalent norms
(see Theorem \ref{t-fBa} below), which
completely covers \cite[Theorem 2.14]{dy} by taking $\varphi$
as in \eqref{fat} (see Remark \ref{r6.1} below).
By a similar way, we also prove that the space
$\mathrm{BMO}^{\varphi}_{\Delta}
(\mathbb{R}^n)$, associated with
the heat semigroup of the Laplace operator on $\rn$, and
$\mathrm{BMO}^{\varphi}(\mathbb{R}^n)$ coincide with equivalent norms
(see Theorem \ref{t-fGr} below), which, together with Theorems
 \ref{t-eBm}, \ref{t-ewB} and \ref{t-fBa}, implies that
 the spaces
 $$\mathrm{BMO}^{\varphi}(\mathbb{R}^n),\
\mathrm{BMO}^{\varphi}_{\sqrt{\Delta}}(\mathbb{R}^n),\
\mathrm{BMO}^{\varphi}_{\sqrt{\Delta},\,\max}(\mathbb{R}^n),\
\widetilde{\mathrm{BMO}}^{\varphi}_{\sqrt{\Delta}}(\mathbb{R}^n),\
\mathrm{BMO}^{\varphi}_{\Delta}(\mathbb{R}^n),$$
$\mathrm{BMO}^{\varphi}_{\Delta,\,\max}(\mathbb{R}^n)$
and $\widetilde{\mathrm{BMO}}^{\varphi}_{\Delta}(\mathbb{R}^n)$
coincide with equivalent norms (see Corollary \ref{c-far} below).
We point out that Theorems \ref{t-fBa}, \ref{t-fGr}
and Corollary \ref{c-far} completely cover,
respectively,
\cite[Theorems 2.14, 2.15 and Corollary 2.16]{dy}
by taking
$\fai$ as in \eqref{fat}.
Moreover, Theorems \ref{t-fBa} and \ref{t-fGr}
and Corollary \ref{c-far} are
also new even when $\varphi$ is as in \eqref{fatb}.

We remark that the key points of the above approach are to
establish the basic properties of
$\mathrm{BMO}^{\varphi}(\mathcal{X})$ and
$\mathrm{BMO}^{\varphi}_A(\mathcal{X})$
(see Propositions \ref{p-fKB}, \ref{p-AtB} and \ref{p-isB} below),
and the John-Nirenberg inequalities on the space $\mathrm{BMO}^{\varphi}_A
(\mathcal{X})$ (see Theorems \ref{t-uj} and \ref{t-j} below).
To this end, we first give out some basic properties of
growth functions $\varphi$
(see Lemmas \ref{l-uAp} and \ref{l-kai} below).
Moreover, the essential difficulty to establish Proposition \ref{p-AtB} comes from
the inseparability of the space variable $x$ and the time variable $t$ appeared in
the grown function $\vz(x,t)$. To overcome this difficulty,
we first clarify, in \eqref{dM} below, the relation between the degree
$(\alpha_0,n_0,N_0)$ of $\cx$, the uniformly lower type
critical index $i(\fai)$ (see \eqref{cult} below), the uniformly Muckenhoupt weight
critical index $p(\fai)$ (see \eqref{crAp} below) and the uniformly reverse H\"older
critical index $q(\fai)$ (see \eqref{crRD} below)
of $\fai$, and the decay order $M$ for the  kernels
$\{a_t\}_{t>0}$ of the generalized approximation to
the identity $\{A_t\}_{t>0}$ (see \eqref{dcubf}). In the proof
of Proposition \ref{p-AtB}, we also need to use dyadic cubes in $\cx$ established
by Christ \cite{ch} (see also Lemma \ref{l-deX} below) and borrow some
ideas from Duong and Yan \cite{dy} to deal with the time parameter $t$ appeared in
$\{A_t\}_{t>0}$.
Using these properties of $\mathrm{BMO}^{\varphi}_A(\mathcal{X})$,
we establish  two variants
of the John-Nirenberg inequality on the space $\mathrm{BMO}^{\varphi}_A
(\mathcal{X})$. Precisely, we obtain the first
John-Nirenberg inequality on $\mathrm{BMO}^{\varphi}_A
(\mathcal{X})$, in Theorem \ref{t-uj} below, by borrowing some ideas from the proof of
\cite[Theorem 3.1]{dy} and using some delicate estimates of
the growth function $\varphi$. Furthermore, following the ways in
\cite[Theorem 3]{mw1} and \cite[Theorem 1]{lai}, we establish the
second John-Nirenberg inequality on $\mathrm{BMO}^{\varphi}_A
(\mathcal{X})$ in Theorem \ref{t-j} below, via using the
Whitney decomposition established
in \cite[Chapter III, Theorem 1.3]{cw2}
and some basic properties of $\varphi$.
Here we also borrow some ideas from the
John-Nirenberg inequality on the
Musielak-Orlicz Campanato spaces
${\mathcal L}_{\varphi,1,s}({\mathbb R}^n)$
established by Liang and Yang \cite{ly13} and
choose the time variant $t:=\|\chi_{B}\|_{L^{\varphi}(\mathcal{X})}^{-1}$
to overcome some essential difficulties caused by the inseparability of
the space variable $x$ and the time variable $t$ appeared in $\vz(x,t)$.

Finally we make some conventions on notation. Throughout the
article, we denote by $C$ a \emph{positive constant} which is
independent of the main parameters, but it may vary from line to
line. The \emph{symbol} $A\ls B$ means that $A\le CB$. If
$A\ls B$ and $B\ls A$, then we write $A\sim B$. The  \emph{symbol}
$\lfz s\rfz$ for $s\in\rr$ denotes the maximal integer not more
than $s$. For any given normed spaces $\mathcal A$ and $\mathcal
B$ with the corresponding norms $\|\cdot\|_{\mathcal A}$ and
$\|\cdot\|_{\mathcal B}$, the \emph{symbol} ${\mathcal
A}\subset{\mathcal B}$ means that, for all $f\in \mathcal A$, then
$f\in\mathcal B$ and $\|f\|_{\mathcal B}\ls \|f\|_{\mathcal A}$.
For any subset $E$ of the space $\mathcal{X}$ of homogeneous type,
 we denote by $E^\complement$ the
\emph{set} $\mathcal{X}\setminus E$ and by
$\chi_{E}$  its \emph{characteristic function}. We also set
$\nn:=\{1,\,2,\, \ldots\}$ and $\mathbb{Z}_{+}:=\nn\cup\{0\}$.
For any index $q\in[1,\fz]$, we denote by $q'$ its
\emph{conjugate index}, namely, $1/q+1/q'=1$. Also,
for any $\az\in (0,\fz)$ and ball $B:=B(x_B,r_B):=\{x\in\cx:\ d(x,x_B)<r_B\}$
with $x_B\in\cx$ and $r_B\in (0,\fz)$, we denote by $\az B$
the ball $B(x_B,\az r_B)$.

\section{Spaces of homogeneous type,
growth functions and \\Musielak-Orlicz BMO-type spaces
}\label{s-BMf}

\hskip\parindent
In this section, we introduce
the Musielak-Orlicz
BMO-type spaces $\mathrm{BMO}^{\varphi}(\mathcal{X})$ on
RD-spaces $\mathcal{X}$.
To this end, we first recall some notions
on spaces of homogeneous type,
RD-spaces and growth functions
considered in this article. Then  we state some properties
of the growth functions. Finally, we give out a basic property for
$\mathrm{BMO}^{\varphi}(\mathcal{X})$.

\subsection{Spaces of homogeneous type and
growth functions}\label{s-hrg}

\hskip\parindent
We first recall the notion of spaces of homogeneous
type in the sense of
Coifman and Weiss \cite{cw2,cw1}.

\begin{definition}
A function $d:\ \mathcal{X}\times\mathcal{X}\rightarrow[0,\infty)$
is called a \emph{quasi-metric}, if it satisfies the following conditions:

\ \ {\rm (i)} $d(x,y)=0$ if and only if $x=y$;

\ {\rm (ii)} $d(x,y)=d(y,x)$ for all $x,y\in\mathcal{X}$;

{\rm (iii)} there exists a constant $C_1\in[1,\infty)$ such that,
for all $x,y,z\in\mathcal{X}$,
\begin{equation}\label{qume}
d(x,y)\leq C_1[d(x,z)+d(z,y)].
\end{equation}
\end{definition}

The quasi-metric $d$ defines a topology for which the balls
$B(x,r):=\{y\in \mathcal{X}:\ d(y,x)<r\}$ for all
$x\in\mathcal{X}$ and $r\in(0,\infty)$ form a basis. However, when
$C_1\in(1,\fz)$, the balls need not be open
(see, for example, \cite{cw2}).

\begin{definition}\label{hosp}
A \emph{space of homogeneous type} $(\mathcal{X},\,d,\,\mu)$
is a set $\mathcal{X}$
equipped with a quasi-metric $d$ and a nonnegative Borel
measure $\mu$ on $\mathcal{X}$ for
which there exists a constant $C_2\in[1,\fz)$ such that, for
all balls $B(x,r)$,
$$\mu(B(x,2r))\le C_2\mu(B(x,r))<\fz\ \ \ \ \ (\hbox{Doubling Property}).$$
\end{definition}

\begin{remark}
(i) The doubling property implies the following
\emph{strong homogeneity property}:
there exist positive constants $n$ and $C$ such that,
for all $x\in\cx$, $r\in(0,\fz)$ and
$\lz\in[1,\fz)$,
\begin{equation}\label{sthp}
\mu(B(x,\lambda r))\le C\lambda^n\mu(B(x,r)).
\end{equation}
Let
\begin{equation}\label{d-n0}
n_0:=\inf\{n\in(0,\infty):\ n\ \hbox{is as in}\ (\ref{sthp})\}.
\end{equation}
The parameter $n_0$ is a measure of the ``dimension'' of $\mathcal{X}$.
Observe that $n_0\in[0,\infty)$ and \eqref{sthp} may not be true for $n_0$.

(ii)
There also exist a positive constant $C$ and $N\in[0,\infty)$ such that,
for all $x,\,y\in\cx$ and $r\in(0,\fz)$,
\begin{equation}\label{eBdc}
\mu(B(y,r))\le C\left[1+\frac{d(x,y)}{r}\right]^N\mu(B(x,r)).
\end{equation}
Indeed, let $n$ be as in \eqref{sthp}. When $N=n$, \eqref{eBdc} is deduced from
the quasi-triangle inequality \eqref{qume}
of the quasi metric $d$ and the strong homogeneity property \eqref{sthp}.
In the case of Euclidean
spaces $\mathbb{R}^n$ and Lie groups of polynomial growth, $N=0$.

Let
\begin{equation}\label{d-N0}
N_0:=\inf\{N\in[0,\infty):\ N\ \mathrm{is\ as\ in}\ (\ref{eBdc})\}.
\end{equation}
Observe that $N_0\in[0,n_0]$ and \eqref{eBdc} may not be true for $N_0$.
\end{remark}

Now we recall the notion of the RD-space introduced in
\cite{hmy08}
(see also
\cite{yz} for more properties of RD-spaces).

\begin{definition}\label{d-RDs}
The triple $(\mathcal{X},\,d,\,\mu)$ is called an
RD\emph{-space}, if there exist
a constant $\alpha\in(0,n]$ and $C\in[1,\fz)$ such that,
for all $x\in\mathcal{X}$,
$r\in(0,2\,\mathrm{diam}(\mathcal{X}))$ and
$\lambda\in[1,2\,\mathrm{diam}(\mathcal{X})/r)$,
\begin{equation}\label{RDin}
C^{-1}\lambda^{\alpha}\mu(B(x,r))\leq\mu(B(x,\lambda r))\leq
C\lambda^n\mu(B(x,r)),
\end{equation}
where $n$ is as in \eqref{sthp} and
$\mathrm{diam}(\mathcal{X}):=\sup_{x,\,y\in\mathcal{X}}d(x,y)$.
\end{definition}

\begin{remark}\label{r-deg}
Obviously, an RD-space is a space of homogeneous type.
It is also known that
a connected space of homogeneous type is an RD-space (see \cite{yz}).
Let
\begin{equation}\label{d-ar0}
\alpha_0:=\sup\{\alpha\in[0,n]:\ \alpha\ \hbox{ is as in}\ (\ref{RDin})\}.
\end{equation}
Obviously, for an RD-space $\mathcal{X}$,
$\alpha_0\in(0,n_0]$ and \eqref{RDin}
may not be true for $\alpha_0$.
If $\mathcal{X}$ is only known to be a space of homogenous
type, then \eqref{RDin} may hold true only for $\alpha=0$,
namely, $\alpha_0=0$ in this case.
In what follows, the triple $(\alpha_0,n_0,N_0)$ is called
the \emph{degree} of the space of homogeneous type, $\mathcal{X}$.
\end{remark}

Throughout this article, we \emph{always assume} that $\mathcal{X}$ is a
space of homogeneous type with degree $(\alpha_0,n_0,N_0)$.

Next, we recall that a function
$\Phi:[0,\fz)\to[0,\fz)$ is called an \emph{Orlicz function} if it
is nondecreasing, $\Phi(0)=0$, $\Phi(t)>0$ for all $t\in(0,\fz)$ and
$\lim_{t\to\fz}\Phi(t)=\fz$ (see, for example,
\cite{m83,rr91,rr00}). We point out that, different from
the classical definition of Orlicz functions, the \emph{Orlicz function in
this article may not be convex}. The function $\Phi$ is said to be of {\it
upper type $p$} (resp. \emph{lower type $p$}) for some $p\in[0,\fz)$ if
there exists a positive constant $C$ such that, for all
$s\in[1,\fz)$ (resp. $s\in(0,1)$) and $t\in[0,\fz)$,
$\Phi(st)\le Cs^p \Phi(t)$.
If $\Phi$ is of both upper type $p_1$ and lower type $p_2$,
then $p_1\geq p_2$ and
$\Phi$ is said to be of \emph{type $(p_1,\,p_2)$}.

Let $\mathcal{X}$ be a space of homogeneous type. For a given function
$\fai:\,\cx\times[0,\fz)\to[0,\fz)$ such that, for
any $x\in\cx$, $\fai(x,\cdot)$ is an Orlicz function,
$\fai$ is said to be of \emph{uniformly upper type $p$} (resp.
\emph{uniformly lower type $p$}) for some $p\in[0,\fz)$, if there
exists a positive constant $C$ such that, for all $x\in\cx$,
$t\in[0,\fz)$ and $s\in[1,\fz)$ (resp. $s\in(0,1)$),
\begin{equation*}
\fai(x,st)\le Cs^p\fai(x,t).
\end{equation*}
Moreover, $\fai$ is said to be of \emph{positive uniformly upper type}
(resp. \emph{uniformly lower type}), if it is of uniformly upper
type (resp. uniformly lower type) $p$ for some $p\in(0,\fz)$, and
let
\begin{equation}\label{cult}
i(\fai):=\sup\{p\in(0,\fz):\ \fai\ \text{is of uniformly lower
type}\ p\}.
\end{equation}
Observe that $i(\fai)$ may not be attainable, namely,
$\fai$ may not be of uniformly
lower type $i(\fai)$ (see, for example, \cite{yys,yys4}).

\begin{definition}\label{d-uAp}
Let $\mathcal{X}$ be a space of homogeneous type and
$\fai:\cx\times[0,\fz)\to[0,\fz)$.
The function $\fai(\cdot,t)$ is said to satisfy the
\emph{uniformly Muckenhoupt condition for some $p\in[1,\fz)$},
denoted by $\fai\in\aa_p(\cx)$, if, when $p\in (1,\fz)$,
\begin{equation*}
\aa_p(\fai):=\sup_{t\in
(0,\fz)}\sup_{B\subset\cx}\frac{1}
{\mu(B)}\int_B\varphi(x,t)\,d\mu(x)
\left\{\frac{1}{\mu(B)}\int_B[\varphi(x,t)]^{-\frac{1}{p-1}}
\,d\mu(x)\right\}^{p-1}<\fz,
\end{equation*}
or, when $p=1$,
\begin{equation*}
\aa_1 (\fai):=\sup_{t\in (0,\fz)}
\sup_{B\subset\cx}\frac{1}{\mu(B)}\int_B \fai(x,t)\,d\mu(x)
\lf\{\esup_{y\in B}[\fai(y,t)]^{-1}\r\}<\fz.
\end{equation*}
Here the first supremums are taken over all $t\in(0,\fz)$ and the
second ones over all balls $B\subset\cx$.

The function $\fai(\cdot,t)$ is said to satisfy the
\emph{uniformly reverse H\"older condition for some
$q\in(1,\fz]$}, denoted by $\fai\in \rh_q(\cx)$, if, when $q\in
(1,\fz)$,
\begin{eqnarray*}
\rh_q (\fai):&&=\sup_{t\in (0,\fz)}\sup_{B\subset\cx}\lf\{\frac{1}
{\mu(B)}\int_B [\fai(x,t)]^q\,d\mu(x)\r\}^{1/q}\lf\{\frac{1}{\mu(B)}\int_B
\fai(x,t)\,d\mu(x)\r\}^{-1}<\fz,
\end{eqnarray*}
or, when $q=\infty$,
\begin{equation*}
\rh_{\fz} (\fai):=\sup_{t\in
(0,\fz)}\sup_{B\subset\cx}\lf\{\esup_{y\in
B}\fai(y,t)\r\}\lf\{\frac{1}{\mu(B)}\int_B
\fai(x,t)\,d\mu(x)\r\}^{-1} <\fz.
\end{equation*}
Here the first supremums are taken over all $t\in(0,\fz)$ and the
second ones over all balls $B\subset\cx$.
\end{definition}

We point out that, in Definition \ref{d-uAp}, when $\cx:=\rn$,
$\aa_p(\rn)$ with $p\in[1,\fz)$ was introduced by Ky \cite{ky} and, moreover,
for any metric space $\cx$ with doubling measure,
the notions of $\aa_p(\cx)$, with $p\in[1,\fz)$, and
$\rh_{q}(\cx)$, with $q\in(1,\fz]$, were introduced in \cite{yys4}.

Let $\aa_{\fz}(\cx):=\cup_{p\in[1,\fz)}\aa_{p}(\cx)$ and
the \emph{critical indices} of $\fai$ be defined as follows:
\begin{equation}\label{crAp}
p(\fai):=\inf\lf\{p\in[1,\fz):\ \fai\in\aa_{p}(\cx)\r\}
\end{equation}
and
\begin{equation}\label{crRD}
r(\fai):=\sup\lf\{q\in(1,\fz]:\ \fai\in\rh_{q}(\cx)\r\}.
\end{equation}
Observe that, if $p(\fai)\in(1,\fz)$, then $\fai\not\in\aa_{p(\fai)}(\cx)$,
and there exists $\fai\not\in\aa_1(\cx)$ such that $p(\fai)=1$
 (see, for example, \cite{jn87}). Similarly,
if $r(\fai)\in(1,\fz)$, then $\fai\not\in\rh_{r(\fai)}(\cx)$, and there exists
$\fai\not\in\rh_\fz(\cx)$ such that $r(\fai)=\fz$
(see, for example, \cite{cn95}).

Now we introduce the notion of growth functions.

\begin{definition}\label{d-grf}
Let $\mathcal{X}$ be a space of homogeneous type. The function
$\fai:\cx\times[0,\fz)\to[0,\fz)$ is called
 a \emph{growth function} if the following hold true:
 \vspace{-0.25cm}
\begin{enumerate}
\item[(i)] $\fai$ is a \emph{Musielak-Orlicz function}, namely,
\vspace{-0.2cm}
\begin{enumerate}
    \item[(i)$_1$] the function $\fai(x,\cdot):\ [0,\fz)\to[0,\fz)$ is an
    Orlicz function for all $x\in\cx$;
    \vspace{-0.2cm}
    \item [(i)$_2$] the function $\fai(\cdot,t)$ is a measurable
    function for all $t\in[0,\fz)$.
\end{enumerate}
\vspace{-0.25cm}\item[(ii)] $\fai\in \aa_{\fz}(\cx)$.
\vspace{-0.25cm}\item[(iii)] The function $\fai$ is of
uniformly upper type $1$ and of uniformly
lower type $p\in(0,1]$.
\end{enumerate}
\end{definition}

Clearly, $\fai(x,t):=t^p$ for all $(x,t)\in\mathcal{X}\times[0,\infty)$
with $p\in(0,1]$ and, more generally,
$\fai(x,t):=\oz(x)\Phi(t)$ for all $(x,t)\in\mathcal{X}\times[0,\infty)$
with
$\oz\in A_{\fz}(\cx)$
and $\Phi$ being an Orlicz function of upper
type $1$ and lower type $p\in(0,1]$ are growth functions.
Let $x_0\in\cx$. Another typical and useful growth
function is
$$\fai(x,t):=\frac{t^{s}}{[\ln(e+d(x,x_0))]^{\bz}+[\ln(e+t)]^{\gz}}$$
for all $x\in\cx$ and $t\in[0,\fz)$ with some $s\in(0,1]$,
$\bz\in[0,\alpha)$ and $\gz\in [0,2s(1+\ln2)]$,
where $\alpha$ is as in \eqref{RDin}. It is easy to show that
$\fai\in \aa_1(\cx)$, $\fai$ is of uniformly upper type $s$ and
$i(\fai)=s$ which is not attainable. We also point out that, when $\cx:=\rn$,
a similar
example of such $\fai$ is given by Ky \cite{ky} via replacing $d(x,x_0)$ by
$|x|$, where $|\cdot|$ denotes the Euclidean distance on $\rn$;
see, for example, \cite{yys,yys4} for more examples of growth functions.

\subsection{Musielak-Orlicz BMO-type spaces
$\mathrm{BMO}^{\varphi}(\mathcal{X})$}\label{s-mbf}

\hskip\parindent
Let us first recall the Musielak-Orlicz-type space
$L^{\varphi}(\mathcal{X})$. Recall that $\mathcal{X}$ is \emph{always assumed to be a
space of homogeneous type}.

\begin{definition}
Let $\mathcal{X}$ be a
space of homogeneous type
and $\varphi$ a growth function as in Definition \ref{d-grf}.
The \emph{Musielak-Orlicz-type space}
$L^{\varphi}(\mathcal{X})$ is defined to be
the space of all measurable functions $f$ such that
$\int_{\mathcal{X}}\varphi(x,|f(x)|)\,d\mu(x)<\infty$ endowed with the
\emph{Luxembourg norm}
$$\|f\|_{L^{\varphi}(\mathcal{X})}:=
\inf\left\{\lambda\in(0,\infty):\ \int_{\mathcal{X}}\varphi
\left(x,\frac{|f(x)|}{\lambda}\right)
\,d\mu(x)\leq1\right\}.$$
\end{definition}

Now we are ready to introduce
Musielak-Orlicz
BMO-type spaces $\mathrm{BMO}^{\varphi}(\mathcal{X})$ as follows.

\begin{definition}
Let $\mathcal{X}$ be a
space of homogeneous type and $\varphi$ a growth function.
A locally integrable function $f$ on $\cx$ is said to belong to the
\emph{Musielak-Orlicz BMO-type space}
$\mathrm{BMO}^{\varphi}(\mathcal{X})$, if
$$\|f\|_{\mathrm{BMO}^{\varphi}(\mathcal{X})}:=
\sup_{B\subset\cx}\frac{1}{\|\chi_B\|_{L^{\varphi}
(\mathcal{X})}}\int_{B}|f(x)-f_{B}|\,d\mu(x)<\infty,$$
where the supremum
is taken over all balls $B\subset\mathcal{X}$ and
\begin{eqnarray}\label{d-fb}
f_{B}:=\frac{1}{\mu(B)}\int_Bf(x)\,d\mu(x).
\end{eqnarray}
\end{definition}

\begin{remark}\label{r-Bf}
(i) When $\varphi$ is as in \eqref{fat},
then $\|\chi_B\|_{L^{\varphi}(\mathcal{X})}=\mu(B)$ and hence
$\mathrm{BMO}^{\varphi}(\mathcal{X})$
is just the space $\mathrm{BMO}(\mathcal{X})$
on the space of homogeneous type, $\mathcal{X}$, introduced by
Long and Yang \cite{ly84}.
$\mathrm{BMO}^{\varphi}(\mathbb{R}^n)$ was introduced by
Ky (see \cite{ky}). When $\mathcal{X}:=\rn$ and $\varphi$ is as in \eqref{fat},
then $\|\chi_B\|_{L^{\varphi}(\rn)}=|B|$ and hence
$\mathrm{BMO}^{\varphi}(\mathbb{R}^n)$ is just the classical
$\mathrm{BMO}(\mathbb{R}^n)$ space
introduced by John and Nirenberg \cite{jn};
when $\mathcal{X}:=\rn$ and $\varphi$ is as in \eqref{fawt} without
the restriction $p_\oz\le 1+1/r_\oz'$, then
$\|\chi_B\|_{L^{\varphi}(\mathcal{X})}=\omega(B)$ and hence
$\mathrm{BMO}^{\varphi}(\mathbb{R}^n)$ is just the weighted BMO space
$\mathrm{\mathrm{BMO}}_\oz(\mathbb{R}^n)$, which
was first introduced by Muckenhoupt
and Wheeden \cite{mw2, mw1}.

(ii)
Another
typical example of the space $\mathrm{BMO}^{\varphi}(\mathbb{R}^n)$ is
$\mathrm{BMO}^{\mathrm{log}}(\mathbb{R}^n)$,
which is  related to the growth function
$\varphi(x,t)=\frac{t}{\ln(e+|x|)+\ln(e+t)}$, $x\in\rn$
and $t\in[0,\fz)$. Notice that the class of pointwise multipliers
for $\mathrm{BMO}(\mathbb{R}^n)$, characterized
by Nakai and Yabuta \cite{ny}, is just
the space $L^{\infty}(\mathbb{R}^n)\cap
\mathrm{BMO}^{\mathrm{log}}(\mathbb{R}^n)$
(see \cite{ky} for more details).
\end{remark}

To give out a basic property of
$\mathrm{BMO}^{\varphi}(\mathcal{X})$,
we need the following lemmas concerning growth functions.
\begin{lemma}\label{l-grf}
Let $\mathcal{X}$ be a
space of homogeneous type and $\varphi$ as in Definition \ref{d-grf}.

\ {\rm(i)} It holds true that
$\int_{\mathcal{X}}\varphi(x,\frac{|f(x)|}
{\|f\|_{L^{\varphi}(\mathcal{X})}})
\,d\mu(x)=1$
for all $f\in L^{\varphi}(\mathcal{X})\setminus\{0\}$.

{\rm(ii)} Let $c$ be a positive constant. Then there exists a positive
constant $C$, depending on $c$, such that, if $\int_\cx
\fai(x,\frac{|f(x)|}{\lz})\,d\mu(x)\le c$ for some $\lz\in(0,\fz)$,
then $\|f\|_{L^{\fai}(\cx)}\le C\lz$.
\end{lemma}

Lemma \ref{l-grf} when $\mathcal{X}:=\rn$ is just
\cite[Lemmas 4.2(i) and 4.3(i)]{ky} and, moreover,
its proof is also similar to those proofs in
\cite{ky},
the details being omitted.

\begin{lemma}\label{l-uAp}
Let $\mathcal{X}$ be a
space of homogeneous type and $\varphi$ as in Definition \ref{d-grf}.

{\rm (i)} If $\varphi \in\mathbb{A}_p(\cx)$ with $p\in[1,\fz)$,
then there exists
a positive constant $C$ such that, for all balls
$B_1,B_2\subset \cx$ with
$B_1\subset B_2$ and $t\in[0,\infty)$,
$\frac{\varphi(B_2,t)}{\varphi(B_1,t)}\leq
C[\frac{\mu(B_2)}{\mu(B_1)}]^p.$

 {\rm(ii)} If $\varphi \in\mathbb{RH}_q(\cx)$ with
 $q\in(1,\infty]$, then there
exists a positive constant $C$ such that,
for all balls $B_1,B_2\subset \cx$
with $B_1\subset B_2$ and $t\in[0,\infty)$,
$\frac{\varphi(B_2,t)}{\varphi(B_1,t)}\geq C\left
[\frac{\mu(B_2)}{\mu(B_1)}\right]^{\frac{q-1}{q}}.$
\end{lemma}

The proof of Lemma \ref{l-uAp} is similar to
that of the corresponding
conclusions in $\rn$ (see, for example, \cite{gr,gra}),
the details being omitted.

\begin{lemma}\label{l-kai}

Let $\mathcal{X}$ be a
space of homogeneous type and $\varphi$ as in Definition \ref{d-grf} with
uniformly lower type $p\in(0,1]$.
Assume that $\varphi\in \mathbb{A}_{p_1}(\cx)$ with
$p_1\in[1,\infty)$, and $\varphi\in \mathbb{RH}_{q}(\cx)$ with
$q\in(1,\infty]$. Then there exists a positive constant $C$ such that,
for all balls $B_1,\,B_2\subset\mathcal{X}$ with $B_1\subset B_2$,
\begin{equation}\label{Bbkai}
\|\chi_{B_2}\|_{L^{\varphi}(\mathcal{X})}\leq C\left[\frac{\mu(B_2)}
{\mu(B_1)}\right]^{\frac{p_1}{p}}
\|\chi_{B_1}\|_{L^{\varphi}(\mathcal{X})}
\end{equation}
and
\begin{equation}\label{bBkai}
\|\chi_{B_1}\|_{L^{\varphi}(\mathcal{X})}\leq C\left[\frac{\mu(B_1)}
{\mu(B_2)}\right]^{\frac{q-1}{q}}
\|\chi_{B_2}\|_{L^{\varphi}(\mathcal{X})}.
\end{equation}
\end{lemma}

\begin{proof}
We first prove \eqref{Bbkai}. By the uniformly lower type $p$ property
of $\fai$, Lemmas \ref{l-grf}(i) and \ref{l-uAp}(i), we know that
\begin{eqnarray*}
&&\int_{B_2}\varphi\left(x,\frac{1}{[\mu(B_2)/\mu(B_1)]^{p_1/p}
\|\chi_{B_1}\|_{L^{\varphi}(\mathcal{X})}}\right)
\,d\mu(x)\\
&&\hs\ls[\mu(B_2)/\mu(B_1)]^{-p_1}\int_{B_2}\varphi\left(x,\frac{1}
{\|\chi_{B_1}\|_{L^{\varphi}(\mathcal{X})}}\right)\,
d\mu(x)\ls\varphi\left(B_1,\frac{1}
{\|\chi_{B_1}\|_{L^{\varphi}(\mathcal{X})}}\right)
\sim1,
\end{eqnarray*}
which, together with Lemma \ref{l-grf}(ii), implies that
\eqref{Bbkai} holds true.

By using the uniformly upper type 1 property of $\fai$
and Lemma \ref{l-uAp}(ii),
 we conclude that \eqref{bBkai} holds true by a way similar
to the above proof of \eqref{Bbkai}, the details being omitted,
which completes the proof of
Lemma \ref{l-kai}.
\end{proof}

\begin{proposition}\label{p-fKB}
Let $\mathcal{X}$ be a space of homogeneous type with degree
$(\alpha_0,n_0,N_0)$,
where $\alpha_0$, $n_0$ and $N_0$ are as in
\eqref{d-ar0}, \eqref{d-n0} and \eqref{d-N0},
respectively. Let $n\in(n_0,\infty)$, $\alpha\in(0,\alpha_0)$ and
$\alpha=0$ when $\alpha_0=0$.
Assume that $\fai$ is as in Definition \ref{d-grf} with
$\fai\in\aa_{p_1}(\cx)$ and $\fai$ of uniformly lower type $p$, where
$p_1\in[1,\fz)$ and $p\in(0,1]$.
Then there exists a positive constant $C$ such that,
for all $f\in \mathrm{BMO}^{\varphi}(\mathcal{X})$,
balls $B\subset\cx$ and
$K\in(1,\fz)$,
$$|f_B-f_{KB}|\leq CK^{\frac{np_1}{p}-\alpha}
\frac{\|\chi_{B}\|_{L^{\varphi}(\mathcal{X})}}{\mu(B)}
\|f\|_{\mathrm{BMO}^{\varphi}(\mathcal{X})},$$
where $f_B$ is as in \eqref{d-fb} and $f_{KB}$ defined similarly.
\end{proposition}

\begin{proof}
Let $K\in(1,\fz)$. Then there exists  $m\in\nn$ such that
$e^m\leq K< e^{m+1}$. If $\mathrm{diam}(\mathcal{X})=\infty$,
by \eqref{RDin} and \eqref{Bbkai}, we see that
\begin{eqnarray}\label{fKB1}
|f_B-f_{KB}|&\leq&|f_B-f_{eB}|+|f_{eB}-f_{e^2B}|+\cdots+
|f_{e^mB}-f_{KB}|\\\nonumber
&\leq&\frac{\|\chi_{eB}\|_{L^{\varphi}(\mathcal{X})}}
{\mu(B)}\frac{1}{\|\chi_{eB}\|_{L^{\varphi}(\mathcal{X})}}
\int_{eB}|f(x)-f_{eB}|\,d\mu(x)\\\nonumber
&&\hs+\frac{\|\chi_{e^2B}\|_{L^{\varphi}(\mathcal{X})}}
{\mu(eB)}\frac{1}{\|\chi_{e^2B}\|_{L^{\varphi}(\mathcal{X})}}
\int_{e^2B}|f(x)-f_{e^2B}|\,d\mu(x)\\\nonumber
&&\hs+\cdots+\frac{\|\chi_{KB}\|_{L^{\varphi}(\mathcal{X})}}
{\mu(e^mB)}\frac{1}{\|\chi_{KB}\|_{L^{\varphi}(\mathcal{X})}}
\int_{KB}|f(x)-f_{KB}|\,d\mu(x)\\\nonumber
&\lesssim& \left[e^{\frac{np_1}{p}}+e^{\frac{np_1}{p}-\alpha}
+\cdots+
e^{m(\frac{np_1}{p}-\alpha)}\right]
\frac{\|\chi_{B}\|_{L^{\varphi}(\mathcal{X})}}{\mu(B)}
\|f\|_{\mathrm{BMO}^{\varphi}(\mathcal{X})}\\\nonumber
&\lesssim&\left
\{e^{\frac{np_1}{p}}+\frac{e^{\frac{np_1}{p}-\alpha}}
{e^{\frac{np_1}{p}-\alpha}-1}
\left[e^{m(\frac{np_1}{p}-\alpha)}-1\right]\right\}
\frac{\|\chi_{B}\|_{L^{\varphi}(\mathcal{X})}}{\mu(B)}
\|f\|_{\mathrm{BMO}^{\varphi}(\mathcal{X})}\\\nonumber
&\lesssim& K^{\frac{np_1}{p}-\alpha}
\frac{\|\chi_{B}\|_{L^{\varphi}(\mathcal{X})}}{\mu(B)}
\|f\|_{\mathrm{BMO}^{\varphi}(\mathcal{X})},
\end{eqnarray}
which is desired.

Now we consider the case that $\mathrm{diam}(\mathcal{X})<\infty$.
Let $B:=B(x_B,r_B)$. Assume that there
exists $m_0\in\mathbb{Z}_+$ with $m_0<m$ such that
$2e^{m_0}r_B<\mathrm{diam}(\mathcal{X})\leq2e^{m_0+1}r_B;$
otherwise, we obtain the desired conclusion by repeating
the procedure same as in \eqref{fKB1}.
In the case that $m_0<m$, it is easy to see that
$\mu(\mathcal{X})\thicksim\mu(e^{m_0+1}B)$. From this,
\eqref{RDin} and \eqref{Bbkai}, it follows that
\begin{eqnarray}\label{fKB2}
\ \ \ \ \ \ \ |f_B-f_{KB}|&\leq&|f_B-f_{eB}|+\cdots
+|f_{e^{m_0}B}-f_{e^{m_0+1}B}|\\\nonumber
&&+|f_{e^{m_0+1}B}-f_{e^{m_0+2}B}|+\cdots+
|f_{e^mB}-f_{KB}|\\\nonumber
&\leq&\frac{\|\chi_{eB}\|_{L^{\varphi}(\mathcal{X})}}
{\mu(B)}\frac{1}{\|\chi_{eB}\|_{L^{\varphi}(\mathcal{X})}}
\int_{eB}|f(x)-f_{eB}|\,d\mu(x)+\cdots\\\nonumber
&&+\frac{\|\chi_{e^{m_0+1}B}\|_{L^{\varphi}(\mathcal{X})}}
{\mu(e^{m_0}B)}\frac{1}{\|\chi_{e^{m_0+1}B}\|_{L^{\varphi}(\mathcal{X})}}
\int_{e^{m_0+1}B}|f(x)-f_{e^{m_0+1}B}|\,d\mu(x)\\\nonumber
&&+\frac{\|\chi_{e^{m_0+2}B}\|_{L^{\varphi}(\mathcal{X})}}
{\mu(e^{m_0+1}B)}\frac{1}{\|\chi_{e^{m_0+2}B}\|_{L^{\varphi}(\mathcal{X})}}
\int_{e^{m_0+2}B}|f(x)-f_{e^{m_0+2}B}|\,d\mu(x)\\\nonumber
&&+\cdots+\frac{\|\chi_{KB}\|_{L^{\varphi}(\mathcal{X})}}
{\mu(e^mB)}\frac{1}{\|\chi_{KB}\|_{L^{\varphi}(\mathcal{X})}}
\int_{KB}|f(x)-f_{KB}|\,d\mu(x)\\\nonumber
&\lesssim&\left[e^{\frac{np_1}{p}}+\cdots+
e^{m_0(\frac{np_1}{p}-\alpha)}+
e^{m_0(\frac{np_1}{p}-\alpha)}\right.\\\nonumber
&&\hs\left.+\cdots+
e^{m_0(\frac{np_1}{p}-\alpha)}\right]
\frac{\|\chi_{B}\|_{L^{\varphi}(\mathcal{X})}}{\mu(B)}
\|f\|_{\mathrm{BMO}^{\varphi}(\mathcal{X})}\\\nonumber
&\lesssim&\left
\{e^{\frac{np_1}{p}}+\frac{e^{\frac{np_1}{p}-\alpha}}
{e^{\frac{np_1}{p}-\alpha}-1}
\left[e^{m_0(\frac{np_1}{p}-\alpha)}-1\right]
+(m-m_0)e^{m_0(\frac{np_1}{p}-\alpha)}\right\}\\\nonumber
&&\times
\frac{\|\chi_{B}\|_{L^{\varphi}(\mathcal{X})}}{\mu(B)}
\|f\|_{\mathrm{BMO}^{\varphi}(\mathcal{X})}\lesssim K^{\frac{np_1}{p}-\alpha}
\frac{\|\chi_{B}\|_{L^{\varphi}(\mathcal{X})}}{\mu(B)}
\|f\|_{\mathrm{BMO}^{\varphi}(\mathcal{X})},
\end{eqnarray}

which, together with \eqref{fKB1}, completes
the proof of Proposition \ref{p-fKB}.
\end{proof}

\section{Musielak-Orlicz BMO-type spaces
$\mathrm{BMO}^{\varphi}_A(\mathcal{X})$ associated \\
with generalized approximations to the identity}\label{s-BfA}

\hskip\parindent
In this section, we first introduce
Musielak-Orlicz BMO-type spaces
$\mathrm{BMO}^{\varphi}_A(\mathcal{X})$ associated with
generalized approximations to the identity, $\{A_t\}_{t>0}$, and then
give out their basic properties and two equivalent
characterizations in terms of the space
$\mathrm{BMO}^{\varphi}_{A,\,\mathrm{max}}(\mathcal{X})$
(see Definition \ref{d-Bma} below) and the space
$\widetilde{\mathrm{BMO}}^{\varphi}_A(\mathcal{X})$
(see Definition \ref{d-wB} below).

\subsection{Definition of
$\mathrm{BMO}^{\varphi}_A(\mathcal{X})$}\label{s-dfA}

\hskip\parindent
Let $\mathcal{X}$ be a space of homogeneous type with degree
$(\alpha_0,n_0,N_0)$, where $\alpha_0$, $n_0$ and $N_0$ are as in
\eqref{d-ar0}, \eqref{d-n0} and \eqref{d-N0},
respectively.
Let $x_0\in\mathcal{X}$,
\begin{eqnarray}\label{dcubf}
M>n_0[1+p(\fai)/i(\fai)]-\az_0,
\end{eqnarray}
where $n_0,\,p(\fai),\,i(\fai)$
and $\az_0$ are, respectively, as in \eqref{d-n0},
\eqref{crAp}, \eqref{cult} and \eqref{d-ar0}, and
$\bz\in(0,M-n_0[1+p(\fai)/i(\fai)]+\az_0)$.
A function $f\in L^1_{\loc}(\cx)$ is said to be of
\emph{type $(x_0,\beta)$},
if there exists a positive constant $C$ such that
\begin{eqnarray}\label{Mx0b}
\int_{\mathcal{X}}\frac{|f(x)|}{[1+d(x_0,x)]
^{\frac{n_0p(\fai)}{i(\varphi)}-\alpha_0+\beta}
\mu(B(x_0,1+d(x_0,x)))}\,d\mu(x)\leq C.
\end{eqnarray}
Moreover, denote by $\mathcal{M}_{(x_0,\beta)}(\mathcal{X})$
the \emph{collection of all
function of type} $(x_0,\beta)$. The \emph{norm} of $f$
in $\mathcal{M}_{(x_0,\beta)}(\mathcal{X})$ is defined by
$$\|f\|_{\mathcal{M}_{(x_0,\beta)}(\mathcal{X})}:=
\inf\{C\in(0,\infty):\ \eqref{Mx0b}\
  \hbox{holds true}\}.$$
For a fixed $x_0\in \mathcal{X}$, it is easy to see that
$\mathcal{M}_{(x_0,\beta)}(\mathcal{X})$ is a Banach space under the norm
$\|\cdot\|_{\mathcal{M}_{(x_0,\beta)}(\mathcal{X})}$. Moreover,
it is easy to show that, for any
$x_1\in \mathcal{X}$, $\mathcal{M}_{(x_1,\beta)}(\mathcal{X})
=\mathcal{M}_{(x_0,\beta)}(\mathcal{X})$
with equivalent norms. Let
\begin{equation}\label{Max0b}
\mathcal{M}(\mathcal{X}):=\bigcup_{x_0\in\mathcal{X}}\bigcup_{\{\beta:\,
0<\beta<M-n_0-N_0\}}\mathcal{M}_{(x_0,\beta)}(\mathcal{X}),
\end{equation}
where $n_0$, $N_0$ and $M$ are
 as in \eqref{d-n0}, \eqref{d-N0} and \eqref{dcubf},
respectively.

To give the definition of the space $\bbmo^\fai_A(\cx)$,
we also need to recall the notion of
the generalized approximation to the identity, $\{A_t\}_{t>0}$.
In this article, we always assume that, \emph{for any $t\in(0,\fz)$,
the operator $A_t$
is defined by the kernel $a_t$ in the sense that}
$$A_tf(x):=\int_{\mathcal{X}}a_t(x,y)f(y)\,d\mu(y)$$
\emph{for all $f\in\mathcal{M}(\mathcal{X})$ and $x\in\mathcal{X}$}.

We further assume that, for any $t\in(0,\infty)$,
the kernel $a_t$ satisfies that, for all $x,\ y\in \mathcal{X}$,
$|a_t(x,y)|\leq h_t(x,y)$,
where $h_t(x,y)$
is given by setting, for all $x,y\in\mathcal{X},$
\begin{eqnarray}\label{ubfA}
h_t(x,y):=\frac{1}{\mu(B(x,t^{1/m}))}g\left(\frac{[d(x,y)]^m}{t}\right),
\end{eqnarray}
in which $m$ is a positive constant and $g$ a positive,
bounded, decreasing
function satisfying that
\begin{eqnarray}\label{pdcu}
\lim_{r\to \infty}r^Mg(r^m)=0,
\end{eqnarray}
where $M$ is as in \eqref{dcubf}.

It is easy to prove that there exists
a positive constant $C$ such that, for all
$x\in\cx$ and $t\in(0,\fz)$,
$$C^{-1}\leq\int_{\mathcal{X}}h_t(x,y)\,d\mu(y)\leq C\ \ \
\hbox{and}\ \ \ C^{-1}\leq\int_{\mathcal{X}}h_t(y,x)\,d\mu(y)\leq C$$
 (see also \cite{dm}).

Then we have the following technical lemma.

\begin{lemma}\label{l-Afi}
Let $\mathcal{X}$ be a space of homogeneous type with degree
$(\alpha_0,n_0,N_0)$,
where $\alpha_0$, $n_0$ and $N_0$ are as in
\eqref{d-ar0}, \eqref{d-n0} and \eqref{d-N0},
respectively.  Assume that $\fai$ is as in Definition \ref{d-grf} and
$\{A_t\}_{t>0}$ a generalized approximation to the identity satisfying
\eqref{ubfA} and \eqref{pdcu}.

\ \ {\rm (i)} If $f\in \mathrm{BMO}^{\varphi}(\mathcal{X})$,
then $f\in \mathcal{M}(\mathcal{X})$.

\ {\rm (ii)} For any $t\in(0,\fz)$ and $f\in \mathcal{M}(\mathcal{X})$,
it holds true that
$|A_tf(x)|<\infty$ for almost every $x\in \mathcal{X}$.

{\rm (iii)} For any $t,s\in(0,\fz)$ and $f\in \mathcal{M}(\mathcal{X})$,
it holds true that
$|A_t(A_sf)(x)|<\infty$ for almost every $x\in \mathcal{X}$.

Moreover, if, for almost every $x,y\in\mathcal{X}$,
\begin{eqnarray}\label{spa}
a_{t+s}(x,y)=\int_{\mathcal{X}}a_t(x,z)a_s(z,y)\,d\mu(z),
\end{eqnarray}
then, for any
$f\in \mathcal{M}(\mathcal{X})$,
$A_{t+s}f=A_t(A_sf)$
almost everywhere.
\end{lemma}

\begin{proof}
Let $f\in \mathrm{BMO}^{\varphi}(\mathcal{X})$. For any
$x_0\in \mathcal{X}$, fix a ball $B:=B(x_0,1)$ centered
at $x_0$ and of radius 1. Let $\beta$ be as in \eqref{Mx0b}.
By the definitions of $n_0$, $\alpha_0$, $p(\varphi)$ and $i(\varphi)$,
respectively, as in \eqref{d-n0}, \eqref{d-ar0}, \eqref{crAp} and \eqref{cult},
we know that there exist $n\in[n_0,\infty)$, $\alpha\in[0,\alpha_0]$,
$p_1\in[p(\fai),\fz)$ and
$p\in(0,i(\fai)]$ such that $\mathcal{X}$ satisfies
\eqref{sthp} and \eqref{RDin}, respectively, for $n$ and $\alpha$,
$\varphi\in\mathbb{A}_{p_1}(\mathcal{X})$,
$\varphi$ is of uniformly lower type $p$ and
$\frac{n_0p(\varphi)}{i(\varphi)}-\alpha_0+\beta>\frac{np_1}{p}-\alpha$.
From this and Proposition \ref{p-fKB}, it follows that,
for all $k\in\mathbb{N}$,
$$|f_B-f_{2^kB}|\ls 2^{k(\frac{np_1}{p}-\alpha)}
\frac{\|\chi_{B}\|_{L^{\varphi}(\mathcal{X})}}{\mu(B)}
\|f\|_{\mathrm{BMO}^{\varphi}(\mathcal{X})},$$
which, together with \eqref{Bbkai} and
$\frac{n_0p(\varphi)}{i(\varphi)}-\alpha_0+\beta>\frac{np_1}{p}-\alpha$,
implies that
\begin{eqnarray*}
&&\int_{\mathcal{X}}\frac{|f(y)-f_B|}{[1+d(x_0,y)]
^{\frac{n_0p(\varphi)}{i(\varphi)}-\alpha_0+\beta}
\mu(B(x_0,1+d(x_0,y)))}\,d\mu(y)\\
&&\hs\leq\sum^{\infty}_{k=0}\int_{2^kB \backslash 2^{k-1}B}
\frac{|f(y)-f_B|}{[1+d(x_0,y)]^{\frac{n_0p(\varphi)}{i(\varphi)}
-\alpha_0+\beta}
\mu(B(x_0,1+d(x_0,y)))}\,d\mu(y)\\
&&\hs\lesssim\sum^{\infty}_{k=0}2^{-k[
\frac{n_0p(\varphi)}{i(\varphi)}
-\alpha_0+\beta]}\left\{[\mu(2^kB)]^{-1}
\int_{2^kB}|f(y)-f_{2^kB}|\,d\mu(y)+|f_B-f_{2^kB}|\right\}\\
&&\hs\lesssim\sum^{\infty}_{k=0}2^{-k[\frac{n_0p(\varphi)}{i(\varphi)}
-\alpha_0-\frac{np_1}{p}+\alpha+\beta]}
\frac{\|\chi_B\|_{L^{\varphi}(\mathcal{X})}}{\mu(B)}
\|f\|_{\mathrm{BMO}^{\varphi}(\mathcal{X})}
\lesssim\frac{\|\chi_B\|_{L^{\varphi}(\mathcal{X})}}{\mu(B)}
\|f\|_{\mathrm{BMO}^{\varphi}(\mathcal{X})},
\end{eqnarray*}
where $2^{-1}B:=\emptyset$, $f_B$ is as in \eqref{d-fb}
and $f_{2^kB}$ defined similarly. Moreover, it is easy to see that
$$C_{x_0}:=\int_{\mathcal{X}}\frac{1}{[1+d(x_0,y)]
^{\frac{n_0p(\varphi)}{i(\varphi)}-\alpha_0+\beta}
\mu(B(x_0,1+d(x_0,y)))}\,d\mu(y)<\infty.$$
Thus, we have
$$\|f\|_{\mathcal{M}_{(x_0,\beta)}(\mathcal{X})}\lesssim
\frac{\|\chi_B\|_{L^{\varphi}(\mathcal{X})}}{\mu(B)}
\|f\|_{\mathrm{BMO}^{\varphi}(\mathcal{X})}
+C_{x_0}|f_B|<\infty,$$
which implies that $f\in \mathcal{M}_{(x_0,\beta)}(\mathcal{X})$,
and hence $f\in \mathcal{M}(\mathcal{X})$.

The proofs of (ii) and (iii) are similar to that of
\cite[Lemma 2.3]{dy},
the details being omitted, which completes the proof of Lemma \ref{l-Afi}.
\end{proof}

\begin{remark}\label{r-sp}
Recall that, if a generalized approximation to the identity $\{A_t\}_{t>0}$
satisfies \eqref{spa}, then $\{A_t\}_{t>0}$ is said
to have the \emph{semigroup property}.
\end{remark}

We now introduce the space
$\mathrm{BMO}^{\varphi}_A(\mathcal{X})$ associated
with the generalized approximation to the identity $\{A_t\}_{t>0}$.

\begin{definition}\label{d-BfA}
Let $\mathcal{X}$ be a
space of homogeneous type, $\varphi$ as in Definition \ref{d-grf}
and $\{A_t\}_{t>0}$ a generalized approximation to the identity satisfying
\eqref{ubfA} and \eqref{pdcu}.
The \emph{Musielak-Orlicz BMO-type space $\mathrm{BMO}^{\varphi}_A(\mathcal{X})$
associated with
$\{A_t\}_{t>0}$}
is defined to be
the space of all functions $f\in \mathcal{M}(\mathcal{X})$ such that
\begin{eqnarray*}
\|f\|_{\bbmo^\fai_A(\cx)}:=\sup_{B\subset\cx}\frac{1}
{\|\chi_B\|_{L^{\varphi}(\mathcal{X})}}
\int_{B}|f(x)-A_{t_B}f(x)|\,d\mu(x)<\fz,
\end{eqnarray*}
where the supremum is taken over all balls $B\subset\cx$,
$t_{B}:=r_{B}^m$, $r_{B}$ is the radius of ball $B$ and
$m$ as in \eqref{ubfA}.
\end{definition}
\begin{remark}\label{r-BfA}
{\rm (i)} We  point out that $(\mathrm{BMO}^{\varphi}_A(\mathcal{X}),\,
\|\cdot\|_{\mathrm{BMO}^{\varphi}_A(\mathcal{X})})$ is a
seminormed vector space, with the seminorm vanishing on the
space $\mathcal{K}_A(\cx)$, which is defined by
$$\mathcal{K}_A(\cx):=\{f\in\mathcal{M}(\mathcal{X}):\ A_tf(x)=f(x)\ \hbox{for}\
\mu\hbox{-almost}\ \hbox{every}\ x\in\mathcal{X}\ \hbox{and}\
\hbox{all}\ t\in(0,\infty)\}.$$
Then, it is customary to think
$\mathrm{BMO}^{\varphi}_A(\mathcal{X})$ to be
modulo $\mathcal{K}_A(\cx)$.

{\rm(ii)} When $\varphi$ is as in \eqref{fat},
then $\|\chi_B\|_{L^{\varphi}(\mathcal{X})}=\mu(B)$
and the space $\mathrm{BMO}^{\varphi}_A(\mathcal{X})$ is just
the space $\mathrm{BMO}_A(\mathcal{X})$
associated with
$\{A_t\}_{t> 0}$ introduced by Duong and
Yan \cite{dy}; when
$\varphi$ is as in \eqref{fatb},
then $\|\chi_B\|_{L^{\varphi}(\mathcal{X})}=[\mu(B)]^{1+\bz}$
and the space $\mathrm{BMO}^{\varphi}_A(\mathcal{X})$ is just
the Morrey-Campanato type spaces $\mathrm{Lip}_A(\bz,\,\mathcal{X})$
introduced by Tang \cite{tang};
when $\varphi$ is as in \eqref{fawt} without the restriction
$p_\oz\le 1+1/r_\oz'$,
$\|\chi_B\|_{L^{\varphi}(\mathcal{X})}=\omega(B)$ and
the space $\mathrm{BMO}^{\varphi}_A(\mathcal{X})$ is just
the weighted BMO space $\mathrm{BMO}_{\mathcal{A}}(\mathcal{X},\omega)$
introduced
by Bui and Duong \cite{bd}.
\end{remark}

Now we establish a relation between the spaces
$\mathrm{BMO}^{\varphi}_A(\mathcal{X})$
and $\mathrm{BMO}^{\varphi}(\mathcal{X})$.

\begin{proposition}\label{p-BAB}
Let $\mathcal{X}$ be a space of homogeneous type with degree
$(\alpha_0,n_0,N_0)$,
where $\alpha_0$, $n_0$ and $N_0$ are as in
\eqref{d-ar0}, \eqref{d-n0} and \eqref{d-N0},
respectively. Assume that $\fai$ is as in Definition \ref{d-grf},
$\{A_t\}_{t>0}$ a generalized approximation to the identity satisfying
\eqref{ubfA} and \eqref{pdcu}, and,
for any $t\in(0,\fz)$, $A_t(1)=1$ almost
everywhere, namely,
$\int_{\mathcal{X}}a_t(x,y)\,d\mu(y)=1$ for almost every
$x\in \mathcal{X}$.
Then, $\mathrm{BMO}^{\varphi}(\mathcal{X})\subset
\mathrm{BMO}^{\varphi}_A(\mathcal{X})$ and there exists a
positive constant $C$ such that, for all
$f\in\mathrm{BMO}^{\varphi}(\mathcal{X})$,
\begin{eqnarray}\label{iBAB}
\|f\|_{\mathrm{BMO}^{\varphi}_A(\mathcal{X})}\leq C
\|f\|_{\mathrm{BMO}^{\varphi}(\mathcal{X})}.
\end{eqnarray}
However, the reverse inequality does not hold true in general.
\end{proposition}

\begin{proof}
Let $M$ be as in \eqref{dcubf}. By $M>n_0[1+p(\fai)/i(\fai)]-\az_0$,
we see that there exist $n\in[n_0,\infty)$, $\alpha\in[0,\alpha_0]$,
$p_1\in[p(\fai),\fz)$ and
$p\in(0,i(\fai)]$ such that $\mathcal{X}$ satisfies
\eqref{sthp} and \eqref{RDin}, respectively, for $n$ and $\alpha$,
$\varphi\in\mathbb{A}_{p_1}(\mathcal{X})$,
$\varphi$ is of uniformly lower type $p$ and $M>n(1+\frac{p_1}{p})-\az$.

Let $f\in \mathrm{BMO}^{\varphi}(\mathcal{X})$,
$B:=B(x_B,r_B)$ and $t_B:=r_B^m$.
Then, by $A_t(1)=1$, we conclude that
\begin{eqnarray}\label{dBMf}
&&\frac{1}{\|\chi_{B}\|_{L^{\varphi}(\mathcal{X})}}\int_{B}
|f(x)-A_{t_B}f(x)|\,d\mu(x)\\ \nonumber
&&\hs\leq\frac{1}
{\|\chi_{B}\|_{L^{\varphi}(\mathcal{X})}}
\int_{B}\int_{\mathcal{X}}
h_{t_B}(x,y)|f(x)-f(y)|\,d\mu(y)d\mu(x)\\ \nonumber
&&\hs=\frac{1}{\|\chi_{B}\|_{L^{\varphi}
(\mathcal{X})}}\int_{B}\int_{2B}
h_{t_B}(x,y)|f(x)-f(y)|\,d\mu(y)d\mu(x)\\ \nonumber
&&\hs\hs+\sum_{k=1}^{\infty}\frac{1}
{\|\chi_{B}\|_{L^{\varphi}(\mathcal{X})}}
\int_{B}\int_{2^{k+1}B\backslash 2^{k}B}\cdots=:\hbox{I+II}.
\end{eqnarray}

We first estimate $\hbox{I}$. Since $x\in B$, by \eqref{eBdc}, we
know that $\mu(B)\lesssim\mu(B(x,r_B))$, which, together
with \eqref{sthp}, \eqref{ubfA} and the decreasing property of $g$, implies that,
for all $y\in 2B$,
$$h_{t_B}(x,y)=\frac{g([d(x,y)]^mt^{-1}_B)}{\mu(B(x,r_B))}
\lesssim\frac{g(0)}{\mu(B)}
\lesssim\frac{1}{\mu(2B)}.$$
From this and \eqref{Bbkai}, we deduce that
\begin{eqnarray}\label{eIBAB}
\hbox{I}&\lesssim&\frac{1}
{\|\chi_{B}\|_{L^{\varphi}(\mathcal{X})}\mu(2B)}
\int_{B}\int_{2B}|f(x)-f(y)|\,d\mu(y)d\mu(x)\\ \nonumber
&\lesssim&\frac{1}{\|\chi_{B}\|_{L^{\varphi}(\mathcal{X})}\mu(2B)}
\int_{B}\int_{2B}[|f(x)-f_{2B}|+|f_{2B}-f(y)|]\,d\mu(y)d\mu(x)\\ \nonumber
&\thicksim&\frac{1}
{\|\chi_{B}\|_{L^{\varphi}(\mathcal{X})}}
\int_{2B}|f(x)-f_{2B}|\,d\mu(x)\ls\|f\|
_{\mathrm{BMO}^{\varphi}(\mathcal{X})}.
\end{eqnarray}

Regarding $\hbox{II}$, for $x\in B$ and $y\in 2^{k+1}B\backslash
2^kB$, we see that $d(x,y)\geq 2^{k-1}r_B$. Then, by \eqref{sthp},
\eqref{ubfA} and the decreasing property of $g$,
we conclude that
$$h_{t_B}(x,y)=\frac{g([d(x,y)]^mt^{-1}_B)}{\mu(B(x,r_B))}
\lesssim \frac{g(2^{(k-1)m})}{\mu(B)}\lesssim
\frac{ g(2^{(k-1)m})2^{(k+1)n}}{\mu(2^{{k+1}}B)},$$
which further implies that
\begin{equation}\label{IIBAB}
\hbox{II}\lesssim \sum_{k=1}^{\infty}2^{kn}g(2^{(k-1)m})
\frac{1}{\|\chi_{B}\|_{L^{\varphi}(\mathcal{X})}\mu(2^{k+1}B)}
\int_B \int_{2^{k+1}B}|f(x)-f(y)|\,d\mu(y)d\mu(x).
\end{equation}
Moreover, from Proposition \ref{p-fKB},
it follows that, for each $k\in\nn$,
\begin{eqnarray*}
&&\frac{1}{\|\chi_{B}\|_{L^{\varphi}(\mathcal{X})}\mu(2^{k+1}B)}
\int_B \int_{2^{k+1}B}|f(x)-f(y)|\,d\mu(y)d\mu(x)\\
&&\hs\leq\frac{\mu(B)}
{\|\chi_{B}\|_{L^{\varphi}(\mathcal{X})}\mu(2^{k+1}B)}
\int_{2^{k+1}B}|f(y)-f_{2^{k+1}B}|\,d\mu(y)\\
&&\hs\hs+\frac{1}{\|\chi_{B}\|_{L^{\varphi}(\mathcal{X})}}\int_{B}
|f(x)-f_{2^{k+1}B}|\,d\mu(x)\\
&&\hs\leq\frac{\mu(B)\|\chi_{2^{k+1}B}\|_{L^{\varphi}(\mathcal{X})}}
{\|\chi_{B}\|_{L^{\varphi}(\mathcal{X})}\mu(2^{k+1}B)} \frac{1}
{\|\chi_{2^{k+1}B}\|_{L^{\varphi}(\mathcal{X})}}
\int_{2^{k+1}B}|f(y)-f_{2^{k+1}B}|\,d\mu(y)\\
&&\hs\hs+\frac{1}{\|\chi_{B}\|_{L^{\varphi}(\mathcal{X})}}\int_{B}
|f(x)-f_{B}|\,d\mu(x)+\frac{\mu(B)}
{\|\chi_{B}\|_{L^{\varphi}(\mathcal{X})}}
\left(|f_B-f_{2B}|\right.\\
&&\left.\hs\hs+\cdots+
|f_{2^kB}-f_{2^{k+1}B}|\right)\lesssim2^{k(\frac{np_1}{p}-\alpha)}
\|f\|_{\mathrm{BMO}^{\varphi}(\mathcal{X})}.
\end{eqnarray*}
By this, \eqref{IIBAB}, \eqref{pdcu} and $M>n(1+p_1/p)-\az$,
we find that
$$\hbox{II}\lesssim \sum_{k=1}^{\infty}2^{k(n+\frac{np_1}{p}-\alpha-M)}
\|f\|_{\mathrm{BMO}^{\varphi}(\mathcal{X})}\lesssim
\|f\|_{\mathrm{BMO}^{\varphi}(\mathcal{X})},$$
which, together with \eqref{dBMf} and \eqref{eIBAB},
implies that \eqref{iBAB} holds true.

Finally, we show that the converse inequality of
\eqref{iBAB} does not hold true in
general. We consider $\mathbb{R}$ with the Lebesgue measure $dx$
and the approximation of the identity, $\{A_t\}_{t>0}$, given by the kernels
$$a_t(x,y):=
\frac{1}{2t^{\frac{1}{m}}}
\chi_{(x-t^{\frac{1}{m}},x+t^{\frac{1}{m}})}(y)\
\mathrm{for\ all}\ x,\, y\in \mathbb{R}.$$
Let $f(x)=:x$ for all $x\in\rr$. For every $t\in(0,\fz)$, $A_tf(x)=x$
and $\|f\|_{\mathrm{BMO}^{\varphi}_A(\mathcal{X})}=0$, but
$\|f\|_{\mathrm{BMO}^{\varphi}(\mathcal{X})}\neq 0$.
Thus, the converse inequality
of \eqref{iBAB} does not hold true in general,
which completes the proof of Proposition \ref{p-BAB}.
\end{proof}

\begin{remark}\label{r-At1}
We remark that the assumption $A_t(1)=1$  almost everywhere
is necessary for \eqref{iBAB}.
Indeed, let $f(x):=1$ for all $x\in\mathcal{X}$.
Then \eqref{iBAB} implies that $\|1\|_{\mathrm{BMO}^{\varphi}_A(\mathcal{X})}
=0$ and hence, for every $t\in(0,\fz)$,
$A_t(1)=1$ almost everywhere.
\end{remark}

\subsection{Some basic properties of
$\mathrm{BMO}^{\varphi}_A(\mathcal{X})$}\label{s-bpB}

\hskip\parindent
From now on, we \emph{always need} the
following assumption on the generalized approximation to the identity,
$\{A_t\}_{t>0}$.

\begin{proof}[\rm\bf Assumption $\mathrm{SP}$]
Let $\mathcal{X}$
be a space of homogeneous type,
$\fai$ as in Definition \ref{d-grf} and
$\{A_t\}_{t>0}$ a generalized approximation to the identity satisfying
\eqref{ubfA} and \eqref{pdcu}.
Assume that $A_0$ is the identity operator $I$ and
the operators $\{A_t\}_{t\ge 0}$ have the semigroup property, namely,
for any $t,s\in[0,\infty)$ and $f\in \mathcal{M}(\mathcal{X})$,
$A_tA_sf=A_{t+s}f$ for almost every $x\in\mathcal{X}$
(see also Remark \ref{r-sp}).
\end{proof}

Then we have the following property for
$\{A_t\}_{t>0}$ on
$\mathrm{BMO}^{\varphi}_A(\mathcal{X})$, which is essential for
developing the theory of $\mathrm{BMO}^{\varphi}_A(\mathcal{X})$.

\begin{proposition}\label{p-AtB}
Let $\mathcal{X}$ be a space of homogeneous type with degree
$(\alpha_0,n_0,N_0)$,
where $\alpha_0$, $n_0$ and $N_0$ are as in
\eqref{d-ar0}, \eqref{d-n0} and \eqref{d-N0},
respectively. Assume that $\fai$ is as in Definition \ref{d-grf} and
$\{A_t\}_{t>0}$ satisfies Assumption $\mathrm{SP}$ with
\begin{equation}\label{dM}
M>n+\frac{np_1}p+N_0-
\frac{n[r(\varphi)-1]}{r(\varphi)},
\end{equation}
where $N_0$ and $r(\varphi)$ are, respectively, as in \eqref{d-N0} and
\eqref{crRD},
$n\in[n_0,\infty)$, $p_1\in[p(\fai),\fz)$ and $p\in(0,i(\fai)]$
with $p(\fai)$ and $i(\varphi)$ being, respectively, as in
\eqref{crAp} and \eqref{cult} such that $\cx$ satisfies
\eqref{sthp} for $n$, $\fai\in\aa_{p_1}(\cx)$
and $\fai$ is of uniformly lower type $p$.
Then there exists a positive constant $C$,
depending on $n$, $\az$, $p$ and $p_1$,
such that, for all
$f\in \mathrm{BMO}^{\varphi}_A(\mathcal{X})$, $t\in(0,\infty)$,
$K\in(1,\infty)$ and almost every $x\in \mathcal{X}$,
\begin{eqnarray}\label{AtB1}
|A_tf(x)-A_{Kt}f(x)|\leq C K^{\frac{1}{m}(\frac{np_1}{p}-\alpha)}
\frac{\|\chi_{B(x,t^{1/m})}\|
_{L^{\varphi}(\mathcal{X})}}
{\mu(B(x,t^{1/m}))}\|f\|_{\mathrm{BMO}^{\varphi}_A(\mathcal{X})},
\end{eqnarray}
where $\alpha\in[0,\alpha_0]$ such that $\mathcal{X}$ satisfies
\eqref{RDin} for $\alpha$.
\end{proposition}

To prove Proposition \ref{p-AtB}, we first recall a result of Christ
\cite[Theorem 11]{ch}, which gives an analogue
of Euclidean dyadic cubes.

\begin{lemma}\label{l-deX}
There exists a collection of open subsets, $\{Q_{\alpha}^k\subset
\mathcal{X}:k\in \mathbb{Z}, \alpha \in I_k\}$, where $I_k$ denotes
some (possibly finite) index set, depending on $k$, and constants
$\delta \in(0,1)$, $a_0\in(0,1)$, and $D\in(0,\infty)$ such that

\ \ {\rm (i)} $\mu(\mathcal{X}\backslash \cup_{\alpha}Q_{\alpha}^k)=0$ for
all $k\in \mathbb{Z}$.

\ {\rm (ii)} If $l\geq k$, then either $Q_{\beta}^l\subset Q_{\alpha}^k$ or
$Q_{\beta}^l\cap Q_{\alpha}^k=\emptyset.$

{\rm (iii)} For each $(k,\alpha)$ and each $l<k$, there exists a unique
$\beta$ such that $Q_{\alpha}^k\subset Q_{\beta}^l$.

{\rm (iv)} The diameter of $Q_{\alpha}^k$ is not more than $D\delta^k$.

\ {\rm (v)} Each $Q_{\alpha}^k$ contains some ball
$B(z_{\alpha}^k,a_0\delta^k).$
\end{lemma}
Now we prove Proposition \ref{p-AtB} by using Lemma \ref{l-deX}.

\begin{proof}[Proof of Proposition \ref{p-AtB}] For any given $t\in(0,\infty)$,
choose $s\in(0,\infty)$ such
that $\frac{t}{4}\leq s\leq t$. With the same
notation as in Lemma \ref{l-deX},
we first
fix $l_0$ such that
\begin{eqnarray}\label{lslAB}
D\delta^{l_0}\leq s^{1/m}< D\delta^{l_0-1}.
\end{eqnarray}
Fix $x\in \mathcal{X}$. By (i) and (iv) of Lemma \ref{l-deX},
we see that there exists a subset $Q_{\alpha_0}^{l_0}$ such that $x\in
Q_{\alpha_0}^{l_0}$ and $Q_{\alpha_0}^{l_0}\subset B(x,D\delta^{l_0})$. For any
$k\in \mathbb{N}$, let
$$M_k:=\{\beta\in I_{l_0}:\
Q_{\beta}^{l_0}\cap B(x,D\delta^{l_0-k})\neq\emptyset\},$$
where $I_{l_0}$ is as in Lemma \ref{l-deX}.
Using (i) and (iv) of Lemma \ref{l-deX} again, we know that
$$B(x,D\delta^{l_0-k})\subset\bigcup_{\beta\in M_k}Q^{l_0}_{\beta}\subset B(x,D
\delta^{l_0-(k+k_0)}),$$
where $k_0$ is an integer such that $\delta^{-k_0}\geq 2C_1$ and
$C_1$ is as in \eqref{qume}.

In \cite{dy}, it was proved that there exists a positive constant $C$,
independent of $k$, such that the number of open subsets,
$\{Q_{\beta}^{l_0}\}_{\beta \in M_k}$,
is less than $C\delta^{-k(n+N)}$. Namely,
$$m_k:=\#\{Q^{l_0}_{\beta}:\ \beta \in M_k\}\leq C\delta^{-k(n+N)},$$
where $n$ and $N$ are, respectively, as in \eqref{sthp} and \eqref{eBdc},
and $\#E$ denotes the \emph{cardinality} of the set $E$.

By \eqref{dM}, we know that
there exist $N\in[N_0,\infty)$ and $q\in(1,r(\fai)]$ such that
$\mathcal{X}$ satisfies \eqref{eBdc} for $N$,
$\fai\in\mathbb{RH}_{q}(\cx)$ and $M>n+\frac{np_1}{p}+N-
\frac{n(q-1)}{q}$.

We first estimate $|A_tf(x)-A_{t+s}f(x)|$ for the case
$\frac{t}{4}\leq s\leq t$. By Assumption $\mathrm{SP}$,
we conclude that $A_tf-A_{t+s}f=A_t(f-A_sf)$ almost
everywhere.
From this and $f\in \mathrm{BMO}^{\varphi}_A(\mathcal{X})$, we deduce that,
for almost every $x\in\mathcal{X}$,
\begin{eqnarray}\label{AtsAB}
&&|A_tf(x)-A_{t+s}f(x)|\\ \nonumber
&&\hs\leq\int_{\mathcal{X}}h_t(x,y)|f(y)-A_sf(y)|\,d\mu(y)\\ \nonumber
&&\hs=\frac{1}{\mu(B(x,t^{1/m}))}
\int_{\mathcal{X}}g\left(\frac{[d(x,y)]^m}{t}\right)
|f(y)-A_sf(y)|\,d\mu(y)\\ \nonumber
&&\hs\lesssim\frac{\|\chi_{B(x,s^{1/m})}
\|_{L^{\varphi}(\mathcal{X})}}
{\mu(B(x,t^{1/m}))}\frac{1}
{\|\chi_{B(x,s^{1/m})}\|_{L^{\varphi}(\mathcal{X})}}
\int_{B(x,s^{1/m})}
|f(y)-A_sf(y)|\,d\mu(y)\\ \nonumber
&&\hs\hs+\frac{1}{\mu(B(x,t^{1/m}))}\int_{B(x,s^{1/m})^\complement
}g\left(\frac{[d(x,y)]^m}{t}\right)
|f(y)-A_sf(y)|\,d\mu(y)\\ \nonumber
&&\hs\lesssim\frac
{\|\chi_{B(x,s^{1/m})}\|_{L^{\varphi}(\mathcal{X})}}
{\mu(B(x,t^{1/m}))}\|f\|_{\mathrm{BMO}^{\varphi}_A(\mathcal{X})}
+\mathrm{I},
\end{eqnarray}
where
$$\mathrm{I}:=\frac{1}{\mu(B(x,t^{1/m}))}
\int_{B(x,s^{1/m})^\complement
}g\left(\frac{[d(x,y)]^m}{t}\right)
|f(y)-A_sf(y)|\,d\mu(y).$$
Notice that, for any $y\in B(x,D\delta^{l_0-(k+1)})
\backslash B(x,D\delta^{l_0-k})$, it holds true that
$d(x,y)\geq D\delta^{l_0-k}$, which, together with
\eqref{lslAB} and the decreasing property of $g$, implies that
\begin{eqnarray}\label{sCAtB}
&&\int_{B(x,s^{1/m})^\complement
}g\left(\frac{[d(x,y)]^m}{t}\right)
|f(y)-A_sf(y)|\,d\mu(y)\\\nonumber
&&\hs\leq\int_{B(x,D\delta^{l_0})^\complement}g
\left(\frac{[d(x,y)]^m}{t}\right)
|f(y)-A_sf(y)|\,d\mu(y)\\ \nonumber
&&\hs\leq\sum_{k=0}^{\infty}\int_{B(x,D\delta^{l_0-(k+1)})\backslash B(x,D
\delta^{l_0-k})}g\left(\frac{[d(x,y)]^m}{t}\right)
|f(y)-A_sf(y)|\,d\mu(y)\\ \nonumber
&&\hs\leq\sum_{k=0}^{\infty}g(4^{-1}\delta^{-(k-1)m})
\int_{B(x,D\delta^{l_0-(k+1)})}|f(y)-A_sf(y)|\,d\mu(y)\\ \nonumber
&&\hs\leq\sum_{k=0}^{\infty}\sum_{\beta\in M_{k+1}}
g_1(\delta^{-km})\int_{Q_{\beta}^{l_0}}|f(y)-A_sf(y)|\,d\mu(y),
\end{eqnarray}
where $g_1(\cdot):=g(\delta^m\cdot/4)$ satisfies
the same property as in \eqref{pdcu}.

Applying (iv) of Lemma \ref{l-deX}, we know that $Q_{\beta}^{l_0}\subset
B(z^{l_0}_{\beta},D\delta^{l_0})\subset B(z^{l_0}_{\beta},s^{1/m})$.
By this, \eqref{sCAtB}, $m_k\lesssim\delta^{-k(n+N)}$, \eqref{Bbkai},
\eqref{bBkai}, \eqref{pdcu} and $M>n+\frac{np_1}{p}+N-\frac{n(q-1)}{q}$,
we see that
\begin{eqnarray*}
\mathrm{I}&\leq&\sum_{k=0}^{\infty}\sum_{\beta\in M_{k+1}}
g_1(\delta^{-km})\frac{1}{\mu(B(x,t^{1/m}))}
\int_{Q_{\beta}^{l_0}}|f(y)-A_sf(y)|\,d\mu(y)\\
&\leq&\sum_{k=0}^{\infty}\sum_{\beta\in M_{k+1}}
g_1(\delta^{-km})\frac{\|\chi_{B(z^{l_0}_{\beta},s^{1/m})}
\|_{L^{\varphi}(\mathcal{X})}}
{\mu(B(x,t^{1/m}))}\\
&&\times\frac{1}{\|\chi_{B(z^{l_0}_{\beta},s^{1/m})}
\|_{L^{\varphi}(\mathcal{X})}}
\int_{B(z^{l_0}_{\beta},s^{1/m})}|f(y)-A_sf(y)|\,d\mu(y)\\
&\leq&\sum_{k=0}^{\infty}\sum_{\beta\in M_{k+1}}
g_1(\delta^{-km})\frac{\|\chi_{B(z^{l_0}_{\beta},s^{1/m})}\|
_{L^{\varphi}(\mathcal{X})}}{\|\chi_{B(x,s^{1/m})}\|
_{L^{\varphi}(\mathcal{X})}}\frac{\|\chi_{B(x,s^{1/m})}\|
_{L^{\varphi}(\mathcal{X})}}{\mu(B(x,t^{1/m}))}
\|f\|_{\mathrm{BMO}^{\varphi}_A(\mathcal{X})}\\
&\lesssim&\sum_{k=0}^{\infty}\delta^{k(\frac{n(q-1)}{q}-n-N
-\frac{np_1}{p}-M)}\frac{\|\chi_{B(x,s^{1/m})}\|
_{L^{\varphi}(\mathcal{X})}}{\mu(B(x,t^{1/m}))}
\|f\|_{\mathrm{BMO}^{\varphi}_A(\mathcal{X})}\\
&\lesssim&\frac{\|\chi_{B(x,s^{1/m})}\|
_{L^{\varphi}(\mathcal{X})}}{\mu(B(x,t^{1/m}))}
\|f\|_{\mathrm{BMO}^{\varphi}_A(\mathcal{X})}.
\end{eqnarray*}
This, together with \eqref{AtsAB}, \eqref{bBkai} and \eqref{sthp},
implies that, when $\frac{t}{4}\leq s\leq t$,
for almost every $x\in\mathcal{X}$,
\begin{eqnarray}\label{At2AB}
|A_tf(x)-A_{t+s}f(x)|&\lesssim&\frac{\|\chi_{B(x,s^{1/m})}\|
_{L^{\varphi}(\mathcal{X})}}{\mu(B(x,t^{1/m}))}
\|f\|_{\mathrm{BMO}^{\varphi}_A(\mathcal{X})}\\ \nonumber
&\lesssim&\frac{\|\chi_{B(x,t^{1/m})}\|
_{L^{\varphi}(\mathcal{X})}}{\mu(B(x,t^{1/m}))}
\|f\|_{\mathrm{BMO}^{\varphi}_A(\mathcal{X})}.
\end{eqnarray}

For the case $0<s<t/4$, by Assumption $\mathrm{SP}$, we write
$$A_tf(x)-A_{t+s}f(x)=A_tf(x)-A_{2t}f(x)-A_{t+s}(f-A_{t-s}f)(x)$$
for almost every $x\in\mathcal{X}$.
In this case, $(t+s)/4<t-s<t+s$. By this observation, together
with \eqref{At2AB} and \eqref{Bbkai},
we conclude that, for almost every $x\in\mathcal{X}$,
\begin{eqnarray}\label{At3AB}
\hs\hs\hs|A_tf(x)-A_{t+s}f(x)|&&\lesssim\left\{\frac{\|\chi_{B(x,t^{1/m})}\|
_{L^{\varphi}(\mathcal{X})}}
{\mu(B(x,t^{1/m}))}+\frac{\|\chi_{B(x,(t+s)^{1/m})}\|
_{L^{\varphi}(\mathcal{X})}}{\mu(B(x,(t+s)^{1/m}))}
\right\}\|f\|_{\mathrm{BMO}^{\varphi}_A(\mathcal{X})}\\ \nonumber
&&\lesssim\frac{\|\chi_{B(x,t^{1/m})}\|
_{L^{\varphi}(\mathcal{X})}}
{\mu(B(x,t^{1/m}))}\|f\|_{\mathrm{BMO}^{\varphi}_A(\mathcal{X})}.
\end{eqnarray}
In general, for any $K\in(1,\infty)$, let $l$ be an integer such that
$2^l\leq K< 2^{l+1}$. If $\mathrm{diam}(\mathcal{X})=\infty$,
from \eqref{At2AB}, \eqref{At3AB}, \eqref{Bbkai} and \eqref{RDin},
we deduce that, for almost every $x\in\mathcal{X}$,
\begin{eqnarray}\label{AatAB}
&&|A_tf(x)-A_{Kt}f(x)|\\\nonumber
&&\hs\leq\sum_{k=0}^{l-1}|A_{2^kt}f(x)-A_{2^{k+1}t}f(x)|+
|A_{2^lt}f(x)-A_{Kt}f(x)|\\\nonumber
&&\hs\lesssim\sum_{k=0}^l\frac{\|\chi_{B(x,(2^kt)^{1/m})}\|
_{L^{\varphi}(\mathcal{X})}}{\mu(B(x,(2^kt)^{1/m}))}
\|f\|_{\mathrm{BMO}^{\varphi}_A(\mathcal{X})}\\\nonumber
&&\hs\lesssim\sum_{k=0}^l2^{\frac{k}{m}
(\frac{np_1}{p}-\alpha)
}\frac{\|\chi_{B(x,t^{1/m})}\|
_{L^{\varphi}(\mathcal{X})}}
{\mu(B(x,t^{1/m}))}
\|f\|_{\mathrm{BMO}^{\varphi}_A(\mathcal{X})}\\\nonumber
&&\hs\lesssim K^{\frac{1}{m}(\frac{np_1}{p}-\alpha)}
\frac{\|\chi_{B(x,t^{1/m})}\|
_{L^{\varphi}(\mathcal{X})}}
{\mu(B(x,t^{1/m}))}\|f\|_{\mathrm{BMO}^{\varphi}_A(\mathcal{X})}.
\end{eqnarray}
If $\mathrm{diam}(\mathcal{X})<\infty$, we also have the
same estimate as in \eqref{AatAB} via some minor modifications similar
to those used in the estimates for
\eqref{fKB2}, which completes the proof of Proposition \ref{p-AtB}.
\end{proof}

Applying Proposition \ref{p-AtB}, we further prove the
following size estimate for
functions in $\mathrm{BMO}^{\varphi}_A(\mathcal{X})$ at infinity.

\begin{proposition}\label{p-isB}
Let $\mathcal{X}$ be a space of homogeneous type with degree
$(\alpha_0,n_0,N_0)$,
where $\alpha_0$, $n_0$ and $N_0$ are as in
\eqref{d-ar0}, \eqref{d-n0} and \eqref{d-N0},
respectively. Assume that $\fai$ is as in Definition \ref{d-grf}
and $\{A_t\}_{t>0}$ satisfies Assumption $\mathrm{SP}$.
Let $x_0\in\mathcal{X}$ and
\begin{eqnarray}\label{d-D}
\delta >n_0+N_0+\frac{1}{m}\left[\frac{n_0p(\varphi)}
{i(\varphi)}-\alpha_0\right]
-\frac{n_0[r(\varphi)-1]}{r(\varphi)},
\end{eqnarray}
where $m,\,p(\fai)$, $i(\varphi)$ and $r(\varphi)$ are,
respectively, as in \eqref{ubfA}
\eqref{crAp}, \eqref{cult} and \eqref{crRD}. Then there exists
a positive constant $C_{(\delta)}$, depending on $\delta$, such
that, for all $f\in \mathrm{BMO}^{\varphi}_A(\mathcal{X})$,
$$\int_{\mathcal{X}}\frac{|f(x)-A_tf(x)|}{[t^{1/m}+d(x_0,x)]^{\delta}
\|\chi_{B(x_0,t^{1/m}+d(x_0,x))}
\|_{L^{\varphi}(\mathcal{X})}}\,d\mu(x)\leq \frac{C_{(\delta)}}{t^{\delta/m}}
\|f\|_{\mathrm{BMO}^{\varphi}_A(\mathcal{X})}.$$
\end{proposition}

\begin{proof}
By \eqref{d-D}, we know that
there exist $n\in[n_0,\infty)$, $N\in[N_0,\infty)$, $\alpha\in[0,\alpha_0]$,
$p_1\in[p(\fai),\fz)$, $p\in(0,i(\fai)]$ and $q\in(1,r(\fai)]$ such that
$\mathcal{X}$ satisfies
\eqref{sthp}, \eqref{eBdc} and \eqref{RDin},
respectively, for $n$, $N$ and $\alpha$,
$\fai\in\aa_{p_1}(\cx)$ and $\fai$ is of uniformly lower type $p$,
$\fai\in\mathbb{RH}_{q}(\cx)$ and $\delta >n+N+\frac{1}{m}
(\frac{np_1}{p}-\alpha)
-\frac{n(q-1)}{q}$.

Let $B:=B(x_0,t^{1/m})$ and $k\in\mathbb{N}$. From Proposition
\ref{p-AtB}, \eqref{Bbkai} and \eqref{bBkai}, we deduce that
\begin{eqnarray*}
&&\frac{1}{\|\chi_{2^{k}B}\|_{L^{\varphi}(\mathcal{X})}}\int_{2^{k}B}
|f(x)-A_tf(x)|\,d\mu(x)\\
&&\hs\leq\frac{1}
{\|\chi_{2^{k}B}\|_{L^{\varphi}(\mathcal{X})}}\int_{2^{k}B}
|f(x)-A_{t_{2^{k}B}}f(x)|\,d\mu(x)\\
&&\hs\hs+\frac{1}{\|\chi_{2^{k}B}\|_{L^{\varphi}(\mathcal{X})}}
\int_{2^{k}B}
|A_{t_{2^{k}B}}f(x)-A_tf(x)|\,d\mu(x)\\
&&\hs\lesssim\|f\|_{\mathrm{BMO}^{\varphi}
_A(\mathcal{X})}+\frac{\|f\|_{\mathrm{BMO}^{\varphi}_A(\mathcal{X})}}
{\|\chi_{2^kB}\|_{L^{\varphi}(\mathcal{X})}}
\int_{2^{k}B}
\frac{\|\chi_{B(x,t^{1/m})}\|_{L^{\varphi}(\mathcal{X})}}
{\mu(B(x,t^{1/m}))}2^{\frac{k}{m}(\frac{np_1}{p}-\alpha)}
\,d\mu(x)\\
&&\hs\lesssim\left\{1+2^{k[n+N+\frac{1}{m}
(\frac{p_1n}{p}-\alpha)
-\frac{n(q-1)}{q}]}\right\}
\|f\|_{\mathrm{BMO}^{\varphi}_A(\mathcal{X})},
\end{eqnarray*}
which, together with $\delta >n+N+\frac{1}{m}
(\frac{np_1}{p}-\alpha)
-\frac{n(q-1)}{q}$, implies that
\begin{eqnarray*}
&&\int_{\mathcal{X}}\frac{|f(x)-A_tf(x)|}
{[t^{1/m}+d(x_0,x)]^{\delta}
\|\chi_{B(x_0,t^{1/m}+d(x_0,x))}
\|_{L^{\varphi}(\mathcal{X})}}\,d\mu(x)\\
&&\hs\leq\sum^{\infty}_{k=0}\int_{2^kB\backslash2^{k-1}B}
\frac{|f(x)-A_tf(x)|}{[t^{1/m}+d(x_0,x)]^{\delta}
\|\chi_{B(x_0,t^{1/m}+d(x_0,x))}
\|_{L^{\varphi}(\mathcal{X})}}\,d\mu(x)\\
&&\hs\lesssim\sum^{\infty}_{k=0}(2^kt^{1/m})^{-\delta}
\frac{1}{\|\chi_{2^{k}B}\|_{L^{\varphi}(\mathcal{X})}}\int_{2^{k}B}
|f(x)-A_tf(x)|\,d\mu(x)\\
&&\hs\lesssim\frac{1}{t^{\delta/m}}\left\{\sum^{\infty}_{k=0}
2^{-\delta k}+\sum^{\infty}_{k=0}2^{k[n+N+\frac{1}{m}
(\frac{np_1}{p}-\alpha)
-\frac{n(q-1)}{q}-\delta]}\right\}
\|f\|_{\mathrm{BMO}^{\varphi}_A(\mathcal{X})}\lesssim\frac{1}{t^{\delta/m}}
\|f\|_{\mathrm{BMO}^{\varphi}_A(\mathcal{X})}.
\end{eqnarray*}
This finishes the proof of Proposition \ref{p-isB}.
\end{proof}

\subsection{Two characterizations of
$\mathrm{BMO}^{\varphi}_A(\mathcal{X})$}\label{s-tcB}

\hskip\parindent
For the need of the following sections, from now on,
we \emph{always assume} that the
following assumption on the generalized approximation to the identity,
$\{A_t\}_{t>0}$.

\begin{proof}[\rm\bf Assumption $\mathrm{A}$] Let $\mathcal{X}$ be
a space of homogeneous type with degree
$(\alpha_0,n_0,N_0)$,
where $\alpha_0$, $n_0$ and $N_0$ are as in
\eqref{d-ar0}, \eqref{d-n0} and \eqref{d-N0},
respectively. Assume that $\fai$ is as in Definition \ref{d-grf} and
$\{A_t\}_{t>0}$ satisfies Assumption $\mathrm{SP}$ with
$M$ in \eqref{pdcu} additionally satisfying that
\begin{eqnarray*}
M>n_0+\frac{2n_0p(\varphi)}{i(\varphi)}+N_0-
\frac{n_0[r(\varphi)-1]}{r(\varphi)}-\alpha_0,
\end{eqnarray*}
where $p(\fai)$, $i(\varphi),$ and $r(\varphi)$ are,
respectively, as in
\eqref{crAp}, \eqref{cult} and \eqref{crRD}.
\end{proof}

\begin{remark}\label{r-mco}
If $M$ is as in Assumption $\mathrm{A}$, we then conclude that $M$ also
satisfies \eqref{dM} in Proposition \ref{p-AtB}.
Indeed, if $M$ is as in Assumption $\mathrm{A}$, by
the definitions of  $n_0$, $N_0$, $\alpha_0$,
$p(\fai)$, $i(\varphi)$ and $r(\varphi)$, we know that there exist
$n\in[n_0,\infty)$, $N\in[N_0,\infty)$, $\alpha\in[0,\alpha_0]$,
$p_1\in[p(\fai),\fz)$, $p\in(0,i(\fai)]$ and $q\in(1,r(\fai)]$ such that
$\mathcal{X}$ satisfies
\eqref{sthp}, \eqref{eBdc} and \eqref{RDin},
respectively, for $n$, $N$ and $\alpha$,
$\fai\in\aa_{p_1}(\cx)$ and $\fai$ is of uniformly lower type $p$,
$\fai\in\mathbb{RH}_{q}(\cx)$ and
\begin{eqnarray}\label{ostc}
M>n+\frac{2np_1}{p}+N-\frac{n(q-1)}{q}
-\alpha,
\end{eqnarray}
which, together with $n\geq\alpha$ and $p_1\geq p$, implies the above claim.
\end{remark}

From now on, we \emph{always use} the labels $n$, $N$, $\alpha$,
$p_1$, $p$ and $q$ as in Remark \ref{r-mco}
to characterize the space of homogeneous type $\mathcal{X}$
and the growth function $\varphi$.
We first introduce the space
$\mathrm{BMO}^{\varphi}_{A,\,\mathrm{max}}(\mathcal{X})$.
\begin{definition}\label{d-Bma}
Let $\mathcal{X}$ be a space of homogeneous type,
$\fai$ as in Definition \ref{d-grf} and
$\{A_t\}_{t>0}$ a generalized approximation to the identity satisfying
\eqref{ubfA} and \eqref{pdcu}.
The \emph{space} $\mathrm{BMO}^{\varphi}_{A,\,\mathrm{max}}(\mathcal{X})$
is defined to be the set of all $f\in \mathcal{M}(\mathcal{X})$ such that
\begin{eqnarray}\label{Bmax1}
\|f\|_{\mathrm{BMO}^{\varphi}_
{A,\,\mathrm{max}}(\mathcal{X})}:=\sup_{t\in(0,\infty),
\,x\in \mathcal{X}}\frac{\mu(B(x,t^{1/m}))}
{\|\chi_{B(x,t^{1/m})}\|_{L^{\varphi}(\mathcal{X})}}
|A_t(|f-A_tf|)(x)|< \infty.
\end{eqnarray}
\end{definition}

We are ready to obtain the first characterization of
$\mathrm{BMO}^{\varphi}_A(\mathcal{X})$ with the
following extra assumption \eqref{atbB}
on the kernel $a_t$ of $A_t$.

\begin{theorem}\label{t-eBm}
Let $\mathcal{X}$ be a space of homogeneous type with degree
$(\alpha_0,n_0,N_0)$,
where $\alpha_0$, $n_0$ and $N_0$ are as in
\eqref{d-ar0}, \eqref{d-n0} and \eqref{d-N0},
respectively. Assume that $\fai$ is as in Definition \ref{d-grf},
$\{A_t\}_{t>0}$ satisfies Assumption $\mathrm{A}$ and,
for any $t\in(0,\infty)$, the kernel
$a_t$ of the operator $A_t$ is a nonnegative function
satisfying the following lower bound:
for all $t\in(0,\infty)$, $x\in\mathcal{X}$ and $y\in B(x,t^{1/m})$,
\begin{eqnarray}\label{atbB}
a_t(x,y)\geq\frac{C}{\mu(B(x,t^{1/m}))},
\end{eqnarray}
where $C$
is a positive constant independent of $t$, $x$ and $y$. Then the spaces
$\mathrm{BMO}^{\varphi}_{A,\,\mathrm{max}}(\mathcal{X})$
and $\mathrm{BMO}^{\varphi}_A(\mathcal{X})$ coincide with equivalent norms.
\end{theorem}

\begin{proof}
We first prove $\mathrm{BMO}^{\varphi}_A(\mathcal{X})\subset
\mathrm{BMO}^{\varphi}_{A,\,\mathrm{max}}(\mathcal{X})$.
For any fixed $x\in\mathcal{X}$ and $t\in(0,\infty)$, let $B:=B(x,t^{1/m})$.
Let $f\in\mathrm{BMO}^{\varphi}_A(\mathcal{X})$.
Since $\{A_t\}_{t>0}$ satisfies Assumption $\mathrm{A}$, we then choose $n$, $N$,
$\alpha$, $p_1$, $p$ and $q$ as in Remark \ref{r-mco} such that
\eqref{ostc} holds true. From \eqref{ostc}, we further deduce that
$M>\frac{np_1}{p}$. By this and \eqref{ostc}, together with
Proposition \ref{p-AtB}, Lemma \ref{l-kai},
and the decreasing property of $g$,
we see that
\begin{eqnarray*}
&&|A_t(|f-A_tf|)(x)|\\
&&\hs\leq\int_{\mathcal{X}}|a_t(x,y)||f(y)-A_tf(y)|\,d\mu(y)\\
&&\hs\leq\sum_{k=0}
^{\infty}\frac{1}{\mu(B)}
\int_{2^kB\setminus2^{k-1}B}g\left(\frac{[d(x,y)]^m}{t}
\right)|f(y)-A_{t}f(y)|\,d\mu(y)\\
&&\hs\lesssim\sum_{k=0}^{\infty}
g(2^{(k-1)m})\frac{1}{\mu(B)}\left[\int_{2^kB}
\left|f(y)-A_{t_{2^kB}}f(y)\right|\,d\mu(y)\right.\\
&&\left.\hs\hs+\int_{2^kB}
\left|A_{t_{2^kB}}f(y)-A_{t}f(y)\right|\,d\mu(y)
\right]\\
&&\hs\lesssim\sum_{k=0}^{\infty}
g(2^{(k-1)m})\frac{1}{\mu(B)}\Bigg[
\|\chi_{2^kB}\|_{L^{\varphi}(\mathcal{X})}\\
&&\hs\hs\left.+2^{k(\frac{np_1}{p}-\alpha)}\int_{2^kB}
\frac{\|\chi_{B(y,t^{1/m})}\|
_{L^{\varphi}(\mathcal{X})}}{\mu(B(y,t^{1/m}))}
\,d\mu(y)\right]\|f\|_{\mathrm{BMO}^{\varphi}_A(\mathcal{X})}\\
&&\hs\lesssim\sum_{k=0}^{\infty}
\left[2^{k(\frac{np_1}{p}-M)
}+2^{k(n+\frac{2np_1}{p}
+N-\frac{n(q-1)}{q}-\alpha-M)}\right]
\frac{\|\chi_B\|_{L^{\varphi}(\mathcal{X})}}
{\mu(B)}\|f\|_{\mathrm{BMO}^{\varphi}_A(\mathcal{X})}\\
&&\hs\lesssim\frac{\|\chi_B\|_{L^{\varphi}(\mathcal{X})}}
{\mu(B)}\|f\|_{\mathrm{BMO}^{\varphi}_A(\mathcal{X})},
\end{eqnarray*}
which, together with the arbitrariness of $x\in\mathcal{X}$ and $t\in(0,\infty)$,
implies that $f\in\mathrm{BMO}^{\varphi}_{A,\,\mathrm{max}}(\mathcal{X})$
and $\|f\|_{\mathrm{BMO}^{\varphi}_{A,\,\mathrm{max}}(\mathcal{X})}\lesssim
\|f\|_{\mathrm{BMO}^{\varphi}_A(\mathcal{X})}$.

Conversely, let $f\in\mathrm{BMO}^{\varphi}_{A,\,\mathrm{max}}(\mathcal{X})$
and
$B:=B(x,r_B)$ with $x\in\mathcal{X}$ and $r_B\in(0,\infty)$.
Let $t_B:=r_B^m$. Then, from the assumption \eqref{atbB},
it follows that
\begin{eqnarray*}
&&\frac{1}
{\|\chi_B\|_{L^{\varphi}(\mathcal{X})}}
\int_{B}|f(y)-A_{t_B}f(y)|\,d\mu(y)\\
&&\hs\lesssim\frac{\mu(B(x,t_B^{1/m}))}
{\|\chi_{B(x,t_B^{1/m})}\|_{L^{\varphi}(\mathcal{X})}}
\int_{B(x,r_B)}a_{t_B}(x,y)|f(y)-A_{t_B}f(y)|\,d\mu(y)
\lesssim\|f\|_{\mathrm{BMO}
^{\varphi}_{A,\,\mathrm{max}}(\mathcal{X})},
\end{eqnarray*}
which implies that $f\in\mathrm{BMO}^{\varphi}_A(\mathcal{X})$
and $\|f\|_{\mathrm{BMO}^{\varphi}_A(\mathcal{X})}\lesssim
\|f\|_{\mathrm{BMO}^{\varphi}_
{A,\,\mathrm{max}}(\mathcal{X})}$. This finishes the proof of Theorem \ref{t-eBm}.
\end{proof}
\begin{remark}\label{r-aat}
It was pointed out by Duong and Yan \cite{dy} that examples of
$a_t(x,y)$ satisfy the condition \eqref{atbB} include heat kernels
of uniformly divergence form elliptic operators with bounded,
real symmetric coefficients on $\mathbb{R}^n$, and the Laplace-Beltrami
operator on a complete Riemannian manifold $M$ with nonnegative
Ricci curvature (see also \cite[Theorems 3.3.4 and 5.6.1]{dav}).
\end{remark}

Next, we give another equivalent characterization of
$\mathrm{BMO}^{\varphi}_A(\mathcal{X})$. In other words,
the average value $A_{t_B}f$ in Definition \ref{d-BfA}
can be changed into other value $f^B$ which satisfies appropriate estimates.
\begin{definition}\label{d-wB}
Let $\mathcal{X}$ be a space of homogeneous type with degree
$(\alpha_0,n_0,N_0)$,
where $\alpha_0$, $n_0$ and $N_0$ are as in
\eqref{d-ar0}, \eqref{d-n0} and \eqref{d-N0},
respectively. Assume that $\fai$ is as in Definition \ref{d-grf}
and $\{A_t\}_{t>0}$ a generalized approximation to the identity satisfying
\eqref{ubfA} and \eqref{pdcu}.
If, for a given function $f\in \mathcal{M}(\mathcal{X})$, there exist
a positive constant $C$ and a collection of functions,
$\{f^B\}_B$ (in other words, for each ball $B$,
there exists a function $f^B$), such that
\begin{eqnarray}\label{wB1}
\sup_{B\subset \mathcal{X}}\frac{1}{\|\chi_B\|_{L^{\varphi}(\mathcal{X})}}
\int_{B}|f(x)-f^B(x)|\,d\mu(x)\leq C,
\end{eqnarray}
\begin{eqnarray}\label{wB2}
|f^{B_2}(x)-f^{B_1}(x)|\leq C\frac{\|\chi_{B(x,r_{B_1})}\|
_{L^{\varphi}(\mathcal{X})}}{\mu(B(x,r_{B_1}))}
\left(\frac{r_{B_2}}{r_{B_1}}\right)^{\frac{np_1}{p}-\alpha}
\end{eqnarray}
for any two balls $B_1\subset B_2$, and
\begin{eqnarray}\label{wB3}
|f^B(x)-A_{t_B}f^B(x)|\leq C\frac{\|\chi_{B(x,r_B)}\|_{L^{\varphi}(\mathcal{X})}}
{\mu(B(x,r_B))}
\end{eqnarray}
for almost every $x\in \mathcal{X}$, where $t_B=r_B^m$,
then it is said that $f$ belongs to the \emph{space}
$\widetilde{\mathrm{BMO}}^{\varphi}_A(\mathcal{X})$ and
$$\|f\|_{\widetilde{\mathrm{BMO}}^
{\varphi}_A(\mathcal{X})}:=\inf\{C:\ C\ \hbox{satisfies}\
\eqref{wB1},\ \eqref{wB2}\ \hbox{and}\ \eqref{wB3}\},$$
where the infimum is taken over all the constants $C$ as above
and all the functions $\{f^B\}_B$ satisfying \eqref{wB1},
\eqref{wB2} and \eqref{wB3}.
\end{definition}
We have the following equivalence between the spaces
$\mathrm{BMO}^{\varphi}_A(\mathcal{X})$
and $\widetilde{\mathrm{BMO}}^{\varphi}_A(\mathcal{X})$.
\begin{theorem}\label{t-ewB}
Let $\mathcal{X}$ be a space of homogeneous type with degree
$(\alpha_0,n_0,N_0)$,
where $\alpha_0$, $n_0$ and $N_0$ are as in
\eqref{d-ar0}, \eqref{d-n0} and \eqref{d-N0},
respectively. Assume that $\fai$ is as in Definition \ref{d-grf}
and $\{A_t\}_{t>0}$ satisfies Assumption $\mathrm{A}$.
The spaces $\mathrm{BMO}^{\varphi}_A(\mathcal{X})$ and
$\widetilde{\mathrm{BMO}}^{\varphi}_A(\mathcal{X})$
coincide with equivalent norms.
\end{theorem}
\begin{proof}
Let $f\in \mathcal{M}(\mathcal{X})$. It is easy to prove that
$\mathrm{BMO}^{\varphi}_A(\mathcal{X})
\subset\widetilde{\mathrm{BMO}}^{\varphi}_A(\mathcal{X})$
and, for all $f\in\mathrm{BMO}^{\varphi}_A(\mathcal{X})$,
$\|f\|_{\widetilde{\mathrm{BMO}}^{\varphi}_A(\mathcal{X})}
\lesssim \|f\|_{\mathrm{BMO}^{\varphi}_A(\mathcal{X})}$.
Indeed, let $f^B(x):=A_{t_B}f(x)$ for each ball $B$ and
$x\in\mathcal{X}$. Then, by Proposition \ref{p-AtB},
the estimates \eqref{wB1}, \eqref{wB2} and \eqref{wB3}
hold true with $C$ replaced by
$\widetilde{C}\|f\|_{\mathrm{BMO}^{\varphi}_A(\mathcal{X})}$,
where $\widetilde{C}$ is a positive constant independent of $f$.

Conversely, we need to prove that
$\widetilde{\mathrm{BMO}}^{\varphi}_A(\mathcal{X})
\subset \mathrm{BMO}^{\varphi}_A(\mathcal{X})$
and, for all $f\in\widetilde{\mathrm{BMO}}^{\varphi}_A(\mathcal{X})$,
$\|f\|_{\mathrm{BMO}^{\varphi}_A(\mathcal{X})}\lesssim
\|f\|_{\widetilde{\mathrm{BMO}}^{\varphi}_A(\mathcal{X})}$.
To this end, for any $f\in\widetilde{\mathrm{BMO}}^{\varphi}_A(\mathcal{X})$
and fixed ball $B_0:=B(x_0,r_{B_0})$ with $x_0\in\mathcal{X}$
and $r_{B_0}\in(0,\infty)$, it suffices to prove that
$$\frac{1}{\|\chi_{B_0}\|_{L^{\varphi}(\mathcal{X})}}\int_{B_0}
|f(x)-A_{t_{B_0}}f(x)|\,d\mu(x)\lesssim
\|f\|_{\widetilde{\mathrm{BMO}}^{\varphi}_A(\mathcal{X})},$$
where $t_{B_0}:=r_{B_0}^m$. For any $x\in B_0$,
by \eqref{eBdc}, we see that
$\mu(B_0)\lesssim\mu(B(x,t^{1/m}_{B_0}))$. Notice that
$M$ satisfies \eqref{ostc} by Assumption $\mathrm{A}$ together with
Remark \ref{r-mco}. Thus, $m>\frac{np_1}{p}$. From this,
\eqref{ostc}, \eqref{wB1} and \eqref{wB2}, together with
the decreasing property of $g$,
we deduce that
\begin{eqnarray*}
&&|A_{t_{B_0}}(f-f^{B_0})(x)|\\
&&\hs\leq\frac{1}{\mu(B(x,t^{1/m}_{B_0}))}
\int_{\mathcal{X}}g\left(\frac{[d(x,y)]^m}{t_{B_0}}
\right)|f(y)-f^{B_0}(y)|\,d\mu(y)\\
&&\hs\lesssim\sum_{k=0}
^{\infty}\frac{1}{\mu(B_0)}
\int_{2^kB_0\setminus2^{k-1}B_0}g\left(\frac{[d(x,y)]^m}{t_{B_0}}
\right)|f(y)-f^{B_0}(y)|\,d\mu(y)\\
&&\hs\lesssim\sum_{k=0}^{\infty}
g(2^{(k-2)m})\frac{1}{\mu(B_0)}\Big[\int_{2^kB_0}
\left|f(y)-f^{2^kB_0}(y)\right|\,d\mu(y)\\
&&\hs\hs+\int_{2^kB_0}
\left|f^{2^kB_0}(y)-f^{B_0}(y)\right|\,d\mu(y)\Big]\\
&&\hs\lesssim\sum_{k=0}^{\infty}
g(2^{(k-2)m})\frac{1}{\mu(B_0)}\Big[
\|\chi_{2^kB_0}\|_{L^{\varphi}(\mathcal{X})}
\|f\|_{\widetilde{\mathrm{BMO}}^{\varphi}_A(\mathcal{X})}\\
&&\hs\hs+2^{k(\frac{np_1}{p}-\alpha)}\int_{2^kB_0}
\frac{\|\chi_{B(y,r_{B_0})}\|
_{L^{\varphi}(\mathcal{X})}}{\mu(B(y,r_{B_0}))}
\|f\|_{\widetilde{\mathrm{BMO}}^{\varphi}_A(\mathcal{X})}\,d\mu(y)\Big]\\
&&\hs\lesssim\sum_{k=0}^{\infty}
\left[2^{k(\frac{np_1}{p}-M)
}+2^{k(n+\frac{2np_1}{p}
+N-\frac{n(q-1)}{q}-\alpha-M)}\right]
\frac{\|\chi_{B_0}\|_{L^{\varphi}(\mathcal{X})}}
{\mu(B_0)}\|f\|_{\mathrm{BMO}^{\varphi}_A(\mathcal{X})}\\
&&\hs\lesssim\frac{\|\chi_{B_0}\|_{L^{\varphi}(\mathcal{X})}}
{\mu(B_0)}\|f\|_{\widetilde{\mathrm{BMO}}^{\varphi}_A(\mathcal{X})},
\end{eqnarray*}
which, together with \eqref{wB1}, \eqref{wB3} and \eqref{Bbkai},
implies that
\begin{eqnarray*}
&&\frac{1}{\|\chi_{B_0}\|_{L^{\varphi}(\mathcal{X})}}\int_{B_0}
|f(x)-A_{t_{B_0}}f(x)|\,d\mu(x)\\
&&\hs\leq\frac{1}{\|\chi_{B_0}\|_{L^{\varphi}(\mathcal{X})}}\int_{B_0}
\left[|f(x)-f^{B_0}(x)|+|f^{B_0}(x)-A_{t_{B_0}}f^{B_0}(x)|\right.\\
&&\hs\hs\left.+|A_{t_{B_0}}(f-f^{B_0})(x)|\right]\,d\mu(x)\\
&&\hs\lesssim \|f\|_{\widetilde{\mathrm{BMO}}^{\varphi}_A(\mathcal{X})}
\left[1+
\frac{1}{\|\chi_{B_0}\|_{L^{\varphi}(\mathcal{X})}}\int_{B_0}
\frac{\|\chi_{B(x,r_{B_0})}\|
_{L^{\varphi}(\mathcal{X})}}{\mu(B(x,r_{B_0}))}
\,d\mu(x)\right]
\lesssim\|f\|_{\widetilde{\mathrm{BMO}}^{\varphi}_A(\mathcal{X})}.
\end{eqnarray*}
By this, combined with the arbitrariness of $B_0\subset\mathcal{X}$,
we then conclude that $f\in \mathrm{BMO}^{\varphi}_A(\mathcal{X})$ and
$\|f\|_{\mathrm{BMO}^{\varphi}_A(\mathcal{X})}\lesssim
\|f\|_{\widetilde{\mathrm{BMO}}^{\varphi}_A(\mathcal{X})},$
which completes the proof of Theorem \ref{t-ewB}.
\end{proof}

\begin{remark}\label{r4.1}
Theorems \ref{t-eBm} and \ref{t-ewB} completely cover, respectively,
\cite[Propositions 2.10 and 2.12]{dy} by taking $\fai$ as in
\eqref{fat}. Moreover, Theorem \ref{t-ewB} completely covers
\cite[Proposition 2.4]{tang} by taking $\fai$ as in \eqref{fatb}.
\end{remark}

\section{Two variants of the John-Nirenberg
inequality\\ on $\mathrm{BMO}^{\varphi}_A(\mathcal{X})$}\label{s-tuj}

\hskip\parindent
In this section, we establish two variants of the
John-Nirenberg inequality on $\mathrm{BMO}^{\varphi}_A(\mathcal{X})$.
We then discuss the relationship
between these two John-Nirenberg inequalities in Remark \ref{r-j} when
$\varphi\in\mathbb{A}_1(\mathcal{X})$. Moreover, we also introduce
the Musielak-Orlicz BMO-type spaces
$\mathrm{BMO}^{\varphi,\,\widetilde{p}}_A(\mathcal{X})$ and
$\widetilde{\mathrm{BMO}}^{\varphi,\,\widetilde{p}}_A(\mathcal{X})$
with $\widetilde{p}\in[1,\infty)$. As an application of these
John-Nirenberg inequalities on $\mathrm{BMO}^{\varphi}_A(\mathcal{X})$,
we further prove that, for any $\widetilde{p}\in[1,\fz)$, the spaces
$\mathrm{BMO}^{\varphi,\,\widetilde{p}}_A(\mathcal{X})$,
$\widetilde{\mathrm{BMO}}^{\varphi,\,\widetilde{p}}_A(\mathcal{X})$
and $\mathrm{BMO}^{\varphi}_A(\mathcal{X})$
coincide with  equivalent norms.

\subsection{The first variant of the John-Nirenberg inequality on
$\mathrm{BMO}^{\varphi}_A(\mathcal{X})$}\label{s-uj}

\hskip\parindent
In order to establish the first variant of
the John-Nirenberg inequality on the space
$\mathrm{BMO}^{\varphi}_A(\mathcal{X})$,
we assume that the growth function $\varphi$ satisfies
the following property: there exists a positive constant $C$
such that, for all balls
$B_1,B_2\subset \mathcal{X}$ with $B_1\subset B_2$,
\begin{eqnarray}\label{afnj}
\frac{\|\chi_{B_1}\|_{L^{\varphi}(\mathcal{X})}}{\mu(B_1)}\leq C
\frac{\|\chi_{B_2}\|_{L^{\varphi}(\mathcal{X})}}{\mu(B_2)}.
\end{eqnarray}

\begin{remark}\label{r-anj}
Let $\mathcal{X}$ be a space of homogeneous type with degree
$(\alpha_0,n_0,N_0)$,
where $\alpha_0$, $n_0$ and $N_0$ are as in
\eqref{d-ar0}, \eqref{d-n0} and \eqref{d-N0},
respectively. There exist following nontrivial
growth functions $\varphi$ satisfying \eqref{afnj}.

(i) Let $\varphi(x,t):=t^{\widetilde{p}}$ for all
$x\in\mathcal{X}$ and $t\in[0,\infty)$,
where $\widetilde{p}\in(0,1]$. Then, 
$\varphi$ satisfies \eqref{afnj}.

Indeed, it is easy to see that, for any ball $B\subset\mathcal{X}$,
$\|\chi_B\|_{L^{\varphi}(\mathcal{X})}=[\mu(B)]^{\frac{1}{\widetilde{p}}}$.
Then, we know that, for all balls
$B_1, B_2\subset \mathcal{X}$ with $B_1\subset B_2$,
$$\frac{\|\chi_{B_1}\|_{L^{\varphi}(\mathcal{X})}}
{\|\chi_{B_2}\|_{L^{\varphi}(\mathcal{X})}}=
\frac{[\mu(B_1)]^{\frac{1}{\widetilde{p}}}}
{[\mu(B_2)]^{\frac{1}{\widetilde{p}}}}
\leq\frac{\mu(B_1)}{\mu(B_2)},$$
which implies that \eqref{afnj} holds true.

(ii) Let $\varphi(x,t):=\omega(x)t^{\widetilde{p}}$ for all
$x\in\mathcal{X}$ and $t\in[0,\infty)$,
where $\widetilde{p}\in(0,1]$, $\omega\in A_{\infty}(\mathcal{X})
\cap\mathrm{RH}_{\widetilde{q}}(\mathcal{X})$
with $\widetilde{q}\in(1,\infty]$. If
$\widetilde{p}\in(0,\frac{\widetilde{q}-1}{\widetilde{q}}]$, then
$\varphi$ satisfies \eqref{afnj}.

Indeed, it is easy to see that, for any ball $B\subset\mathcal{X}$,
$\|\chi_B\|_{L^{\varphi}(\mathcal{X})}=[\omega(B)]^{\frac{1}{\widetilde{p}}}$.
Then, by \eqref{bBkai}, we know that, for all balls
$B_1, B_2\subset \mathcal{X}$ with $B_1\subset B_2$,
$$\frac{\|\chi_{B_1}\|_{L^{\varphi}(\mathcal{X})}}
{\|\chi_{B_2}\|_{L^{\varphi}(\mathcal{X})}}=
\frac{[\omega(B_1)]^{\frac{1}{\widetilde{p}}}}
{[\omega(B_2)]^{\frac{1}{\widetilde{p}}}}\lesssim
\left[\frac{\mu(B_1)}{\mu(B_2)}\right]
^{\frac{\widetilde{q}-1}{\widetilde{p}\widetilde{q}}}
\lesssim\frac{\mu(B_1)}{\mu(B_2)},$$
which also implies that \eqref{afnj} holds true.

(iii) Let $x_0\in\mathcal{X}$ and
$\varphi(x,t):=\frac{t^{\widetilde{p}}}
{[\ln(e+d(x_0,x))+\ln(e+t^{\widetilde{p}})]^{\widetilde{p}}}$,
with $\widetilde{p}\in(0,1]$, for all
$x\in\mathcal{X}$ and $t\in[0,\infty)$. If $\alpha_0\in(0,\infty)$ and
$\widetilde{p}\in(0,\frac{\alpha_0}{\alpha_0+1})$, then
$\varphi$ satisfies \eqref{afnj}.

Indeed, $\mathcal{X}$ satisfies \eqref{eBdc} and \eqref{RDin}, respectively,
for any $\widetilde{N}\in(N_0,\infty)$ and $\widetilde{\alpha}\in(0,\alpha_0)$.
It is easy to see that $\varphi \in \mathbb{A}_1(\mathcal{X})$ and
$$\|\chi_{B}\|_{L^{\varphi}(\mathcal{X})}
\thicksim\frac{[\mu(B)]^{\frac{1}{\widetilde{p}}}}
{\ln(e+[\mu(B)]^{-1})+\sup_{x\in B}\ln(e+d(x_0,x))}$$
for any ball $B\subset\mathcal{X}$.
Then, for any
$\widetilde{p}\in(0,\frac{\widetilde{\alpha}}
{\widetilde{\alpha}+1})$,
$B_1:=B(x_1,r_1)$ with $x_1\in\mathcal{X}$ and $r_1\in(0,\infty)$,
$B_2:=B(x_2,r_2)$ with $x_2\in\mathcal{X}$ and $r_2\in(0,\infty)$,
and $B_1\subset B_2$, it holds true that
\begin{eqnarray*}
\frac{\mu(B_2)\|\chi_{B_1}\|_{L^{\varphi}(\mathcal{X})}}
{\mu(B_1)\|\chi_{B_2}\|_{L^{\varphi}(\mathcal{X})}}&\thicksim&
\left[\frac{\mu(B_1)}{\mu(B_2)}\right]^{\frac{1}{\widetilde{p}}-1}
\frac{\ln(e+[\mu(B_2)]^{-1})
+\sup_{x\in B_2}\ln(e+d(x_0,x))}{\ln(e+[\mu(B_1)]^{-1})
+\sup_{x\in B_1}\ln(e+d(x_0,x))}\\
&\lesssim&\left(\frac{r_1}{r_2}\right)
^{\widetilde{\alpha}(\frac{1}{\widetilde{p}}-1)}
\left[1+\frac{\ln(e+C_1[d(x_0,x_2)+r_2])}
{\ln(e+d(x_0,x_1))}\right]
\lesssim\left(\frac{r_1}{r_2}\right)
^{\widetilde{\alpha}(\frac{1}{\widetilde{p}}-1)-1}
\lesssim 1,
\end{eqnarray*}
which implies that \eqref{afnj} holds true.

Moreover, we point out that, when $\mathcal{X}:=\rn$ and
$\varphi(x,t)=\frac{t^{\widetilde{p}}}
{[\ln(e+|x|)+\ln(e+t^{\widetilde{p}})]^{\widetilde{p}}}$,
with $\widetilde{p}\in(0,1]$, for all
$x\in\mathbb{R}^n$ and $t\in[0,\infty)$, the
Musielak-Orlicz Hardy space $H^{\varphi}(\mathbb{R}^n)$
related to $\fai$ arises naturally in the study of pointwise products of
functions in Hardy spaces $H^{\widetilde{p}}(\mathbb{R}^n)$ with functions
in $\mathrm{BMO}(\mathbb{R}^n)$
(see \cite{bg10} in the setting of holomorphic functions
in convex domains of finite type or strictly pseudoconvex
domains in $\mathbb{C}^n$), where the space $H^{\varphi}(\mathbb{R}^n)$
is introduced by Ky \cite{ky}, and its dual space is
$\mathrm{BMO}^{\varphi}(\rn)$. In this case, \eqref{afnj} holds true for
$\varphi$ with $\widetilde{p}\in(0,1)$,
which needs more precise estimates, the details being omitted.
However, when $\widetilde{p}=1$, \eqref{afnj} does not hold true. Indeed,
we choose $B_1:=B(0,1)$ and $B_2:=B(0,r)$ for $r\in(1,\infty)$. Letting
$r\to\infty$, we then find that
\begin{eqnarray*}
\frac{|B_2|\|\chi_{B_1}\|_{L^{\varphi}(\mathcal{X})}}
{|B_1|\|\chi_{B_2}\|_{L^{\varphi}(\mathcal{X})}}&\thicksim&
\frac{\ln(e+|B_2|^{-1})
+\sup_{x\in B_2}\ln(e+|x|)}{\ln(e+|B_1|^{-1})
+\sup_{x\in B_1}\ln(e+|x|)}\\
&\thicksim&\frac{\ln(e+r^{-n}|B_1|^{-1})+\ln(e+r)}
{\ln(e+|B_1|^{-1})+\ln(e+1)}\rightarrow\infty,
\end{eqnarray*}
which implies that \eqref{afnj} does not hold true.
\end{remark}

Now we establish the first variant of the John-Nirenberg inequality on
$\mathrm{BMO}^{\varphi}_A(\mathcal{X})$ by using Proposition
\ref{dBMf} and borrowing some ideas from Duong and Yan \cite{dy}.
Recall that the \emph{Hardy-Littlewood maximal operator} $M$
is defined by setting, for all $f\in L^1_\loc(\cx)$ and $x\in\mathcal{X}$,
$$M(f)(x):=\sup_{x\in B}\frac{1}{\mu(B)}\int_{B}|f(y)|\,d\mu(y),$$
where the supremum is taken over all balls in $\cx$ containing $x$.

\begin{theorem}\label{t-uj}
Let $\mathcal{X}$ be a space of homogeneous type with degree
$(\alpha_0,n_0,N_0)$,
where $\alpha_0$, $n_0$ and $N_0$ are as in
\eqref{d-ar0}, \eqref{d-n0} and \eqref{d-N0},
respectively.
Assume that $\varphi$ is as in Definition \ref{d-grf} and satisfies
\eqref{afnj}. Let $\{A_t\}_{t>0}$ satisfy Assumption $\mathrm{A}$.
Then, there exist positive constants $c_1$ and $c_2$
such that, for all $f\in \mathrm{BMO}^{\varphi}_A(\mathcal{X})$,
balls $B$ and $\lz\in(0,\infty)$,
\begin{eqnarray}\label{uj1}
\ \ \ \ \mu(\{x\in B:\ |f(x)-A_{t_B}f(x)|>\lambda\})\leq c_1\mu(B)
\exp\left\{-\frac{c_2\lambda \mu(B)}{\|\chi_{B}\|_{L^{\varphi}(\mathcal{X})}
\|f\|_{\mathrm{BMO}^{\varphi}_A(\mathcal{X})}}\right\},
\end{eqnarray}
where $t_B:=r_B^m$ and $m$ is as in \eqref{ubfA}.
\end{theorem}
\begin{proof}
Since $\{A_t\}_{t>0}$ satisfies Assumption $\mathrm{A}$, we then choose $n$, $N$,
$\alpha$, $p_1$, $p$ and $q$ as in Remark \ref{r-mco} such that
\eqref{ostc} holds true.

Let $B:=B(x_B,r_B)$, with $x_B\in\mathcal{X}$
and $r_B\in(0,\infty)$, and $f\in \mathrm{BMO}^{\varphi}_A(\mathcal{X})$.
In order to prove \eqref{uj1}, it suffices to consider the case
$\|f\|_{\mathrm{BMO}^{\varphi}_A(\mathcal{X})}>0$. Otherwise,
\eqref{uj1} holds true obviously. Without loss of generality,
we may assume that
\begin{equation}\label{huj}
\frac{\mu(B)}{\|\chi_{B}\|_{L^{\varphi}(\mathcal{X})}
\|f\|_{\mathrm{BMO}^{\varphi}_A(\mathcal{X})}}=1.
\end{equation}
Otherwise, we replace $f$ by $\frac{\mu(B)f}{\|\chi_{B}
\|_{L^{\varphi}(\mathcal{X})}
\|f\|_{\mathrm{BMO}^{\varphi}_A(\mathcal{X})}}$.
Thus, we only need to prove that there exist positive
constants $c_1$ and $c_2$ such that, for any $\lambda\in(0,\infty)$,
\begin{equation}\label{suj}
\mu(\{x\in B:\ |f(x)-A_{t_B}f(x)|>\lambda\})\leq c_1 e^{-c_2 \lambda}\mu(B),
\end{equation}
where $t_B:=r_B^m$.

It is obvious that, when $\lambda\in(0,1)$, \eqref{suj}
holds true for $c_1:=e$ and $c_2:=1$.

Let $\lambda\in[1,\infty)$ and $f_0:=(f-A_{t_B}f)\chi_{10C_1^4B}$,
 where $C_1$ is as in \eqref{qume}. By Proposition \ref{p-AtB},
 \eqref{sthp}, \eqref{eBdc}, \eqref{Bbkai}, \eqref{bBkai}
 and \eqref{huj}, we know that
\begin{eqnarray}\label{efuj}
\|f_0\|_{L^1(\mathcal{X})}&=& \int_{10C_1^4B}|f(x)-A_{t_B}f(x)|\,d\mu(x)\\\nonumber
&&\leq\int_{10C_1^4B}\left|f(x)-A_{t_{10C_1^4B}}f(x)\right|
\,d\mu(x)\\\nonumber
&&\hs+\int_{10C_1^4B}\left|A_{t_{10C_1^4B}}f(x)-A_{t_B}f(x)\right|
\,d\mu(x)\\\nonumber
&&\lesssim
\|f\|_{\mathrm{BMO}^{\varphi}_A(\mathcal{X})}\left[
\left\|\chi_{10C_1^4B}\right\|_{L^{\varphi}(\mathcal{X})}
+\int_{10C_1^4B}\frac{\|\chi_{B(x,r_B)}\|
_{L^{\varphi}(\mathcal{X})}}{\mu(B(x,r_B))}
\,d\mu(x)\right]\\ \nonumber
&&\lesssim\|\chi_{B}\|_{L^{\varphi}
(\mathcal{X})}\|f\|_{\mathrm{BMO}^{\varphi}_A(\mathcal{X})}
\thicksim\mu(B).
\end{eqnarray}

Let $\beta\in(1,\infty)$ be determined later,
$$F:=\{x\in\mathcal{X}:\ M(f_0)(x)\leq \beta\}\ \ \  \hbox{and} \ \ \
\Omega:=F^\complement=\{x\in\mathcal{X}:\ M(f_0)(x)> \beta\}.$$
From \cite[Chapter III, Theorem 1.3]{cw2}, we deduce that there exists a
collection of balls, $\{B_{1,i}\}_{i\in\mathbb{N}}$, satisfying that

\ \ (i) $\cup_iB_{1,i}=\Omega$;

\ (ii) each point of $\Omega$ is contained in at most a finite
number $L$ of the balls $B_{1,i}$;

(iii) there exists $\widetilde{C}\in(1,\infty)$ such that
$\widetilde{C}B_{1,i}\cap
F\neq \emptyset$ for each $i$.

By (i), we see that, for any $x \in B\backslash(\cup_iB_{1,i}),$
$$|f(x)-A_{t_B}f(x)|=|f_0(x)|\chi_F(x)\leq
M(f_0)(x)\chi_F(x)\leq\beta.$$
From the fact that $M$ is of weak type (1,1), (i), (ii) and \eqref{efuj},
it follows that there exists a positive constant $c_3$ such that
\begin{eqnarray}\label{eBuj1}
\sum_i\mu(B_{1,i})\leq L\mu(\Omega)\lesssim \frac{1}{\beta}
\|f_0\|_{{L^1(\mathcal{X})}}\leq \frac{c_3}{\beta}\mu(B).
\end{eqnarray}

For any $B_{1,i}\cap B\neq\emptyset$, we denote by
$B_{1,i}:=B(x_{B_{1,i}},r_{B_{1,i}})$ the ball centered at
$x_{B_{1,i}}$ and of radius $r_{B_{1,i}}$. Notice that
$d(x_B, x_{B_{1,i}})\leq C_1(r_B+r_{B_{1,i}})$.
If $r_B<r_{B_{1,i}}$, by \eqref{eBuj1}, we know that there exists
a positive constant
$\widetilde{c}_3$ such that
\begin{eqnarray}\label{eBuj2}
\mu(B(x_B,r_{B_{1,i}}))\lesssim
\left[1+\frac{r_B+r_{B_{1,i}}}{r_{B_{1,i}}}\right]^N
\mu(B(x_{B_{1,i}},r_{B_{1,i}}))\leq
\frac{\widetilde{c}_3\mu(B)}{\beta}.
\end{eqnarray}

We now choose $\beta>\widetilde{c}_3$ and,
therefore, $r_B\geq r_{B_{1,i}}$. Otherwise,
if $r_B<r_{B_{1,i}}$, then $\mu(B)\leq \mu(B(x_B,r_{B_{1,i}}))$,
which contradicts
to \eqref{eBuj2}.

By $r_B\geq r_{B_{1,i}}$, together with \eqref{eBdc}, \eqref{sthp}
and $d(x_B, x_{B_{1,i}})\leq C_1(r_B+r_{B_{1,i}})$,
we find that, for some positive constant $c_4$,
\begin{eqnarray}\label{eBuj3}
\mu(B)&\lesssim&\left(\frac{r_B}{r_{B_{1,i}}}
\right)^n\mu(B(x_B,r_{B_{1,i}}))\\\nonumber
&\lesssim&\left(\frac{r_B}{r_{B_{1,i}}}
\right)^n\left[1+\frac{C_1(r_B+r_{B_{1,i}})}{r_{B_{1,i}}}
\right]^N\mu(B(x_{B_{1,i}},r_{B_{1,i}}))\\\nonumber
&\leq&\frac{c_4}{\beta}\left(\frac{r_B}{r_{B_{1,i}}}
\right)^{n+N}\mu(B).
\end{eqnarray}

We further choose
$\beta>\max\{\widetilde{c}_3,
c_4(10C_1)^{n+N}, c_3e\}$.
Then, from \eqref{eBuj3}
and the fact $r_B>10C_1r_{B_{1,i}}$ together
with \eqref{qume}, we deduce that, for any
$B_{1,i}\cap B\neq\emptyset$, $B_{1,i}\subset 2C_1B$.

We claim that there exists a positive constant $c_5$ such
that, for any $B_{1,i}\cap B\neq\emptyset$ and almost every $x\in B_{1,i}$,
\begin{eqnarray}\label{eAuj}
|A_{t_{B_{1,i}}}f(x)-A_{t_{B}}f(x)|\leq c_5\beta.
\end{eqnarray}

Indeed, from Assumption $\mathrm{A}$, it follows that,
for almost every $x\in\mathcal{X}$,
\begin{equation}\label{deAuj}
A_{t_{B_{1,i}}}f(x)-A_{t_{B}}f(x)=
A_{t_{B_{1,i}}}(f-A_{t_B}f)(x)+
\left[A_{(t_{B_{1,i}}+t_{B})}f(x)-
A_{t_{B}}f(x)\right].
\end{equation}
By the fact that $t_{B_{1,i}}+t_B$ and $t_B$
have comparable sizes, Proposition \ref{AtB1},
\eqref{Bbkai}, \eqref{bBkai}, and \eqref{eBdc},
we find that, for almost every $x\in B_{1,i}$,
\begin{eqnarray*}
\left|A_{(t_{B_{1,i}+t_{B}})}f(x)-
A_{t_{B}}f(x)\right|&\lesssim&
\frac{\|\chi_{B(x,t_B^{1/m})}\|_{L^{\varphi}(\mathcal{X})}}
{\mu(B(x,t_B^{1/m}))}\|f\|_{\mathrm{BMO}^{\varphi}_A(\mathcal{X})}\\
&\lesssim&\frac{\|\chi_{B}\|_{L^{\varphi}
(\mathcal{X})}\|f\|_{\mathrm{BMO}^{\varphi}_A(\mathcal{X})}}
{\mu(B)}\lesssim\beta.
\end{eqnarray*}
From this and \eqref{deAuj}, we deduce that, to prove \eqref{eAuj},
we only need to prove that, for almost every $x\in B_{1,i}$,
\begin{eqnarray}\label{eAuj0}
|A_{t_{B_{1,i}}}(f-A_{t_B}f)(x)|\lesssim\beta
\end{eqnarray}

Let $q_i$ be the smallest integer such that
$2C_1^2B\subset2^{q_i+1}B_{1,i}$ and
$2C_1^2B\cap(2^{q_i}B_{1,i})^\complement\neq \emptyset$.
We claim that $2^{q_i+1}B_{1,i}\subset10C_1^4B$. Indeed, by
$2C_1^2B\cap(2^{q_i}B_{1,i})^\complement\neq \emptyset$, we see that
$$2^{q_i}r_{B_{1,i}}<C_1[r_{2C_1^2B}+d(x_B,x_{B_{1,i}})]\leq
C_1[2C_1^2r_B+C_1(r_B+r_{B_{1,i}})],$$
which, together with \eqref{qume},
implies that, for all $z\in2^{q_i+1}B_{1,i}$,
\begin{eqnarray*}
d(x_B,z)&\leq&C_1[d(x_B,x_{B_{1,i}})+d(x_{B_{1,i}},z)]<
C_1[C_1(r_B+r_{B_{1,i}})+2^{q_i+1}r_{B_{1,i}}]\\
&\leq&C_1^2[r_B+r_{B_{1,i}}+4C_1^2r_B+2C_1(r_B+r_{B_{1,i}})]
\leq10C_1^4r_B.
\end{eqnarray*}
Thus, the above claim holds true. Moreover, we write
\begin{eqnarray}\label{eAuj1}
&&\left|A_{t_{B_{1,i}}}(f-A_{t_B}f)(x)\right|\\\nonumber
&&\hs\lesssim
\frac{1}{\mu(B_{1,i})}\int_{\mathcal{X}}g\left(
\frac{[d(x,y)]^m}{t_{B_{1,i}}}\right)|f(y)-A_{t_B}f(y)|
\,d\mu(y)\\\nonumber
&&\hs\lesssim\sum_{k=0}^{q_i+1}\frac{1}{\mu(B_{1,i})}
\int_{2^kB_{1,i}\backslash2^{k-1}B_{1,i}}g\left(
\frac{[d(x,y)]^m}{t_{B_{1,i}}}\right)|f(y)-A_{t_B}f(y)|
\,d\mu(y)\\\nonumber
&&\hs\hs+\frac{1}{\mu(B_{1,i})}\int_{\mathcal{X}
\backslash2^{q_i+1}B_{1,i}}
\cdots=:\hbox{I}+\hbox{II}.
\end{eqnarray}

It follows immediately from property (iii)
of $\{B_{1,i}\}_{i\in \mathbb{N}}$
that, for all
$k\in\{0,\ldots,q_i+1\}$,
\begin{eqnarray*}
\frac{1}{\mu(2^kB_{1,i})}\int_{2^kB_{1,i}}|f_0(x)|
\,d\mu(x)\lesssim\beta,
\end{eqnarray*}
which, together with $f_0:=(f-A_{t_B}f)\chi_{10C_1^4B}$
and $2^{q_i+1}B_{1,i}\subset10C_1^4B$ by the above claim, implies that
\begin{eqnarray}\label{eAuj2}
\ \ \ \ \ \ \ \ \frac{1}{\mu(2^kB_{1,i})}
\int_{2^kB_{1,i}}|f(x)-A_{t_B}f(x)|
\,d\mu(x)=\frac{1}{\mu(2^kB_{1,i})}
\int_{2^kB_{1,i}}|f_0(x)|\,d\mu(x)
\lesssim \beta.
\end{eqnarray}

Notice that, for any $x\in B_{1,i}$,
$y\in 2^kB_{1,i}\setminus2^{k-1}B_{1,i}$
with $k\in\mathbb{N}$ and $k\geqslant\lfloor\log_2C_1\rfloor+2$,
there exists a positive constant $c_6$
such that $d(y,x)\geq c_62^kr_{B_{1,i}}$.
This, together with \eqref{eAuj2}, \eqref{pdcu},
the decreasing property of $g$, and
$M>n+\frac{2np_1}{p}+N-\frac{n(q-1)}{q}-\alpha>n$,
implies that
\begin{eqnarray}\label{eAuj3}
\hbox{I}&\lesssim&\sum_{k=0}^{\lfloor\log_2C_1\rfloor+1}2^{kn}g(0)
\frac{1}{\mu(2^{k}B_{1,i})}\int_{2^{k}B_{1,i}}|f(y)-A_{t_B}f(y)|
\,d\mu(y)\\\nonumber
&&+\sum_{k=\lfloor\log_2C_1\rfloor+2}^{q_i+1}2^{kn}g(c_6^m2^{km})
\frac{1}{\mu(2^{k}B_{1,i})}\int_{2^{k}B_{1,i}}|f(y)-A_{t_B}f(y)|
\,d\mu(y)\\\nonumber
&\lesssim&\beta\sum_{k=0}^{\lfloor\log_2C_1\rfloor+1}2^{kn}+
\beta\sum_{k=\lfloor\log_2C_1\rfloor+2}^{q_i+1}2^{k(n-M)}
\lesssim\beta,
\end{eqnarray}
where the second term is vacant, if $q_i<\lfloor\log_2C_1\rfloor+1$.

Now we estimate \hbox{II}. Let $s_i$ be an integer
satisfying $2^{s_i}r_{B_{1,i}}\leq r_B<2^{s_i+1}r_{B_{1,i}}$.
Let $2^{-1}B_{1,i}=\emptyset$. Then, by \eqref{eBdc},
we know that $\mu(B(x_B,r_{B_{1,i}}))\lesssim2^{s_iN}\mu(B_{1,i})$.
It is easy to see that, for any $x\in B_{1,i}$
and $y\in 2^{k+1}B\setminus2^{k}B$ with $k\in\mathbb{N}$
and $k\geqslant\lfloor\log_2C_1\rfloor+2$,
there exists a positive constant
$c_7$ such that $d(y,x)\geq c_72^{k+s_i}r_{B_{1,i}}$.
Recall that $2C_1^2B\subset2^{q_i+1}B_{1,i}$ and
$M>n+\frac{2np_1}{p}+N-\frac{n(q-1)}{q}
-\alpha> \max\{\frac{np_1}{p},\,n+N\}$.
Thus, from these facts, the decreasing property of $g$,
Proposition \ref{p-AtB}, \eqref{RDin} and
Lemma \ref{l-kai}, it follows that
\begin{eqnarray*}
\hbox{II}&&\lesssim\sum_{k=\lfloor2\log_2C_1\rfloor+1}
^{\infty}\frac{1}{\mu(B_{1,i})}
\int_{2^{k+1}B\setminus2^{k}B}g\left(\frac{[d(x,y)]^m}{t_{B_{1,i}}}
\right)|f(y)-A_{t_B}f(y)|\,d\mu(y)\\
&&\lesssim\sum_{k=\lfloor2\log_2C_1\rfloor+1}^{\infty}2^{s_i(N+n)}
g(c_7^m2^{(k+s_i)m})\frac{1}{\mu(B)}\int_{2^{k+1}B}
|f(y)-A_{t_B}f(y)|\,d\mu(y)\\
&&\lesssim\sum_{k=\lfloor2\log_2C_1\rfloor+1}^{\infty}2^{s_i(N+n)}
g(c_7^m2^{(k+s_i)m})\frac{1}{\mu(B)}\left[\int_{2^{k+1}B}
\left|f(y)-A_{t_{2^{k+1}B}}f(y)\right|\,d\mu(y)\right.\\
&&\hs+\left.\int_{2^{k+1}B}
\left|A_{t_{2^{k+1}B}}f(y)-A_{t_B}f(y)\right|\,d\mu(y)\right]\\
&&\lesssim\sum_{k=\lfloor2\log_2C_1\rfloor+1}^{\infty}2^{s_i(N+n)}
g(c_7^m2^{(k+s_i)m})\frac{\|f\|_{\mathrm{BMO}^{\varphi}_A(\mathcal{X})}}
{\mu(B)}\Bigg[
\|\chi_{2^{k+1}B}\|_{L^{\varphi}(\mathcal{X})}\\
&&\left.\hs+2^{(k+1)(\frac{np_1}{p}-\alpha)}\int_{2^{k+1}B}
\frac{\|\chi_{B(y,t_B^{1/m})}\|
_{L^{\varphi}(\mathcal{X})}}{\mu(B(y,t_B^{1/m}))}\,d\mu(y)\right]\\
&&\lesssim\sum_{k=\lfloor2\log_2C_1\rfloor+1}^{\infty}2^{s_i(N+n-M)}
\left[2^{k(\frac{np_1}{p}-M)
}+2^{k(n+\frac{2np_1}{p}
+N-\frac{n(q-1)}{q}-\alpha-M)}\right]
\\
&&\hs\times\frac{\|\chi_B\|_{L^{\varphi}(\mathcal{X})}}
{\mu(B)}\|f\|_{\mathrm{BMO}^{\varphi}_A(\mathcal{X})}\lesssim\beta,
\end{eqnarray*}
which, together with \eqref{eAuj3} and \eqref{eAuj1},
implies that \eqref{eAuj0}
and hence \eqref{eAuj} hold true.

We claim that, for each $B_{1,i}$
with $B_{1,i}\cap B\neq\emptyset$, $B_{1,i}\subset2C_1^2B$.
Indeed, notice that
$B_{1,i}\cap B\neq\emptyset$ and
$r_B\geq 10C_1r_{B_{1,i}}$.
Then, by \eqref{qume}, we know that, for all $x\in B_{1,i}$,
$d(x_B,x)<
C_1(2C_1r_{B_{1,i}}+r_B)\leq\frac{6}{5}C_1r_B
\leq2C_1^2r_B$. Thus, $B_{1,i}\subset2C_1^2B$, namely, the claim
holds true. This, together
with \eqref{afnj} and Lemma \ref{l-uAp}(i),
implies that there exists a positive constant $c_7$ such that
\begin{eqnarray*}
\frac{\|\chi_{B_{1,i}}\|
_{L^{\varphi}(\mathcal{X})}}{\mu(B_{1,i})}\lesssim \frac{\|\chi_{2C_1^2B}\|
_{L^{\varphi}(\mathcal{X})}}{\mu(2C_1^2B)}\leq
c_6\frac{\|\chi_{B}\|
_{L^{\varphi}(\mathcal{X})}}{\mu(B)}.
\end{eqnarray*}
From this and \eqref{huj}, we deduce that
\begin{eqnarray*}
\frac{\|\chi_{B_{1,i}}\|
_{L^{\varphi}(\mathcal{X})}}{\mu(B_{1,i})}
\|f\|_{\mathrm{BMO}^{\varphi}_A(\mathcal{X})}\leq
c_6\frac{\|\chi_{B}\|
_{L^{\varphi}(\mathcal{X})}}{\mu(B)}
\|f\|_{\mathrm{BMO}^{\varphi}_A(\mathcal{X})}\leq c_6.
\end{eqnarray*}

Applying use the decomposition in \cite[Chapter III, Theorem 1.3]{cw2} for
$$f_{1,i}:=(f-A_{t_{B_{1,i}}}f)\chi_{10C_1^4B_{1,i}}$$
with the same value $\beta$ as above again, we obtain a collection
$\{B_{2,m}\}_{m\in\mathbb{N}}$ of balls satisfying that
$B_{2,j}\cap B_{1,i}\neq\emptyset$ for any
$B_{2,j}\in\{B_{2,m}\}_{m\in\mathbb{N}}$,
$|f(x)-A_{t_{B_{1,i}}}f(x)|\leq \beta$ for any
$x \in B_{1,i}\setminus(\cup_mB_{2,m})$, and
$\sum_m\mu(B_{2,m})\leq \frac{c_3c_6}{\beta}\mu(B_{1,i})$.
We now further choose
$\beta>(\max\{\widetilde{c}_3,
c_4(10C_1)^{n+N}, c_3e\})(\max\{1,c_6\})$and let
$\widetilde{c}_6:=\max\{1,c_6\}$.
By a method similar to that used in the proof of \eqref{eAuj},
we see that, for almost every $x\in B_{2,m}$,
$$|A_{t_{B_{2,m}}}f(x)-A_{t_{B_{1,i}}}f(x)|\leq c_5c_6\beta.$$
Now we put together all families $\{B_{2,m}\}$ corresponding to
different $B_{1,i}$, which are still denoted by $\{B_{2,m}\}$.
Then, for all $x\in B\setminus(\cup_mB_{2,m})$, we have
$$|f(x)-A_{t_B}f(x)|\leq|f(x)-A_{t_{B_{1,i}}}f(x)|+
|A_{t_{B_{1,i}}}f(x)-A_{t_B}f(x)|\leq 2c_5\widetilde{c}_6\beta$$
and
$$\sum_m\mu(B_{2,m})\leq\left(\frac{c_3}{\beta}\right)^2c_6\mu(B).$$

Therefore, by induction, we know that, for each $K\in\mathbb{N}$,
there exists a family $\{B_{K,m}\}_{K\in\mathbb{N}}$
of balls satisfying that, for any $B_{K+1,m}$,
there exists a ball $B_{K,m}$ satisfying
$B_{K+1,m}\cap B_{K,m}\neq\emptyset$,
$r_{B_{K,m}}>10C_1r_{B_{K+1,m}}$ and
$B_{K+1,m}\subset2C_1^2B$.
Then, from \eqref{afnj}, we deduce that
\begin{eqnarray*}
\frac{\|\chi_{B_{K+1,m}}\|
_{L^{\varphi}(\mathcal{X})}}{\mu(B_{K+1,m})}
\|f\|_{\mathrm{BMO}^{\varphi}_A(\mathcal{X})}\leq
c_6\frac{\|\chi_{B}\|
_{L^{\varphi}(\mathcal{X})}}{\mu(B)}
\|f\|_{\mathrm{BMO}^{\varphi}_A(\mathcal{X})}\leq c_6.
\end{eqnarray*}
Moreover, we also have
$$|f(x)-A_{t_B}f(x)|\leq Kc_5\widetilde{c}_6\beta\ {\rm for\ almost\ every}\ x\in
B\setminus\left(\bigcup_mB_{K,m}\right),$$
and
$$\sum_m\mu(B_{K,m})\leq
\left(\frac{c_3c_6}{\beta}\right)^K\frac{\mu(B)}{c_6}.$$

If $Kc_5\widetilde{c}_6
\beta\leq \alpha<(K+1)c_5\widetilde{c}_6
\beta$ with $K\in\mathbb{N}$, then,
from
$\beta>(c_3c_6)^2$, we deduce that
\begin{eqnarray*}
\mu(\{x\in B:\ |f(x)-A_{t_B}f(x)|>\alpha\})&\leq&\sum_m\mu(B_{K,m})\leq
\left(\frac{c_3c_6}{\beta}\right)^K\frac{\mu(B)}{c_6}\\
&\leq&e^{-(K\log\beta)/2}\frac{\mu(B)}{c_6}
\leq\frac{\sqrt{\beta}}{c_6}
e^{-\frac{\alpha\log\beta}{4c_5\widetilde{c}_6\beta}}\mu(B).
\end{eqnarray*}
On the other hand, if $\alpha<c_5\widetilde{c}_6\beta$,
we just have the following trivial estimate
$$\mu(\{x\in B:\ |f(x)-A_{t_B}f(x)|>\alpha\})\leq \mu(B)
\leq e^{1-\frac{\alpha}{c_5\widetilde{c}_6\beta}}\mu(B).$$
Combining both estimates, we then obtain \eqref{suj} for each $\alpha>1$ by choosing
$$c_1:=\max\lf\{e,\frac{\sqrt{\beta}}{c_6}\r\}\ \ \ \hbox{and}
\ \ \ c_2:=\frac{\min\{(\log\beta)/4,1\}}{c_5\widetilde{c}_6\beta},$$
which completes the proof of Theorem \ref{t-uj}.
\end{proof}
\begin{remark}\label{r-uj}
(i) When $\varphi$ is as in \eqref{fat}, \eqref{afnj} automatically holds true
and Theorem \ref{t-uj} is just \cite[Theorem 3.1]{dy}.

(ii) When
$\varphi$ is as in \eqref{fatb},
\eqref{afnj} also automatically holds true
and Theorem \ref{t-uj} is just \cite[Theorem 3.1]{tang}.
\end{remark}

As a consequence of Theorem \ref{t-uj},
we obtain the following conclusion
for $\mathrm{BMO}^{\varphi}_A(\mathcal{X})$.

\begin{theorem}\label{t-cuj}
Let $\mathcal{X}$ be a space of homogeneous type with degree
$(\alpha_0,n_0,N_0)$,
where $\alpha_0$, $n_0$ and $N_0$ are as in
\eqref{d-ar0}, \eqref{d-n0} and \eqref{d-N0},
respectively, $\varphi$ as in Definition \ref{d-grf} satisfying \eqref{afnj}, and
$\{A_t\}_{t>0}$ satisfy Assumption $\mathrm{A}$.
Assume that $f\in\mathrm{BMO}^{\varphi}_A(\mathcal{X})$.
Then there exist
positive constants $\lambda$ and $C$ such that
$$
\sup_{B\subset\mathcal{X}}\frac{1}{\mu(B)}\int_B\exp
\left\{\frac{\lambda\mu(B)}{\|\chi_{B}\|_{L^{\varphi}(\mathcal{X})}
\|f\|_{\mathrm{BMO}^{\varphi}
_A(\mathcal{X})}}|f(x)-A_{t_B}f(x)|\right\}
\,d\mu(x)\leq C,
$$
where the supremum is taken over all balls
$B\subset\mathcal{X}$ and $t_B:=r_B^m$.
\end{theorem}

\begin{proof}
Let $\lambda:=c_2/2$, where $c_2$ is as in Theorem \ref{t-uj}.
Then, by Theorem \ref{t-uj}, we see that, for any $B\subset\mathcal{X}$,
\begin{eqnarray*}
&&\int_B\exp
\left\{\frac{\lambda\mu(B)}{\|\chi_{B}\|_{L^{\varphi}(\mathcal{X})}
\|f\|_{\mathrm{BMO}^{\varphi}_A
(\mathcal{X})}}|f(x)-A_{t_B}f(x)|\right\}
\,d\mu(x)\\
&&\hs=\int^{\infty}_0\mu\left(\left\{x\in B:\ \exp
\left[\frac{\lambda\mu(B)}{\|\chi_{B}\|_{L^{\varphi}(\mathcal{X})}
\|f\|_{\mathrm{BMO}^{\varphi}_A(\mathcal{X})}}
|f(x)-A_{t_B}f(x)|\right]>t\right\}\right)\,dt\\
&&\hs\leq\mu(B)+\int_1^{\infty}\mu\left(\left\{x\in B:\
|f(x)-A_{t_B}f(x)|>\frac{\|\chi_{B}\|_{L^{\varphi}(\mathcal{X})}
\|f\|_{\mathrm{BMO}^{\varphi}_A(\mathcal{X})}\ln t}
{\lambda\mu(B)}\right\}\right)\,dt\\
&&\hs\leq\mu(B)+c_1\mu(B)\int_1^{\infty}\exp\left\{-
\frac{c_2\ln t}{\lambda}\right\}\,dt
\leq\mu(B)+c_1\mu(B)\int_1^{\infty}t^{-c_2/\lambda}\,dt
\lesssim\mu(B),
\end{eqnarray*}
which completes the proof of Theorem \ref{t-cuj}.
\end{proof}

Now we introduce the spaces
$\mathrm{BMO}^{\varphi,\,\widetilde{p}}_A(\mathcal{X})$ for
$\widetilde{p}\in[1,\infty)$.
\begin{definition}\label{d-puj}
Let $\mathcal{X}$ be a space of homogeneous type,
$\fai$ as in Definition \ref{d-grf},
$\mathcal{M}(\mathcal{X})$ as in \eqref{Max0b},
and $\{A_t\}_{t>0}$ a generalized approximation to the identity satisfying
\eqref{ubfA} and \eqref{pdcu}.
Let $\widetilde{p}\in[1,\infty)$.
The \emph{space}
$\mathrm{BMO}^{\varphi,\,\widetilde{p}}_A(\mathcal{X})$
is defined to be the set of all $f\in\mathcal{M}(\mathcal{X})$
such that
\begin{eqnarray*}
\|f\|_{\mathrm{BMO}^{\varphi,\,\widetilde{p}}_A(\mathcal{X})}:=
\sup_{B\subset\mathcal{X}}\frac{\mu(B)}
{\|\chi_{B}\|_{L^{\varphi}(\mathcal{X})}}\left\{
\frac{1}{\mu(B)}\int_B|f(x)-A_{t_B}f(x)|^{\widetilde{p}}\,d\mu(x)\right\}
^{\frac{1}{\widetilde{p}}}< \infty,
\end{eqnarray*}
where the supremum is taken over all balls $B\subset\mathcal{X}$,
$t_B:=r_B^m$ and $r_B$ denotes the radius of the ball $B$.
\end{definition}

By Theorem \ref{t-uj}, we obtain the following conclusion.

\begin{theorem}\label{t-puj}
Let $\mathcal{X}$ be a space of homogeneous type with degree
$(\alpha_0,n_0,N_0)$,
where $\alpha_0$, $n_0$ and $N_0$ are as in
\eqref{d-ar0}, \eqref{d-n0} and \eqref{d-N0},
respectively.
Assume that $\varphi$ is as in Definition \ref{d-grf}
satisfying \eqref{afnj}
and $\{A_t\}_{t>0}$ satisfies Assumption $\mathrm{A}$.
For different $\widetilde{p}\in[1,\infty)$, the spaces
$\mathrm{BMO}^{\varphi,\,\widetilde{p}}_A(\mathcal{X})$
coincide with equivalent norms.
\end{theorem}

\begin{proof}
For any $f\in\mathrm{BMO}^{\varphi,\,\widetilde{p}}_A(\mathcal{X})$
with $\widetilde{p}\in[1,\infty)$, by H\"{o}lder's inequality,
we see that, for any ball $B\subset\mathcal{X}$,
\begin{eqnarray*}
&&\frac{1}{\|\chi_B\|_{L^{\varphi}(\mathcal{X})}}\int_{B}
|f(x)-A_{t_B}f(x)|\,d\mu(x)\\
&&\hs\leq\frac{\mu(B)}{\|\chi_{B}\|_{L^{\varphi}(\mathcal{X})}}\left\{
\frac{1}{\mu(B)}\int_B|f(x)-A_{t_B}f(x)|^{\widetilde{p}}\,d\mu(x)\right\}
^{\frac{1}{\widetilde{p}}},
\end{eqnarray*}
which implies that $f\in \mathrm{BMO}^{\varphi}_A(\mathcal{X})$ and
$\|f\|_{\mathrm{BMO}^{\varphi}_A(\mathcal{X})}\leq
\|f\|_{\mathrm{BMO}^{\varphi,\widetilde{p}}_A(\mathcal{X})}$.

Let $f\in \mathrm{BMO}^{\varphi}_A(\mathcal{X})$ and
$\widetilde{p}\in[1,\infty)$. From Theorem \ref{t-uj},
it follows that, for any $B\subset\mathcal{X}$,
\begin{eqnarray*}
\int_B|f(x)-A_{t_B}f(x)|^{\widetilde{p}}\,
d\mu(x)&=&\widetilde{p}\int_0^{\infty}
\lambda^{\widetilde{p}-1}\mu(\{x\in B:\
|f(x)-A_{t_B}f(x)|>\lambda\})\,d\lambda\\
&\lesssim&\mu(B)\int_0^{\infty}\lambda^{\widetilde{p}-1}
e^{-\frac{c_2\lambda \mu(B)}{\|\chi_{B}\|_{L^{\varphi}(\mathcal{X})}
\|f\|_{\mathrm{BMO}^{\varphi}_A(\mathcal{X})}}}\,d\lambda\\
&\lesssim&\|f\|_{\mathrm{BMO}^
{\varphi}_A(\mathcal{X})}^{\widetilde{p}}\frac{\|\chi_{B}\|
_{L^{\varphi}(\mathcal{X})}^{\widetilde{p}}}
{\mu(B)^{\widetilde{p}}}\mu(B),
\end{eqnarray*}
which implies that
$$\frac{\mu(B)}{\|\chi_{B}\|_{L^{\varphi}(\mathcal{X})}}\left[
\frac{1}{\mu(B)}\int_B|f(x)-A_{t_B}f(x)|^{\widetilde{p}}\,d\mu(x)\right]
^{\frac{1}{\widetilde{p}}}\lesssim
\|f\|_{\mathrm{BMO}^{\varphi}_A(\mathcal{X})}.$$
Thus, $f\in\mathrm{BMO}^{\varphi,\,\widetilde{p}}
_A(\mathcal{X})$ and $\|f\|_{\mathrm{BMO}^{\varphi,\,\widetilde{p}}
_A(\mathcal{X})}\lesssim \|f\|_{\mathrm{BMO}^{\varphi}_A(\mathcal{X})}$.
This finishes the proof of Theorem \ref{t-puj}.
\end{proof}

\begin{remark}\label{r5.2}
Theorem \ref{t-puj} completely covers \cite[Theorem 3.4]{dy}
and \cite[Theorem 3.3]{tang} by taking $\varphi$,
respectively, as in \eqref{fat} and \eqref{fatb}.
\end{remark}

\subsection{The weighted version of the John-Nirenberg inequality on
$\mathrm{BMO}^{\varphi}_A(\mathcal{X})$}\label{s-j}

\hskip\parindent In this subsection,
we establish a weighted John-Nirenberg
inequality on $\mathrm{BMO}^{\varphi}_A(\mathcal{X})$.
We begin with the following
Lemma \ref{l-j}, whose proof is similar to that of \cite[Lemma 2]{lai},
the details being omitted.

In what follows, for any set $E\subset\mathcal{X}$
and $t\in(0,\infty)$, let $\varphi(E,t):=\int_E\varphi(x,t)\,d\mu(x)$.

\begin{lemma}\label{l-j}
Let $\mathcal{X}$ be a space of homogeneous type. Assume that
$\fai$ is as in Definition \ref{d-grf} and
$\varphi\in \mathbb{A}_{\widetilde{p}_1}(\mathcal{X})$
with $\widetilde{p}_1\in (1,\infty)$.
Then there exists a positive constant $C$ such that, for all balls
$B\subset\mathcal{X}$, $\lambda\in(0,\infty)$
and $t\in(0,\infty)$,
$$\varphi\left(\left\{x\in B:\ M_{\varphi(\cdot,t)}
\left(\frac{1}{\varphi(\cdot,t)}
\chi_B\right)(x)>\lambda\right\},t\right)
\leq C\left[\frac{\mu(B)}{\lambda \varphi(B,t)}\right]
^{\widetilde{p}_1'}\varphi(B,t),$$
where $1/\widetilde{p}_1+1/\widetilde{p}_1'=1$ and $M_{\varphi(\cdot,t)}$
denotes the \emph{maximal function associated with $\varphi(\cdot,t)$},
namely, for all $f\in L^1_{\mathop\mathrm{loc}}
(\varphi(\cdot,t)\,d\mu)$
and $x\in\mathcal{X}$,
$$M_{\varphi(\cdot,t)}(f)(x):=\sup_{B\ni x}\frac{1}
{\varphi(B,t)}\int_B|f(y)|\varphi(y,t)\,d\mu(y).$$
\end{lemma}

Now we give out the weighted John-Nirenberg
inequality on $\mathrm{BMO}^{\varphi}_A(\mathcal{X})$ as follows.

\begin{theorem}\label{t-j}
Let $\mathcal{X}$ be a space of homogeneous type with degree
$(\alpha_0,n_0,N_0)$,
where $\alpha_0$, $n_0$ and $N_0$ are as in
\eqref{d-ar0}, \eqref{d-n0} and \eqref{d-N0},
respectively.
Let $\varphi$ be as in Definition \ref{d-grf}
and $\{A_t\}_{t>0}$ satisfy Assumption $\mathrm{A}$.

{\rm (i)} Assume that $\varphi\in\mathbb{A}_1(\mathcal{X})$.
Then, there exist positive constants $c_1$ and $c_2$ such that,
for all $f\in \mathrm{BMO}^{\varphi}_A(\mathcal{X})$,
balls $B$ and $\lambda\in(0,\infty)$,
\begin{eqnarray}\label{jn1}
&&\varphi\left(\left\{x\in B:\ \frac{|f(x)-A_{t_B}f(x)|}
{\varphi(x,\|\chi_{B}\|_{L^{\varphi}(\mathcal{X})}^{-1})}
>\lambda\right\},\|\chi_{B}\|_{L^{\varphi}(\mathcal{X})}^{-1}\right)\\
&&\hs\leq c_1\exp\left\{-\frac{c_2\lambda}
{\|\chi_{B}\|_{L^{\varphi}(\mathcal{X})}
\|f\|_{\mathrm{BMO}^{\varphi}_A(\mathcal{X})}}\right\},\nonumber
\end{eqnarray}
where $t_B:=r_B^m$ and $m$ is as in \eqref{ubfA}.

{\rm (ii)} Assume that $\varphi\in\mathbb{A}_{p_1}(\mathcal{X})$ for some
$p_1\in(1,\infty)$ and $p(\varphi)\leq1+\frac{1}{[r(\varphi)]'}$,
where $p(\varphi)$ and $r(\varphi)$ are, respectively, as in
\eqref{crAp} and \eqref{crRD}, and $1/r(\varphi)+1/[r(\varphi)]'=1$.
Then, there exist positive constants $b_1$ and $b_2$ such
that, for all $f\in \mathrm{BMO}^{\varphi}_A(\mathcal{X})$,
balls $B$ and $\lambda\in(0,\infty)$,
\begin{eqnarray}\label{jn2}
&&\varphi\left(\left\{x\in B:\ \frac{|f(x)-A_{t_B}f(x)|}
{\varphi(x,\|\chi_{B}\|_{L^{\varphi}(\mathcal{X})}^{-1})}
>\lambda\right\},\|\chi_{B}\|_{L^{\varphi}
(\mathcal{X})}^{-1}\right)\\\nonumber
&&\hs\leq b_1\left(
\min\left\{1,\ b_2\left[\frac{\lambda
\varphi(B,\|\chi_{B}\|_{L^{\varphi}(\mathcal{X})}^{-1})}
{\|\chi_{B}\|_{L^{\varphi}(\mathcal{X})}
\|f\|_{\mathrm{BMO}^{\varphi}_A(\mathcal{X})}}\right]^{-p_1'}\right\}\right).
\end{eqnarray}
\end{theorem}

\begin{proof}
Let $f\in \mathrm{BMO}^{\varphi}_A(\mathcal{X})$.
Fix a ball $B_0\subset\mathcal{X}$.
In order to prove \eqref{jn1} and \eqref{jn2},
it suffices to consider the case
$\|f\|_{\mathrm{BMO}^{\varphi}_A(\mathcal{X})}>0$.
Otherwise, they holds true obviously.
Let $t_0:=\|\chi_{B_0}\|_{L^{\varphi}(\mathcal{X})}^{-1}$.
Without loss of generality, we may assume that
\begin{equation}\label{hjn}
\|f\|_{\mathrm{BMO}^{\varphi}_A(\mathcal{X})}=t_0;
\end{equation}
otherwise, we replace $f$ by
$\frac{t_0f}{
\|f\|_{\mathrm{BMO}^{\varphi}_A(\mathcal{X})}}$.
Thus, we only need to prove that there
exist positive constants $c_1$, $c_2$, $b_1$ and
$b_2$ such that, for any $\lambda\in(0,\infty)$,
when $\varphi\in\mathbb{A}_1(\mathcal{X})$,
\begin{eqnarray}\label{sjn1}
\varphi\left(\left\{x\in B_0:\ \frac{|f(x)-A_{t_{B_0}}f(x)|}
{\varphi(x,t_0)}
>\lambda\right\},t_0\right)\leq c_1e^{-c_2\lambda}
\end{eqnarray}
and, when $\varphi\in\mathbb{A}_{p_1}(\mathcal{X})$
with $p_1\in(1,\infty)$,
\begin{eqnarray}\label{sjn2}
\varphi\left(\left\{x\in B_0:\ \frac{|f(x)-A_{t_{B_0}}f(x)|}
{\varphi(x,t_0)}
>\lambda\right\},t_0\right)
\leq b_1\lf(\min\lf\{1,b_2\lambda^{-p_1'}\r\}\r),
\end{eqnarray}
where $t_B:=r_{B_0}^m$.

It is obvious that, when $\lambda\in(0,1)$,
\eqref{sjn1} and \eqref{sjn2} hold true for
$c_1:=e$, $c_2:=1$, $b_1:=1$ and $b_2:=1$.

Now, let $\lambda\in[1,\infty)$. Let $B:=B(x_B,r_B)\subset B_0$
and, for all $x\in\mathcal{X}$,
$$f_0(x):=\frac{f(x)-A_{t_B}f(x)}
{\varphi(x,t_0)}
\chi_{10C_1^4B}(x),$$
where $C_1$ is as in \eqref{qume}. Similar to the proof of
\eqref{efuj}, by the property of uniformly upper type $1$ of
$\varphi$,
Lemma \ref{l-grf}(i) and \eqref{hjn}, we know that
\begin{eqnarray}\label{efjn}
\|f_0\|_{L^1_{\varphi(\cdot,t_0)}(\mathcal{X})}
&:=&\int_{\mathcal{X}}|f_0(x)|
\varphi(x,t_0)\,d\mu(x)\\ \nonumber
&\lesssim&\|\chi_{B}\|_{L^{\varphi}
(\mathcal{X})}\|f\|_{\mathrm{BMO}^{\varphi}_A(\mathcal{X})}\\ \nonumber
&\lesssim&\frac{\varphi(B,t_0)\|\chi_{B}\|_{L^{\varphi}
(\mathcal{X})}\|f\|_{\mathrm{BMO}^{\varphi}_A(\mathcal{X})}}
{\varphi(B,\|\chi_{B}\|_{L^{\varphi}(\mathcal{X})}^{-1}
t_0\|\chi_{B}\|_{L^{\varphi}(\mathcal{X})})}\\ \nonumber
&\lesssim&\frac{\varphi(B,t_0)\|\chi_{B}\|_{L^{\varphi}
(\mathcal{X})}\|f\|_{\mathrm{BMO}^{\varphi}_A(\mathcal{X})}}
{\varphi(B,\|\chi_{B}\|_{L^{\varphi}(\mathcal{X})}^{-1})
t_0\|\chi_{B}\|_{L^{\varphi}(\mathcal{X})}}
\thicksim\varphi(B,t_0).
\end{eqnarray}

Let
$\beta\in(1,\infty)$,
$$F:=\{x\in\mathcal{X}:\ M_{\varphi(\cdot,t_0)}(f_0)(x)
\leq \beta\}\ \ \  \hbox{and} \ \ \
\Omega:=F^\complement=\{x\in\mathcal{X}:\
M_{\varphi(\cdot,t_0)}(f_0)(x)> \beta\}.$$
By \cite[Chapter III, Theorem 1.3]{cw2},
we know that there exists a
collection of balls, $\{B_{1,i}\}_{i\in\mathbb{N}}$, satisfying
that

\ \ (i) $\cup_iB_{1,i}=\Omega$;

\ (ii) each point of $\Omega$ is contained in at most a finite
number $L$ of the balls $B_{1,i}$;

(iii) there exists $C\in(1,\infty)$ such that $CB_{1,i}\cap
F\neq \emptyset$ for each $i$.

From (i), we deduce that, for any $x \in B\backslash(\cup_iB_{1,i}),$
$$|f(x)-A_{t_B}f(x)|=|f_0(x)|\chi_F(x)\leq
M_{\varphi(\cdot,t_0)}(f_0)(x)\chi_F(x)\leq\beta.$$
By the fact that $M_{\varphi(\cdot,t_0)}$ is of
weak type (1,1), (i), (ii) and \eqref{efjn}, we conclude that
there exists a positive constant $c_3$ such that
\begin{eqnarray}\label{eBjn1}
\sum_i\varphi(B_{1,i},t_0)\leq L\varphi(\Omega,t_0)\lesssim \frac{1}{\beta}
\|f_0\|_{{L^1_{\varphi}(\mathcal{X})}}\leq \frac{c_3}{\beta}\varphi(B,t_0).
\end{eqnarray}

For any $B_{1,i}\cap B\neq\emptyset$, we denote by
$B_{1,i}:=B(x_{B_{1,i}},r_{B_{1,i}})$ the ball centered at
$x_{B_{1,i}}\in \mathcal{X}$ and of radius
$r_{B_{1,i}}\in(0,\infty)$. Notice that
$d(x_B, x_{B_{1,i}})< C_1(r_B+r_{B_{1,i}})$. If $r_B<r_{B_{1,i}}$,
then $d(x_B, x_{B_{1,i}})< 2C_1r_{B_{1,i}}$. By Lemma \ref{l-uAp}(ii),
\eqref{RDin}, Lemma \ref{l-uAp}(i),
\eqref{eBdc}, \eqref{eBjn1} and some estimates
similar to those used in
\eqref{eBuj2}, we know that
there exists a positive constant $\widetilde{c}_3$ such that
\begin{eqnarray*}
\varphi(B,t_0)&\leq&\frac{\widetilde{c}_3}
{\beta}\left(\frac{r_B}{r_{B_{1,i}}}\right)
^{\frac{\alpha(q-1)}{q}}
\varphi(B,t_0),
\end{eqnarray*}
which further implies that, when $\beta>\widetilde{c}_3$,
it holds true $r_B\geq r_{B_{1,i}}$. We now choose $\bz>\wz{c}_3$ and
hence $r_B\geq r_{B_{1,i}}$.
By this, together with Lemma \ref{l-uAp}(i), \eqref{eBdc}, \eqref{sthp},
\eqref{eBjn1}, $d(x_B, x_{B_{1,i}})<2C_1r_B$
and some estimates similar to those used in
\eqref{eBuj3}, we find that
 there exists a positive constant $c_4$ such that
\begin{eqnarray}\label{eBjn2}
\varphi(B,t_0)&\leq&
\frac{c_4}{\beta}\left(\frac{r_B}{r_{B_{1,i}}}
\right)^{2np_1}\varphi(B,t_0).
\end{eqnarray}

We further choose
\begin{eqnarray}\label{btjn}
\beta>\max\{\widetilde{c}_3,
c_4(10C_1)^{2np_1}, 4^{p_1'}c_3e\},
\end{eqnarray}
where $1/p_1+1/p_1'=1$. Then, from
\eqref{eBjn2}, we deduce that $r_B>10C_1r_{B_{1,i}}$
which, together with \eqref{qume}, implies that,
for any $B_{1,i}\cap B\neq\emptyset$, $B_{1,i}\subset 2C_1B$.

Now we prove (i) and (ii) separately.

{\rm (i)} When $\varphi\in\mathbb{A}_1(\mathcal{X})$,
we claim that there exists a positive constant $c_5$ such
that, for any $B_{1,i}\cap B\neq\emptyset$ and almost every $x\in B_{1,i}$,
\begin{eqnarray}\label{eAjn}
\frac{|A_{t_{B_{1,i}}}f(x)-A_{t_{B}}f(x)|}{\varphi(x,t_0)}
\leq c_5\beta.
\end{eqnarray}

Indeed, when $\varphi\in \mathbb{A}_1(\mathcal{X})$,
by Definition \ref{d-uAp}, for any ball
$\widetilde{B}\subset\mathcal{X}$ and  $t\in(0,\infty)$, we see that
$$\frac{\varphi(\widetilde{B},t)}{\mu(\widetilde{B})}=
\frac{1}{\mu(\widetilde{B})}\int_{\widetilde{B}} \fai(x,t)\,d\mu(x)\lesssim
\einf_{y\in \widetilde{B}}\fai(y,t).$$
Then, we obtain \eqref{eAjn} by a
procedure similar to that used in the estimates for \eqref{eAuj}.

Let $b:=c_5\beta$. Then, for any $\lambda\in(0,\infty)$, we find that
\begin{eqnarray}\label{fbbjn}
&&\left\{x\in B:\ \frac{|f(x)-A_{t_{B}}f(x)|}{\varphi(x,t_0)}
>\lambda+b\right\}\\\nonumber
&&\hs\subset\bigcup_i\left(\left\{x\in B_{1,i}:\
\frac{|f(x)-A_{t_{B_{1,i}}}f(x)|}{\varphi(x,t_0)}
>\lambda\right\}\right.\\\nonumber
&&\hs\hs\left.\bigcup\left\{x\in B_{1,i}:\
\frac{|A_{t_{B_{1,i}}}f(x)-A_{t_B}f(x)|}{\varphi(x,t_0)}
>b\right\}\right)\\\nonumber
&&\hs\subset\bigcup_i\left\{x\in B_{1,i}:\
\frac{|f(x)-A_{t_{B_{1,i}}}f(x)|}{\varphi(x,t_0)}>\lambda\right\}.
\end{eqnarray}

For any $\lambda\in(0,\infty)$, we define
$\sigma_{f,B}(\lambda):=\varphi(\{x\in B:\
\frac{|f(x)-A_{t_{B}}f(x)|}{\varphi(x,t_0)}>\lambda\},t_0)$ and
$F_f(\lambda):=\sup_{B\subset B_0}\frac{\sigma_{f,B}
(\lambda)}{\varphi(B,t_0)}$.
Then, by \eqref{fbbjn}, we know that, for any $\lambda\in(0,\infty)$,
$\sigma_{f,B}(\lambda+b)\leq\sum_iF_f(\lambda)\mu(B_{1,i})$,
which, together with \eqref{eBjn1} and
\eqref{btjn},
implies that, for any $\lambda\in(0,\infty)$,
$$F_f(\lambda+b)\leq \frac{c_3}{\beta}F_f(\lambda)\leq e^{-1}F_f(\lambda).$$
By induction, we know that, for all $n\in\mathbb{N}$, $F_f(nb)\leq e^{1-n}$.
Thus, for any $n\in\mathbb{N}$ and $\lambda\in[nb,(n+1)b)$,
by the fact that $F_f$ is non-increasing, we conclude that
\begin{eqnarray}\label{fuj}
F_f(\lambda)\leq F_f(nb)\leq e^{1-n}<e^{2-\frac{\lambda}{b}}.
\end{eqnarray}
Notice that $F_f(\lambda)\leq1$. It is obvious
that \eqref{fuj} holds true for $\lambda\in[1,b)$.
Thus, \eqref{sjn1} always holds true, which completes
the proof of Theorem \ref{t-j}(i).

{\rm (ii)} In this case, by  Assumption $\mathrm{A}$ and
$p(\varphi)\leq1+\frac{1}{[r(\varphi)]'}$,
we see that there exist $n\in[n_0,\infty)$, $N\in[N_0,\infty)$,
$\alpha\in[0,\alpha_0]$,
$p_1\in[p(\fai),\fz)$, $p\in(0,i(\fai)]$ and $q\in(1,r(\fai)]$ such that
$\mathcal{X}$ satisfies
\eqref{sthp}, \eqref{eBdc} and \eqref{RDin},
respectively, for $n$, $N$ and $\alpha$,
$\fai\in\aa_{p_1}(\cx)$, $\fai$ is of uniformly lower type $p$,
$\fai\in\mathbb{RH}_{q}(\cx)$, $p_1-1-\frac{q-1}{q}\leq0$
and $M>n+\frac{2np_1}{p}+N-\frac{n(q-1)}{q}-\alpha> np_1$.

We claim that there exists a positive constant $\widetilde{c}_5$ such
that, for any $B_{1,i}\cap B\neq\emptyset$
and almost every $x\in B_{1,i}$,
\begin{eqnarray}\label{eAj}
\frac{|A_{t_{B_{1,i}}}f(x)-A_{t_{B}}f(x)|}{\varphi(x,t_0)}
\leq \frac{\widetilde{c}_5\beta\varphi(B_{1,i},t_0)}
{\varphi(x,t_0)\mu(B_{1,i})}.
\end{eqnarray}

Indeed, from Assumption $\mathrm{A}$, it follows that, for almost every $x\in\mathcal{X}$,
\begin{equation}\label{deAj}
A_{t_{B_{1,i}}}f(x)-A_{t_{B}}f(x)=
A_{t_{B_{1,i}}}(f-A_{t_B}f)(x)+\left[A_{(t_{B_{1,i}}+t_{B})}f(x)-
A_{t_{B}}f(x)\right].
\end{equation}

By some estimates similar to those used in the proof of Proposition \ref{dBMf}
(see \eqref{AtsAB}, \eqref{sCAtB} and \eqref{At2AB}),
\eqref{hjn}, Lemma \ref{l-uAp}(i), \eqref{bBkai} and
$p_1-1-\frac{q-1}{q}\leq0$, we find that, for almost every $x\in B_{1,i}$,
\begin{eqnarray}\label{reAA}
\left|A_{(t_{B_{1,i}+t_{B}})}f(x)-
A_{t_{B}}f(x)\right|&\lesssim&
\frac{\|\chi_{B(x,t_{B_{1,i}}^{1/m})}\|_{L^{\varphi}(\mathcal{X})}}
{\mu(B(x,t_B^{1/m}))}\|f\|_{\mathrm{BMO}^{\varphi}_A(\mathcal{X})}\\\nonumber
&\lesssim&\frac{\|\chi_{B_{1,i}}\|_{L^{\varphi}
(\mathcal{X})}}{\mu(B)}\frac{\varphi(B_0,t_0)}{\|\chi_{B_0}\|_{L^{\varphi}
(\mathcal{X})}}\\\nonumber
&\lesssim&\left[\frac{\mu(B_0)}{\mu(B_{1,i})}\right]^{p_1}
\frac{\varphi(B_{1,i},t_0)}{\mu(B)}\frac{\|\chi_{B_{1,i}}\|_{L^{\varphi}
(\mathcal{X})}}{\|\chi_{B_0}\|_{L^{\varphi}
(\mathcal{X})}}\\\nonumber
&\lesssim&\frac{\varphi(B_{1,i},t_0)}{\mu(B_{1,i})}\left[\frac{\mu(B_0)}
{\mu(B_{1,i})}\right]^{p_1-1-\frac{q-1}{q}}\lesssim
\frac{\beta\varphi(B_{1,i},t_0)}{\mu(B_{1,i})}.
\end{eqnarray}
From this and \eqref{deAj}, it follows that, to prove \eqref{eAj},
we only need to show that, for almost every $x\in B_{1,i}$,
\begin{eqnarray*}
|A_{t_{B_{1,i}}}(f-A_{t_B}f)(x)|\lesssim
\frac{\beta\varphi(B_{1,i},t_0)}{\mu(B_{1,i})}.
\end{eqnarray*}

Following the estimates same as those used in \eqref{eAuj1},
we still divide $|A_{t_{B_{1,i}}}(f-A_{t_B}f)(x)|$
into two parts $\mathrm{I}$ and $\mathrm{II}$ same as in \eqref{eAuj1}.

By $M>n+np_1-\alpha$ and some estimates similar to those used
in \eqref{eAuj2} and \eqref{eAuj3}, we know
that $\mathrm{I}\lesssim\frac{\beta\varphi(B_{1,i},t_0)}{\mu(B_{1,i})}$,
the details being omitted.

By $M>n+\frac{2np_1}{p}+N-\frac{n(q-1)}{q}-\alpha>np_1$
and the arguments similar to those used in the estimate $\mathrm{II}$
of Theorem \ref{t-uj}, we obtain
$\mathrm{II}\lesssim\frac{\beta\varphi(B_{1,i},t_0)}{\mu(B_{1,i})}$,
the details being omitted again.
Thus, \eqref{eAj} holds true.

For $\lambda\in(0,\infty)$, we write
\begin{eqnarray}\label{BBj1}
&&\left\{x\in B:\ \frac{|f(x)-A_{t_{B}}f(x)|}
{\varphi(x,t_0)}>\lambda\right\}\\\nonumber
&&\hs\subset\bigcup_i\left(\left\{x\in B_{1,i}:\
\frac{|f(x)-A_{t_{B_{1,i}}}f(x)|}{\varphi(x,t_0)}>
\frac{\lambda}{2}\right\}\right.\\\nonumber
&&\hs\hs\left.\bigcup\left\{x\in B_{1,i}:\
\frac{|A_{t_{B_{1,i}}}f(x)-A_{t_B}f(x)|}{\varphi(x,t_0)}>\frac{\lambda}{2}
\right\}\right).
\end{eqnarray}

By Lemma \ref{l-j} and \eqref{eAj}, we know that there
exists a positive constant $c_6$ such that,
for any $\lambda\in(0,\infty)$,
\begin{eqnarray}\label{BBj2}
&&\varphi\left(\left\{x\in B:\ \frac{|A_{t_{B_{1,i}}}
f(x)-A_{t_{B}}f(x)|}{\varphi(x,t_0)}
>\frac{\lambda}{2}\right\},t_0\right)\\\nonumber
&&\hs\leq\varphi\left(\left\{x\in B:\ M_{\varphi}
\left(\frac{1}{\varphi(x,t_0)}\chi_{B_{1,i}}(x)\right)
>\frac{\lambda\mu(B_{1,i})}{2\widetilde{c}_5\beta
\varphi(B_{1,i},t_0)}\right\},t_0\right)\\\nonumber
&&\hs\leq c_6\left(\frac{\beta}{\lambda}\right)^{p_1'}\varphi(B_{1,i},t_0).
\end{eqnarray}

We also define $\sigma_{f,B}(\lambda):=\varphi(\{x\in B:\
\frac{|f(x)-A_{t_{B}}f(x)|}{\varphi(x,t_0)}>\lambda\},t_0)$ and
$$F_f(\lambda):=\sup_{B\subset B_0}
\frac{\sigma_{f,B}(\lambda)}{\varphi(B,t_0)}.$$
Then, from \eqref{BBj1}, \eqref{BBj2} and \eqref{eBjn1}, it follows that,
for any $\lambda\in(0,\infty)$
\begin{eqnarray*}
\sigma_{f,B}(\lambda)&\leq& \left[F_f
\left(\frac{\lambda}{2}\right)+c_6\left(\frac{\beta}
{\lambda}\right)^{p_1'}\right]
\sum_i\varphi(B_{1,i},t_0)\\
&\leq&\frac{c_3}{\beta}\left[F_f
\left(\frac{\lambda}{2}\right)+c_6\left(\frac{\beta}
{\lambda}\right)^{p_1'}\right]\varphi(B,t_0),
\end{eqnarray*}
which, together with \eqref{btjn}, implies that there exists
 a positive constant $c_7$ such that, for any $\lambda\in(0,\infty)$,
$F_f(\lambda)\leq 4^{p_1'}F_f(\frac{\lambda}{2})+c_7\lambda^{-p_1'}.$

By induction, we see that, for all $m\in\mathbb{Z}_+$
and $\lambda\in(c_7,2c_7]$,
$$F_f(2^m\lambda)\leq (2c_7)^{p_1'}(2^m\lambda)^{-p_1'}$$
which implies that \eqref{sjn2} holds true and hence
completes the proof of Theorem \ref{t-j}.
\end{proof}

\begin{remark}\label{r-j}
(i) Theorem \ref{t-j}(i) completely covers \cite[Theorem 3.1]{dy}
and \cite[Theorem 3.1]{tang}, respectively, by taking $\fai$,
respectively, as in \eqref{fat} and \eqref{fatb}.
Moreover, Theorem \ref{t-j}(i) completely
covers \cite[Theorem 3.6]{bd} by taking $\varphi$
as in \eqref{faw1t}.

(ii) Theorem \ref{t-j}(ii) is new even when $\fai$
is as in \eqref{fawt}.

(iii) Let $\varphi$ be as in Definition \ref{d-grf} and satisfy \eqref{afnj}.
If $\varphi\in \mathbb{A}_1(\mathcal{X})$, then
\eqref{jn1} can be deduced from \eqref{uj1}.

Indeed, assume that $f\in\mathrm{BMO}^{\varphi}_A(\mathcal{X})$
and \eqref{uj1} holds true. Then, by $\varphi\in \mathbb{A}_1(\mathcal{X})$,
we see that, for any ball $B\subset\mathcal{X}$ and $t\in(0,\infty)$,
$\frac{\varphi(B,t)}{\mu(B)}\lesssim \einf_{x\in B}\fai(x,t)$,
which, together with \eqref{uj1}, implies that
\begin{eqnarray*}
&&\mu\left(\left\{x\in B:\ \frac{|f(x)-A_{t_B}f(x)|}
{\varphi(x,\|\chi_{B}\|_{L^{\varphi}(\mathcal{X})}^{-1})}
>\lambda\right\}\right)\\
&&\hs\lesssim\mu\left(\left\{x\in B:\ |f(x)-A_{t_B}f(x)|>\frac{\lambda\varphi
(B,\|\chi_{B}\|_{L^{\varphi}(\mathcal{X})}^{-1})}
{\mu(B)}\right\}\right)\\
&&\hs\lesssim \mu(B)\exp\left\{-\frac{c_2\lambda }
{\|\chi_{B}\|_{L^{\varphi}(\mathcal{X})}
\|f\|_{\mathrm{BMO}^{\varphi}_A(\mathcal{X})}}\right\},
\end{eqnarray*}
where $c_2$ is as in \eqref{uj1}.
This, together with Lemma \ref{l-uAp}(ii), implies that \eqref{jn1} holds true.

However, when $\varphi \notin\mathbb{A}_1(\mathcal{X})$,
the relationship between Theorems \ref{t-uj} and \ref{t-j}
 is still unclear.
\end{remark}

Now, we introduce the space
$\widetilde{\mathrm{BMO}}^{\varphi,\,\widetilde{p}}_A(\mathcal{X})$
 for $\widetilde{p}\in[1,\infty)$.

\begin{definition}\label{d-pj}
Let $\mathcal{X}$ be a space of homogeneous type,
$\fai$ as in Definition \ref{d-grf},
$\mathcal{M}(\mathcal{X})$ as in \eqref{Max0b} and
$\{A_t\}_{t>0}$ a generalized approximation to the identity satisfying
\eqref{ubfA} and \eqref{pdcu}. Let $\widetilde{p}\in[1,\infty)$.
The \emph{space}
$\widetilde{\mathrm{BMO}}^{\varphi,\,\widetilde{p}}_A(\mathcal{X})$
is defined as the set of all
$f\in\mathcal{M}(\mathcal{X})$ such that
\begin{eqnarray*}
\|f\|_{\widetilde{\mathrm{BMO}}^{\varphi,\,\widetilde{p}}_A(\mathcal{X})}&:=&
\sup_{B\subset\mathcal{X}}\frac{1}
{\|\chi_{B}\|_{L^{\varphi}(\mathcal{X})}}\left\{
\int_B\left|\frac{f(x)-A_{t_B}f(x)}
{\varphi(x,\|\chi_{B}\|_{L^{\varphi}(\mathcal{X})}^{-1})}\right|
^{\widetilde{p}}
\varphi\left(x,\|\chi_{B}\|_{L^{\varphi}(\mathcal{X})}^{-1}\right)
\,d\mu(x)\right\}
^{\frac{1}{\widetilde{p}}}\\
&<& \infty,
\end{eqnarray*}
where the supremum is taken over all balls $B$ in $\mathcal{X}$,
$t_B:=r_B^m$ and $r_B$ denotes the radius of the ball $B$.
\end{definition}

By Theorem \ref{t-j}, we obtain the following conclusion.
\begin{theorem}\label{t-pj}
Let $\mathcal{X}$ be a space of homogeneous type with degree
$(\alpha_0,n_0,N_0)$,
where $\alpha_0$, $n_0$ and $N_0$ are as in
\eqref{d-ar0}, \eqref{d-n0} and \eqref{d-N0},
respectively, $\fai$ as in Definition \ref{d-grf} and
$\{A_t\}_{t>0}$ satisfy Assumption $\mathrm{A}$.

{\rm (i)} Assume that $\varphi\in\mathbb{A}_1(\mathcal{X})$.
For different $\widetilde{p}\in[1,\infty)$, the spaces
$\widetilde{\mathrm{BMO}}^{\varphi,\,\widetilde{p}}_A(\mathcal{X})$
coincide with equivalent norms.

{\rm (ii)} Assume that $p(\varphi)\leq1+\frac{1}{[r(\varphi)]'}$,
where $p(\varphi)$ and
$r(\varphi)$ are, respectively, as in \eqref{crAp} and \eqref{crRD}.
For different $\widetilde{p}\in[1,[p(\varphi)]')$, the spaces
$\widetilde{\mathrm{BMO}}^{\varphi,\,\widetilde{p}}_A(\mathcal{X})$
coincide with equivalent norms.
\end{theorem}

\begin{proof}
For any $f\in\widetilde{\mathrm{BMO}}^{\varphi,\,\widetilde{p}}_A(\mathcal{X})$
with $\widetilde{p}\in[1,\infty)$, by H\"{o}lder's inequality
and some estimates similar to those used in the proof of
Theorem \ref{t-puj}, we see that
$f\in \mathrm{BMO}^{\varphi}_A(\mathcal{X})$
and $\|f\|_{\mathrm{BMO}^{\varphi}_A(\mathcal{X})}\leq
\|f\|_{\widetilde{\mathrm{BMO}}^{\varphi,\widetilde{p}}_A(\mathcal{X})}$,
the details being omitted.

Conversely, let $f\in \mathrm{BMO}^{\varphi}_A(\mathcal{X})$.
Now we prove (i) and (ii) separately.

(i) In this case, $\varphi\in \mathbb{A}_1(\mathcal{X})$.
From this, Theorem \ref{t-j}(i) and some estimates similar
to those used in the proof of
Theorem \ref{t-puj}, it follows that $f\in\widetilde{\mathrm{BMO}}
^{\varphi,\,\widetilde{p}}_A(\mathcal{X})$ and $\|f\|_{\widetilde{\mathrm{BMO}}
^{\varphi,\,\widetilde{p}}_A(\mathcal{X})}\lesssim
\|f\|_{\mathrm{BMO}^{\varphi}_A(\mathcal{X})}$, the details being omitted.
This finishes the proof of (i).

(ii)  In this case, let $\wz{p}\in(1,[p(\fai)]')$.
By the definition of
$p(\fai)$, we see that $\fai\in\aa_{\wz{p}'}(\cx)$.
Moreover, let $p_1\in(p(\fai),\wz{p}')$.
Then, $\varphi\in \mathbb{A}_{p_1}(\mathcal{X})$ and
$\wz{p}<p_1'<[p(\varphi)]'$.
By Theorem \ref{t-j}(ii), we find that, for any
$B\subset\mathcal{X}$ and $t\in(0,\infty)$,
\begin{eqnarray*}
&&\int_B\left|\frac{f(x)-A_{t_B}f(x)}
{\varphi(x,\|\chi_{B}\|_{L^{\varphi}(\mathcal{X})}^{-1})}
\right|^{\wz{p}}\varphi\left(x,\|\chi_{B}\|_{L^{\varphi}(\mathcal{X})}^{-1}\right)
\,d\mu(x)\\
&&\hs=\wz{p}\int_0^{\infty}\lambda^{\wz{p}-1}\varphi\left(\left\{x\in B:\
\frac{|f(x)-A_{t_B}f(x)|}
{\varphi(x,\|\chi_{B}\|_{L^{\varphi}(\mathcal{X})}^{-1})}
>\lambda\right\},\|\chi_{B}\|_{L^{\varphi}(\mathcal{X})}^{-1}\right)\,d\lambda\\
&&\hs\lesssim \wz{p}\int_0^{\infty}\lambda^{\wz{p}-1}\min\left\{1,\
C\left(\frac{\lambda }{\|\chi_{B}\|_{L^{\varphi}(\mathcal{X})}
\|f\|_{\mathrm{BMO}^{\varphi}_A(\mathcal{X})}}\right)
^{-p_1'}\right\}\,d\lambda\\
&&\hs\lesssim\int_0^{C^{1/p_1'}
\|\chi_{B}\|_{L^{\varphi}(\mathcal{X})}
\|f\|_{\mathrm{BMO}^{\varphi}_A(\mathcal{X})}}
\lambda^{\wz{p}-1}\,d\lambda\\
&&\hs\hs+\int_{C^{1/p_1'}
\|\chi_{B}\|_{L^{\varphi}(\mathcal{X})}
\|f\|_{\mathrm{BMO}^{\varphi}_A(\mathcal{X})}
}^{\infty}\lambda^{\wz{p}-p_1'-1}
\|\chi_{B}\|_{L^{\varphi}(\mathcal{X})}^{p_1'}\,d\lambda
\lesssim\lf[\|\chi_{B}\|_{L^{\varphi}(\mathcal{X})}
\|f\|_{\mathrm{BMO}^{\varphi}_A(\mathcal{X})}\r]^{\wz{p}},
\end{eqnarray*}
which implies that
$f\in\widetilde{\mathrm{BMO}}^{\varphi,\,\wz{p}}_A(\mathcal{X})$ and
$\|f\|_{\widetilde{\mathrm{BMO}}^{\varphi,\,\wz{p}}_A(\mathcal{X})}
\lesssim \|f\|_{\mathrm{BMO}^{\varphi}_A(\mathcal{X})}$, and hence
completes the proof of Theorem \ref{t-pj}.
\end{proof}

\begin{remark}\label{r5.3}
(i) It is easy to see that, when $\fai(x,t):=t^s$, with $s\in(0,1]$,
for all $x\in\cx$ and $t\in[0,\fz)$, the spaces
$\mathrm{BMO}^{\varphi,\,\wz{p}}(\cx)$ and $\widetilde{\mathrm{BMO}
}^{\varphi,\,\wz{p}}(\cx)$, with $\wz{p}\in[1,\fz)$, are same. Thus,
Theorem \ref{t-pj} also completely covers \cite[Theorem 3.4]{dy}
and \cite[Theorem 3.4]{tang}
by taking $\varphi$, respectively, as in
\eqref{fat} and \eqref{fatb}.

(ii) Theorem \ref{t-pj} is new even when $\fai$ is as in \eqref{fawt}.
\end{remark}

\section{Equivalence of $\mathrm{BMO}^{\varphi}(\mathbb{R}^n)$
and $\mathrm{BMO}^{\varphi}_{\sqrt{\Delta}}(\mathbb{R}^n)$}\label{s-eBd}

\hskip\parindent
In this section, by the $\fai$-Carleson measure characterization of
$\mathrm{BMO}^{\fai}(\rn)$ established in \cite{hyy},
the boundedness of
the classical Littlewood-Paley $g$-function on $L^2(\rn)$,
and the John-Nirenberg inequality obtained in Theorems \ref{t-uj}
or \ref{t-j}, we show
that, when $\mathcal{X}:=\mathbb{R}^n$, the new
Musielak-Orlicz BMO-type space
$\mathrm{BMO}^{\varphi}_{\sqrt{\Delta}}(\mathbb{R}^n)$,
associated with $\{A_t\}_{t>0}$ given
by the Poisson kernel, is equivalent to the
Musielak-Orlicz BMO-type space
$\mathrm{BMO}^{\varphi}(\mathbb{R}^n)$ introduced by
Ky \cite{ky}. We begin with some notions.

Let $\Delta:=-\sum_{i=1}^{n}\partial^2_{x_i}$ denote
the \emph{Laplace operator} on $\mathbb{R}^n$ and
$\{e^{-t\sqrt{\Delta}}\}_{t>0}$ be the corresponding
\emph{Poisson semigroup}. Observe that, if $f$ is
a function belonging to the set
$$\mathcal{M}_{\sqrt{\Delta}}(\mathcal{X})
:=\{f\in L_{\mathop\mathrm{loc}}^1
(\mathbb{R}^n):\ f(x)(1+|x|^{n+1})^{-1}\in L^1(\mathbb{R}^n)\},$$
then we can define the generalized approximation
to the identity $\{A_t\}_{t>0}$ by the Poisson integral as follows:
for all $t\in(0,\infty)$ and $x\in\rn$,
$$A_tf(x):=P_tf(x):=\int_{\mathbb{R}^n}p_t(x-y)f(y)\,dy,$$
where, for all $t\in(0,\infty)$ and $x\in\rn$,
$$p_t(x):=\frac{c_nt}{(t^2+|x|^2)^{(n+1)/2}}\ \ \hbox{and}\ \
c_n:=\frac{\Gamma[(n+1)/2]}{\pi^{(n+1)/2}}.$$

In this case, we have the following assumption for $\varphi$.
\begin{proof}[\rm\bf Assumption B]
Let $\varphi$ be as in Definition \ref{d-grf} and satisfy
\begin{eqnarray*}
\frac{2np(\varphi)}{i(\varphi)}-
\frac{n[r(\varphi)-1]}{r(\varphi)}<n+1,
\end{eqnarray*}
where $p(\fai)$, $i(\varphi)$ and $r(\varphi)$ are, respectively, as in
\eqref{crAp}, \eqref{cult} and \eqref{crRD}.
\end{proof}
\begin{remark}
From Assumption {\rm B}, we deduce that the Poisson
kernel satisfies the Assumption $\mathrm{A}$
and that $p(\varphi)\leq1+\frac{1}{[r(\varphi)]'}$
in Theorem \ref{t-j}(ii) automatically holds true.
Moreover, $m$ in \eqref{ubfA} and
\eqref{pdcu} is equal to $1$.
Then, it is easy to see that $\{P_t\}_{t>0}$ satisfies
\eqref{ubfA}, \eqref{pdcu}
and \eqref{atbB}.
\end{remark}

For any $f\in L^p(\mathbb{R}^n)$
with $p\in[1,\infty]$, $P_tf=e^{-t\sqrt{\Delta}}f$.
We \emph{use $\mathrm{BMO}^{\varphi}_{\sqrt{\Delta}}(\mathbb{R}^n)$
to denote $\mathrm{BMO}_{A}^{\varphi}(\mathbb{R}^n)$
space associated with the Poisson semigroup $\{e^{-t\sqrt{\Delta}}\}_{t>0}$.}

Now, we recall the $\varphi$-Carleson measure
introduced in \cite{hyy}.
\begin{definition}\label{d-fcm}
Let $\varphi$ be as in Definition \ref{d-grf}. A measure $d\mu$
on $\mathbb{R}^{n+1}_+$ is called a
$\varphi$-\emph{Carleson\ measure}, if
$$\|d\mu\|_{\varphi}:=\sup_{B\subset\mathbb{R}^n}\frac{|B|^{1/2}}
{\|\chi_{B}\|_{L^{\varphi}(\mathbb{R}^n)}}\left\{\int_{\widehat{B}}|
d\mu(x,t)|
\right\}^{1/2}<\infty,$$
where the supremum is taken over all balls $B\subset\mathbb{R}^n$,
$$\widehat{B}:=\{(x,t)\in\mathbb{R}^{n+1}_+:\ x\in B,\ t\in(0,r_B)\}$$
and $r_B$ denotes the radius of the ball $B$.
\end{definition}

Let $\phi\in\cs(\rn)$ be a radial real-valued function satisfying that
$$\int_{\rn}\phi(x)x^{\gz}\,dx=0$$
for all $\gz\in\zz_+^n$ with $|\gz|\le s$,
where $s\in\zz_+$, $s\ge\lfz n[p(\fai)/i(\fai)-1]\rfz$ and,
for all $\xi\in\rn\setminus\{0\}$,
\begin{eqnarray}\label{cofa}
\int_0^\fz|\widehat{\phi}(t\xi)|^2\frac{dt}{t}=1,
\end{eqnarray}
where $\widehat{\phi}$ denotes the
\emph{Fourier transform} of
$\phi$.

The following $\fai$-Carleson measure characterization of $\mathrm{BMO}^\fai(\rn)$
is just \cite[Theorem 5.3]{hyy}.

\begin{lemma}\label{l-fBr}
Let $\fai$ be as in Definitions \ref{d-grf} and
$\phi$ as above.

\ {\rm (i)} Assume that $b\in\bbmo^\fai(\rn)$ and
$q(\fai)[r(\fai)]'\in(1,2)$. Then
$$d\mu(x,t):=|\phi_t\ast b(x)|^2
\frac{dx\,dt}{t}$$ is a $\fai$-Carleson measure on $\rr^{n+1}_+$; moreover,
there exists a positive constant $C$, independent of $b$, such that
$\|d\mu\|_{\fai}\le C\|b\|_{\bbmo^\fai(\rn)}$.

{\rm(ii)} Assume that $np(\fai)<(n+1)i(\fai)$.
Let $b\in L^2_{\loc}(\rn)$ and, for all $(x,t)\in\rr^{n+1}_+$,
$d\mu(x,t):=|\phi_t\ast b(x)|\frac{dxdt}{t}$
be a $\fai$-Carleson measure
on $\rr^{n+1}_+$. Then $b\in\bbmo^\fai(\rn)$ and, moreover, there exists a
positive constant $C$, independent of $b$, such that
$\|b\|_{\bbmo^\fai(\rn)}\le C\|d\mu\|_{\fai}$.
\end{lemma}

\begin{remark}
Actually, if we replace $p(\fai)[r(\fai)]'\in(1,2)$
by \eqref{afnj}, then Lemma \ref{l-fBr}(i) still holds true.
To this end, we need to use a
John-Nirenberg inequality on $\bbmo^\fai(\rn)$ similar to
Theorem \ref{t-uj}, which further leads to Lemma \ref{l-fBr}(i) with
the assumption $p(\fai)[r(\fai)]'\in(1,2)$ replaced by \eqref{afnj}.
In this way, \eqref{afnj} is needed. We omit the details.
\end{remark}

The main result of this section is as follows. Recall that the space
$\mathcal{K}_{\sqrt{\Delta}}(\rn)$ is defined as in 
Remark \ref{r-BfA}(i) with $A$ and $\cx$ replaced, respectively, by 
$\{e^{-t\sqrt{\Delta}}\}_{t>0}$ and $\rn$.

\begin{theorem}\label{t-fBa}
Let $\varphi$ satisfy Assumption {\rm B}.

\ {\rm (i)}
If $\varphi$ additionally satisfies \eqref{afnj} or
$p(\fai)[r(\fai)]'\in(1,2)$, where $p(\fai)$ and $r(\fai)$ are, respectively,
as in \eqref{crAp} and \eqref{crRD},
then, for any $f\in \mathrm{BMO}_{\sqrt{\Delta}}^{\varphi}(\mathbb{R}^n)$,
$$\left|t\frac{\partial}{\partial t}
P_t(\mathcal{I}-P_t)f(x)\right|^2\,\frac{dxdt}{t}$$
is a $\varphi$-$Carleson\ measure$ on $\mathbb{R}^{n+1}_+$.

{\rm(ii)}
Assume further that
$np(\fai)<(n+1)i(\fai)$, where $i(\fai)$ is
as in \eqref{cult}.
Then the spaces $\mathrm{BMO}^{\varphi}(\mathbb{R}^n)$ and
$\mathrm{BMO}^{\varphi}_{\sqrt{\Delta}}(\mathbb{R}^n)$
$(modulo\ \mathcal{K}_{\sqrt{\Delta}}(\rn))$ coincide with equivalent norms.
\end{theorem}

\begin{proof}
To prove (i), by Definition \ref{d-fcm},
it suffices to prove that, for any ball $B\subset\mathbb{R}^n$,
\begin{eqnarray}\label{afB}
\frac{|B|}
{\|\chi_{B}\|_{L^{\varphi}(\mathbb{R}^n)}^2}
\int_{\widehat{B}}\left|t\frac{\partial}{\partial t}
P_t(\mathcal{I}-P_t)f(x)\right|^2\,\frac{dxdt}{t}\lesssim
\|f\|^2_{\mathrm{BMO}^{\varphi}_{\sqrt{\Delta}}(\mathbb{R}^n)}.
\end{eqnarray}
Let $B:=B(x_B,r_B)$. Recall that, in this case, $m=1$,
where $m$ is as in \eqref{ubfA} and \eqref{pdcu}. Thus, $t_B=r_B$.
Notice that
$\mathcal{I}-P_t=(\mathcal{I}-P_t)(\mathcal{I}-P_{t_B})+
(\mathcal{I}-P_t)P_{t_B}.$
Then, we have
\begin{eqnarray}\label{dPfB}
t\frac{\partial}{\partial t}P_t(\mathcal{I}-P_t)&=&
t\frac{\partial}{\partial t}P_t(\mathcal{I}-P_t)
(\mathcal{I}-P_{t_B})+t\frac
{\partial}{\partial t}P_t(\mathcal{I}-P_t)P_{t_B}\\\nonumber
&=&t\frac{\partial}{\partial t}P_t(\mathcal{I}-P_t)
(\mathcal{I}-P_{t_B})+t\frac
{\partial}{\partial t}P_{(2t+t_B)/2}(P_{t_B/2}-P_{(2t+t_B)/2}).
\end{eqnarray}

Once we prove that
\begin{eqnarray}\label{afB1}
\frac{|B|}
{\|\chi_{B}\|_{L^{\varphi}(\mathbb{R}^n)}^2}\int_{\widehat{B}}
\left|t\frac{\partial}{\partial t}P_t(\mathcal{I}-P_t)
(\mathcal{I}-P_{t_B})f(x)\right|^2\,\frac{dxdt}{t}
\lesssim\|f\|^2_{\mathrm{BMO}_{\sqrt{\Delta}}
^{\varphi}(\mathbb{R}^n)}
\end{eqnarray}
and
\begin{eqnarray}\label{afB2}
\ \ \ \ \frac{|B|}
{\|\chi_{B}\|_{L^{\varphi}(\mathbb{R}^n)}^2}
\int_{\widehat{B}}\left|t\frac{\partial}{\partial t}
P_{(2t+t_B)/2}(P_{t_B/2}-P_{(2t+t_B)/2})f(x)
\right|^2\,\frac{dxdt}{t}
\lesssim\|f\|^2_{\mathrm{BMO}_{\sqrt{\Delta}}
^{\varphi}(\mathbb{R}^n)},
\end{eqnarray}
by \eqref{dPfB}, we then conclude that \eqref{afB} holds true,
which is obvious.

To show \eqref{afB1} and \eqref{afB2},
we borrow some ideas from \cite[pp. 1393-1395]{dy}.

It is easy to see that, for any $f\in\cup_{p\geq1}L^p(\rn)$,
$$g(f,x):=\left(\int_{0}^{\infty}
\left|t\frac{\partial}{\partial t}P_t(\mathcal{I}-P_t)f(x)\right|
^2\,\frac{dt}{t}\right)^{\frac{1}{2}}$$
is the Littlewood-Paley $g$-function. By \cite[p.\,80]{jo},
we know that $g$ is bounded on $L^2(\mathbb{R}^n)$.

Let $b_1:=(\mathcal{I}-P_{t_B})f\chi_{2B}$ and
$b_2:=(\mathcal{I}-P_{t_B})f\chi_{(2B)^\complement}$.
If $\varphi$ satisfies \eqref{afnj}, by Proposition
\ref{p-AtB}, Theorem \ref{t-puj}, \eqref{Bbkai},
\eqref{bBkai} and the boundedness of $g$ on $L^2(\mathbb{R}^n)$, we find that
\begin{eqnarray}\label{b1fB}
&&\frac{|B|}
{\|\chi_{B}\|_{L^{\varphi}(\mathbb{R}^n)}^2}\int_{\widehat{B}}
\left|t\frac{\partial}{\partial t}P_t(\mathcal{I}-P_t)b_1(x)\right|
^2\,\frac{dxdt}{t}\\\nonumber
&&\hs\leq\frac{|B|}
{\|\chi_{B}\|_{L^{\varphi}(\mathbb{R}^n)}^2}
\int_{\mathbb{R}^{n+1}_+}
\left|t\frac{\partial}{\partial t}P_t(\mathcal{I}-P_t)b_1(x)\right|
^2\,\frac{dxdt}{t}\\\nonumber
&&\hs\lesssim\frac{|B|}
{\|\chi_{B}\|_{L^{\varphi}(\mathbb{R}^n)}^2}
\|b_1\|^2_{L^2(\mathbb{R}^n)}\thicksim\frac{|B|}
{\|\chi_{B}\|_{L^{\varphi}(\mathbb{R}^n)}^2}\int_{2B}|(
\mathcal{I}-P_{t_{B}})f(x)|^2\,dx\\\nonumber
&&\hs\lesssim\frac{|B|}
{\|\chi_{B}\|_{L^{\varphi}(\mathbb{R}^n)}^2}
\left\{\int_{2B}|(
\mathcal{I}-P_{t_{2B}})f(x)|^2\,dx\right.\\\nonumber
&&\left.\hs\hs+\int_{2B}\left[\frac
{\|\chi_{B(x,r_B)}\|
_{L^{\varphi}(\mathbb{R}^n)}}{|B|}
\|f\|_{\mathrm{BMO}^{\varphi}_{\sqrt{\Delta}}
(\mathbb{R}^n)}\right]^2\,dx\right\}
\lesssim\|f\|^2_{\mathrm{BMO}_{\sqrt{\Delta}}^{\varphi}
(\mathbb{R}^n)}.
\end{eqnarray}

If $p(\fai)[r(\fai)]'\in(1,2)$,
we see that $[p(\fai)]'>2$
and $r(\fai)>\frac{2([p(\fai)]'-1)}{[p(\fai)]'-2}$.
From this and the definition of $r(\fai)$,
we deduce that there exists $\widetilde{p}\in(2,[p(\fai)]')$ such that
$r(\fai)>\frac{2(\widetilde{p}-1)}{\widetilde{p}-2}$
and hence $\fai\in\rh_{2(\widetilde{p}-1)/(\widetilde{p}-2)}(\rn)$.
By this, H\"older's inequality and Theorem \ref{t-pj}, we conclude that
\begin{eqnarray}\label{b2fB}
\ \ \ \ \ &&\int_{2B}|(\mathcal{I}-P_{t_{2B}})f(x)|^2\,dx\\\nonumber
&&\hs=\int_{2B}\left|\frac{(\mathcal{I}-P_{t_{2B}})f(x)}
{\varphi(x,\|\chi_{2B}\|_{L^{\varphi}(\rn}^{-1})}\right|^2
\left[\varphi\left(x,
\|\chi_{2B}\|_{L^{\varphi}(\rn)}^{-1}\right)\right]^{\frac{2}{\widetilde{p}}}
\left[\varphi\left(x,
\|\chi_{2B}\|_{L^{\varphi}(\rn)}
^{-1}\right)\right]^{\frac{2(\widetilde{p}-1)}{\widetilde{p}}}\,dx\\\nonumber
&&\hs=\left[\int_{2B}\left|\frac{(\mathcal{I}-P_{t_{2B}})f(x)}
{\varphi(x,\|\chi_{2B}\|_{L^{\varphi}(\rn)}^{-1})}\right|
^{\widetilde{p}}
\varphi\left(x,
\|\chi_{2B}\|_{L^{\varphi}(\rn)}^{-1}\right)\,dx\right]
^{\frac{2}{\widetilde{p}}}\\\nonumber
&&\hs\hs\times
\left\{\int_{2B}\left[\varphi\left(x,
\|\chi_{2B}\|_{L^{\varphi}(\rn)}
^{-1}\right)\right]^{\frac{2(\widetilde{p}-1)}
{\widetilde{p}-2}}\,dx\right\}^{\frac{\widetilde{p}-2}
{\widetilde{p}}}\\\nonumber
&&\hs\lesssim\frac{\|\chi_{2B}\|_{L^{\varphi}(\rn)}^2}{|2B|}
\|f\|^2_{\mathrm{BMO}_{\sqrt{\Delta}}^{\varphi}(\rn)}
\thicksim \frac{\|\chi_{B}\|_{L^{\varphi}(\rn)}^2}{|B|}
\|f\|^2_{\mathrm{BMO}_{\sqrt{\Delta}}^{\varphi}(\rn)}.
\end{eqnarray}

On the other hand,
by Assumption {\rm B},
we see that there exist
$p_1\in[p(\fai),\fz)$, $p\in(0,i(\fai)]$ and $q\in(1,r(\fai)]$ such that
$\fai\in\aa_{p_1}(\cx)$, $\fai$ is of uniformly lower type $p$,
$\fai\in\mathbb{RH}_{q}(\cx)$
and $\frac{2np_1}{p}-\frac{n(q-1)}{q}<n+1$. From this, we further deduce
that $n<\frac{np_1}{p}<n+1$,
which implies that $\frac{np_1}{p}-n-1<0$ and
$\frac{2np_1}{p}-\frac{n(q-1)}{q}-n-1<0$.

For any $x\in B$ and $y\in(2B)^\complement$, it holds true that
$|x-y|^{n+1}> r_B^{n+1}=t_B^{n+1}$.
This, together with Proposition \ref{p-AtB},
\eqref{Bbkai}, \eqref{bBkai}, $\frac{np_1}{p}-n-1<0$,
$\frac{2np_1}{p}-\frac{n(q-1)}{q}-n-1<0$ and some estimates similar to those used
in the proof of Proposition \ref{p-isB}, implies that
\begin{eqnarray*}
&&\left|\frac{|B|}{\|\chi_{B}\|_{L^{\varphi}(\mathbb{R}^n)}}
t\frac{\partial}{\partial t}P_t(\mathcal{I}-P_t)b_2(x)\right|\\
&&\hs\lesssim\frac{|B|}{\|\chi_{B}\|_{L^{\varphi}(\mathbb{R}^n)}}
\int_{\mathbb{R}^n\setminus2B}\frac{t}
{(t^2+|x-y|^{n+1})}
|(\mathcal{I}-P_{t_B})f(y)|\,dy\\
&&\hs\lesssim\left(\frac{t}{t_B}\right)
\left(\frac{t_B|B|}{\|\chi_{B}\|_{L^{\varphi}(\mathbb{R}^n)}}\right)
\int_{\mathbb{R}^n\setminus2B}\frac{|(\mathcal{I}-P_{t_B})f(y)|}
{|x-y|^{n+1}}\,dy\\
&&\hs\lesssim\left(\frac{t}{t_B}\right)
\left(\frac{t_B^{n+1}}{\|\chi_{B}\|_{L^{\varphi}
(\mathbb{R}^n)}}\right)\sum_{k=1}^{\infty}
\int_{2^{k+1}B\setminus{2^kB}}\left[\frac{|f(y)-
P_{t_{2^{k+1}B}}f(y)|}{|x-y|^{n+1}}\right.\\
&&\hs\hs\left.+\frac{|P_{t_{2^{k+1}B}}f(y)-
P_{t_B}f(y)|}{|x-y|^{n+1}}\right]\,dy\\
&&\hs\lesssim\left(\frac{t}{t_B}\right)\sum_{k=1}^{\infty}\left[
\frac{\|\chi_{2^{k+1}B}\|_{L^{\varphi}(\mathbb{R}^n)}}
{2^{k(n+1)}\|\chi_{B}\|_{L^{\varphi}(\mathbb{R}^n)}}\right.\\
&&\hs\hs\left.+2^{k(\frac{np_1}{p}-n
)}\int_{2^{k+1}B\setminus{2^kB}}
\frac{\|\chi_{B(y,t_B)}
\|_{L^{\varphi}(\mathbb{R}^n)}}{|B(y,t_B)|}\,dy\right]
\|f\|_{\mathrm{BMO}^{\varphi}_{\sqrt{\Delta}}
(\mathbb{R}^n)}\\
&&\hs\lesssim\left(\frac{t}{t_B}\right)\sum_{k=1}^{\infty}
\left[2^{k(\frac{np_1}{p}-n-1)}+2^{k(
\frac{2np_1}{p}-\frac{n(q-1)}{q}-n-1)}
\right]\|f\|_{\mathrm{BMO}^{\varphi}
_{\sqrt{\Delta}}(\mathbb{R}^n)}\\
&&\hs\lesssim\left(\frac{t}{t_B}\right)
\|f\|_{\mathrm{BMO}^{\varphi}_{\sqrt{\Delta}}(\mathbb{R}^n)}.
\end{eqnarray*}
From this, we further deduce that
\begin{eqnarray*}
&&\frac{|B|}
{\|\chi_{B}\|_{L^{\varphi}(\mathbb{R}^n)}^2}\int_{\widehat{B}}
\left|t\frac{\partial}{\partial t}P_t(\mathcal{I}-P_t)b_2(x)\right|
^2\,\frac{dxdt}{t}\\
&&\hs=\frac{1}{|B|}\int_{\widehat{B}}
\left|\frac{|B|}{\|\chi_{B}\|_{L^{\varphi}}}t
\frac{\partial}{\partial t}P_t(\mathcal{I}-P_t)b_2(x)\right|^2
\,\frac{dxdt}{t}\\
&&\hs\lesssim\frac{\|f\|^2_
{\mathrm{BMO}_{\sqrt{\Delta}}^{\varphi}(\mathbb{R}^n)}}
{t_B^2|B|}\int_{\widehat{B}}t\,dxdt\lesssim\|f\|^2_{\mathrm{BMO}_
{\sqrt{\Delta}}^{\varphi}(\mathbb{R}^n)}.
\end{eqnarray*}
By this and \eqref{b1fB} or \eqref{b2fB}, we conclude that \eqref{afB1} holds true.

Now we prove \eqref{afB2}. For any $t\in(0,r_B)$, it follows,
from Proposition \ref{p-AtB}, that, for any $x\in \mathbb{R}^n$,
\begin{eqnarray}\label{PPfB}
|P_{t_B/2}f(x)-P_{(2t+t_B)/2}f(x)|\lesssim
\frac{\|\chi_{B(x,t_B/2)}\|_{L^{\varphi}(\mathbb{R}^n)}}
{|B(x,t_B/2)|}\|f\|_{\mathrm{BMO}_{\sqrt{\Delta}}^{\varphi}
(\mathbb{R}^n)}.
\end{eqnarray}
Moreover, it is easy to show that the kernel $k_{t,t_B}$ of
the operator $T_{t,t_B}=t\frac{\partial}{\partial t}P_{(2t+t_B)/2}$
satisfies that, for all $x,y\in \mathbb{R}^n$,
$$|k_{t,t_B}(x,y)|\lesssim\left(\frac{t}{t_B}\right)\frac{t_B}
{t_B^{n+1}+|x-y|^{n+1}}.$$
By this and \eqref{PPfB}, similar to the proof of \eqref{afB1}, we see that
\begin{eqnarray*}
&&\frac{|B|}
{\|\chi_{B}\|_{L^{\varphi}(\mathbb{R}^n)}^2}
\int_{\widehat{B}}\left|t\frac{\partial}{\partial t}P_{(2t+t_B)/2}
(P_{t_B/2}-P_{(2t+t_B)/2})f(x)
\right|^2\,\frac{dxdt}{t}\\
&&\hs\lesssim\frac{|B|}
{t_B^2\|\chi_{B}\|_{L^{\varphi}(\mathbb{R}^n)}^2}
\int_{\widehat{B}}t
\left|\int_{\mathbb{R}^n}\frac{t_B}{t_B^{n+1}+|x-y|^{n+1}}
|(P_{t_B/2}-P_{(2t+t_B)/2})f(y)|\,dy\right|^2\,dxdt\\
&&\hs\lesssim\frac{1}{t_B^2|B|}
\int_{\widehat{B}}t
\left|\frac{|B|}{\|\chi_{B}\|_{L^{\varphi}(\mathbb{R}^n)}}
\int_{\mathbb{R}^n}\frac{t_B}
{t_B^{n+1}+|x-y|^{n+1}}
|(P_{t_B/2}-P_{(2t+t_B)/2})f(y)|\,dy\right|^2\,dxdt\\
&&\hs\lesssim\frac{\|f\|^2_{\mathrm{BMO}_{\sqrt{\Delta}}^{\varphi}}}
{t_B^2|B|}\int_{\widehat{B}}t\,dxdt
\lesssim\|f\|^2_{\mathrm{BMO}_{\sqrt{\Delta}}^{\varphi}
(\mathbb{R}^n)},
\end{eqnarray*}
which implies that \eqref{afB2} holds true
and hence completes the proof of (i).

Now we prove (ii). From Proposition \ref{p-BAB}, it follows that
$\mathrm{BMO}^{\varphi}(\mathbb{R}^n)\subset\mathrm{BMO}
_{\sqrt{\Delta}}^{\varphi}
(\mathbb{R}^n)$. Thus, we only need to prove the reverse inclusion.
For any $f\in\mathrm{BMO}_{\sqrt{\Delta}}^{\varphi}
(\mathbb{R}^n)$, by (i), we see that $\left|t\frac{\partial}
{\partial t}P_t(\mathcal{I}-P_t)f(x)\right|^2\,\frac{dxdt}{t}$
is a $\varphi$-Carleson\ measure on $\mathbb{R}^{n+1}_+$. Let
$\phi_t(\cdot):=t^{-n}\phi(\cdot/t)$ be the kernel of the operator
$t\frac{\partial}{\partial t}P_t(\mathcal{I}-P_t)$. Then, $\phi$ satisfies that
$$\int_{\mathbb{R}^n}\phi(x)\,dx=0\ \ \ \ \hbox{and}\ \ \ \
\int_0^\fz|\widehat{\phi}(t\xi)|^2\frac{dt}{t}=\frac{\pi}{3}$$
for all $\xi\in\rn\setminus\{0\}$.
Notice that the condition $\int_0^\fz|\widehat{\phi}(t\xi)|^2\frac{dt}{t}=1$
for all $\xi\in\rn\setminus\{0\}$ in \eqref{cofa}
can be replaced by $\int_0^\fz|\widehat{\phi}(t\xi)|^2\frac{dt}{t}=c$
for all $\xi\in\rn\setminus\{0\}$, where $c$ is some positive constant
 (see \cite{hyy}). Then, from Lemma \ref{l-fBr}(ii) and the conclusion of (i), we deduce that
$f\in\mathrm{BMO}^{\varphi}(\mathbb{R}^n)$ and
$\|f\|_{\mathrm{BMO}^\fai(\rn)}\ls\|f\|_{\mathrm{BMO}^\fai_{\sqrt{\Delta}}(\rn)}$,
which completes the proof of Theorem \ref{t-fBa}.
\end{proof}

\begin{remark}\label{r6.1}
Theorem \ref{t-fBa} completely covers \cite[Theorem 2.14]{dy}
by taking $\varphi$ as in \eqref{fat} with $\cx$ replaced by $\rn$. Moreover,
Theorem \ref{t-fBa} is also new even when $\varphi$ is as in \eqref{fatb}
with $\cx$ replaced by $\rn$.
\end{remark}

Next, we consider the space $\mathrm{BMO}^{\varphi}_A(\mathbb{R}^n)$
associated with the generalized approximation to the identity
$\{A_t\}_{t>0}$ acting on the function $f$ which satisfies that
$$\int_{\mathbb{R}^n}|f(x)|e^{-|x|}dx<\infty.$$
Here, the generalized approximation to
the identity $\{A_t\}_{t>0}$ is given by
setting, for all $t\in(0,\infty)$ and $x\in\rn$,
$$A_tf(x):=H_tf(x):=\int_{\mathbb{R}^n}h_t(x-y)f(y)dy$$
and, for all $x$, $y\in\rn$,
$$h_t(x-y):=\frac{1}{(4\pi t)^{n/2}}e^{-|x-y|^2/4t}.$$
It is easy to see that $\{H_t\}_{t>0}$ satisfies \eqref{ubfA} and
\eqref{pdcu}. Notice that, for any $f\in L^p(\rn)$ with $p\in[1,\infty]$,
$H_tf=e^{-t\Delta}f$. Then, as above, we use
$\mathrm{BMO}^{\varphi}_{\Delta}(\mathbb{R}^n)$ to denote
$\mathrm{BMO}^{\varphi}_A(\mathbb{R}^n)$ associated with the heat
semigroup $\{e^{-t\Delta}\}_{t>0}$. By making some minor
modifications on the proof of Theorem \ref{t-fBa},
we obtain the following theorem, the details being omitted.
Recall that the space $\mathcal{K}_{\Delta}(\rn)$ is defined as in
Remark \ref{r-BfA}(i) with $A$ and $\cx$ replaced, respectively, by
$\{e^{-t\Delta}\}_{t>0}$ and $\rn$.

\begin{theorem}\label{t-fGr}
Let $\varphi$ be as in Definition \ref{d-grf} and satisfy
$np(\fai)<(n+1)i(\fai)$, where $p(\fai)$ and $i(\fai)$ are, respectively,
as in \eqref{crAp} and \eqref{cult}.
If $\varphi$ additionally satisfies \eqref{afnj} or $p(\fai)[r(\fai)]'\in(1,2)$,
where $r(\varphi)$ is as in \eqref{crRD},
then the spaces $\mathrm{BMO}^{\varphi}(\mathbb{R}^n)$ and
$\mathrm{BMO}^{\varphi}_{\Delta}(\mathbb{R}^n)$ $(modulo\
\mathcal{K}_{\Delta}(\rn))$ coincide with equivalent norms.
\end{theorem}

By Theorems \ref{t-eBm}, \ref{t-ewB}, \ref{t-fBa}
and \ref{t-fGr}, we finally obtain the following conclusion.

\begin{corollary}\label{c-far}
Let $\varphi$ be as in Definition \ref{d-grf} and satisfy
$np(\fai)<(n+1)i(\fai)$ and
$$\frac{2np(\varphi)}{i(\varphi)}-
\frac{n[r(\varphi)-1]}{r(\varphi)}<n+1,$$
where $n,\,p(\fai)$, $i(\varphi),\,N$, $r(\varphi)$ and $\alpha$ are,
respectively, as in \eqref{sthp},
\eqref{crAp}, \eqref{cult}, \eqref{eBdc}, \eqref{crRD} and \eqref{RDin}.
If $\varphi$ additionally satisfies \eqref{afnj} or $p(\fai)[r(\fai)]'\in(1,2)$,
then the spaces
$$\mathrm{BMO}^{\varphi}(\mathbb{R}^n),\
\mathrm{BMO}^{\varphi}_{\sqrt{\Delta}}(\mathbb{R}^n),\
\mathrm{BMO}^{\varphi}_{\sqrt{\Delta},\,\max}(\mathbb{R}^n),\
\widetilde{\mathrm{BMO}}^{\varphi}_{\sqrt{\Delta}}(\mathbb{R}^n),\
\mathrm{BMO}^{\varphi}_{\Delta}(\mathbb{R}^n),$$
$\mathrm{BMO}^{\varphi}_{\Delta,\,\max}(\mathbb{R}^n)$
and $\widetilde{\mathrm{BMO}}^{\varphi}_{\Delta}(\mathbb{R}^n)$
coincide with equivalent norms.
\end{corollary}

\begin{remark}\label{r6.2}
We point out that Theorem \ref{t-fGr} and Corollary \ref{c-far}
completely cover
\cite[Theorem 2.15 and Corollary 2.16]{dy} by taking
$\fai$, respectively, as in \eqref{fat} and \eqref{fatb} with $\cx$ replaced by $\rn$.
\end{remark}

\bigskip

Shaoxiong Hou, Dachun Yang (Corresponding author) and Sibei Yang

\medskip

School of Mathematical Sciences, Beijing Normal University,
Laboratory of Mathematics and Complex Systems, Ministry of
Education, Beijing 100875, People's Republic of China

\smallskip

{\it E-mails}: \texttt{houshaoxiong@mail.bnu.edu.cn} (S. Hou)

\hspace{1.55cm}\texttt{dcyang@bnu.edu.cn} (D. Yang)

\hspace{1.55cm}\texttt{yangsibei@mail.bnu.edu.cn} (S. Yang)

\end{document}